\documentclass{article}
\usepackage{amssymb,amsfonts,amsmath,amsthm,amsopn,amstext,amscd,latexsym,xy,color,stmaryrd,epic}
\usepackage{hyperref}
\usepackage{tikz,multicol,graphicx,longtable}
\usepackage{mathrsfs}
\usetikzlibrary{arrows,positioning,decorations.pathmorphing,
  decorations.markings
 }
\tikzset{inner sep=0pt, 
  root/.style={circle,draw,minimum size=7pt,thick}, 
  fatroot/.style={circle,draw,minimum size=10pt,thick}, 
  short root/.style={circle,fill,minimum size=7pt}, 
  doublearrow/.style={postaction={decorate}, 
  decoration={markings,mark=at position .7
  with {\arrow{angle 60}}},double distance=3pt,thick}
} 

\theoremstyle{plain}
\input xy
\xyoption{all}
\newtheorem{theorem}{Theorem}[section]
\newtheorem{contheorem}[theorem]{Theorem-Construction}
\newtheorem{lemma}[theorem]{Lemma}
\newtheorem{proposition}[theorem]{Proposition}
\newtheorem{corollary}[theorem]{Corollary}

\theoremstyle{remark}
\newtheorem{remark}[theorem]{Remark}
\numberwithin{equation}{section}
\numberwithin{paragraph}{section}

\newcommand{\dquot}{{\,\!\sslash\!\,}}

\DeclareMathOperator{\Hom}{Hom}
\DeclareMathOperator{\vol}{vol}

\DeclareMathOperator{\Ad}{Ad}

\DeclareMathOperator{\Gal}{Gal}
\DeclareMathOperator{\Aut}{Aut}

\DeclareMathOperator{\Sel}{Sel}
\DeclareMathOperator{\ad}{ad}
\DeclareMathOperator{\ord}{ord}

\DeclareMathOperator{\Lie}{Lie}

\DeclareMathOperator{\Pic}{Pic}

\DeclareMathOperator{\diag}{diag}

\DeclareMathOperator{\End}{End}

\DeclareMathOperator{\Frob}{Frob}

\DeclareMathOperator{\disc}{disc}
\DeclareMathOperator{\Frac}{Frac}
\DeclareMathOperator{\Spec}{Spec}

\newcommand{\cC}{{\mathcal C}}

\newcommand{\cJ}{{\mathcal J}}

\newcommand{\cO}{{\mathcal O}}

\newcommand{\cW}{{\mathcal W}}

\newcommand{\fra}{{\mathfrak a}}

\newcommand{\frg}{{\mathfrak g}}
\newcommand{\frh}{{\mathfrak h}}

\newcommand{\frl}{{\mathfrak l}}
\newcommand{\ffrm}{{\mathfrak m}}

\newcommand{\frs}{{\mathfrak s}}

\newcommand{\frt}{{\mathfrak t}}

\newcommand{\frz}{{\mathfrak z}}
\newcommand{\bbA}{{\mathbb A}}

\newcommand{\bbC}{{\mathbb C}}

\newcommand{\bbF}{{\mathbb F}}
\newcommand{\bbG}{{\mathbb G}}

\newcommand{\bbP}{{\mathbb P}}
\newcommand{\bbQ}{{\mathbb Q}}
\newcommand{\bbR}{{\mathbb R}}

\newcommand{\bbZ}{{\mathbb Z}}

\newcommand{\PGL}{\mathrm{PGL}}
\newcommand{\SO}{\mathrm{SO}}
\newcommand{\SL}{\mathrm{SL}}

\newcommand{\al}{\alpha}
\newcommand{\be}{\beta}

\newcommand{\lam}{\lambda}

\newcommand{\frakh}{\mathfrak h}

\newcommand{\sH}{\mathscr{H}}

\newcommand{\la}{\langle}
\newcommand{\ra}{\rangle}

\newcommand{\wt}{\widetilde}

\DeclareMathOperator{\Span}{span}

\newcommand{\htvar}{a}

\newcommand{\Cen}{A} %max torus formally known as C
\newcommand{\cen}{\mathfrak a} %lie alg of max torus formally known as C

\newcommand{\irr}{{\text{irr}}}
\DeclareMathOperator{\Ht}{ht} %The height

\newcommand{\Curve}{\mathscr{C}} %the curve C. The open curve will be \Curve^0
\newcommand{\OpenSurface}{\mathscr{X}} %the affine elliptic surface X
\newcommand{\Surface}{\mathscr{Y}} %the compactified elliptic surface Y
\newcommand{\Family}{\mathscr{F}} %the set of all odd genus 2 curves
\newcommand{\Equations}{\mathscr{E}} %the set of equations for odd genus 2 curves
\newcommand{\Point}{\mathscr{P}} %The point at infinity on an odd genus 2 curve
\newcommand{\Jacobian}{\mathscr{J}} %The Jacobian of a genus 2 curve
\newcommand{\Lattice}{\mathscr{W}} %The Mordell--Weil lattice of an elliptic surface

\newcommand{\OO}{\mathscr{O}} %Identity section of elliptic surface
\newcommand{\FF}{\mathscr{F}} %Fibre at infinity of elliptic surface

\newcommand{\Kostant}{\sigma} %The Kostant section (viewed as a map \Kostant : B \to V)
\newcommand{\Torsor}{\mathscr{T}} %The torsor for a torus over a subvariety of the regular semisimple locus in V or \frh

\newcommand{\intH}{\underline{H}} %H as a group over Z
\newcommand{\intG}{\underline{G}} %G as a group over Z
\newcommand{\intT}{\underline{T}} %T as a group over Z
\newcommand{\intV}{\underline{V}} %V as a module over Z

\newcommand{\intfrh}{{\underline{\mathfrak{h}}}} %Lie H as a Lie algebra over Z
\newcommand{\intB}{\underline{B}} %integral model for the geometric quotient B
\newcommand{\intCen}{\underline{A}} 

\newcommand{\extJ}{{\mathscr{G}}} % Heisenberg group of Jacobian

\newcommand{\prin}{\mathscr{M}} % Principal polarization on Jacobian
\newcommand{\tprin}{\mathscr{L}} % 3 times the principal polarization
\author{Beth Romano \and Jack A. Thorne}
\AtEndDocument{\bigskip{\footnotesize 
\textsc{Beth Romano}~~ \texttt{blr24@dpmms.cam.ac.uk}\\\textsc{Department of Pure Mathematics and Mathematical Statistics, Wilberforce Road, Cambridge, CB3 0WB, UK}
\par
\textsc{Jack A. Thorne} ~~ \texttt{thorne@dpmms.cam.ac.uk}\\  \textsc{Department of Pure Mathematics and Mathematical Statistics, Wilberforce Road, Cambridge, CB3 0WB, UK}}}

\title{$E_8$ and the average size of the 3-Selmer group of the Jacobian of a pointed genus-2 curve }

\begin{document}

\maketitle

\begin{abstract}
We prove that the average size of the 3-Selmer group of a genus-2 curve with a marked Weierstrass point is 4. We accomplish this by studying rational and integral orbits in the representation associated to a stably $\bbZ / 3 \bbZ$-graded simple Lie algebra of type $E_8$. We give new techniques to construct integral orbits, inspired by the proof of the fundamental lemma and by the twisted vertex operator realisation of affine Kac--Moody algebras. 
\end{abstract}

\tableofcontents
\section{Introduction}

\subsection{Statement of results}

In this paper we prove new theorems about the arithmetic statistics of odd genus-2 curves. If $f(x) = x^5 + c_{12} x^3 + c_{18} x^2 + c_{24} x + c_{30} \in \bbQ[x]$ is a polynomial of non-zero discriminant, then the smooth projective completion of the affine curve
\[ \Curve_f^0 : y^2 = f(x) \]
is a genus-2 curve with a marked Weierstrass point (the unique point at infinity). Conversely, any pair $(\Curve, \Point)$, where $\Curve$ is a (smooth, projective, connected) curve of genus 2 and $\Point \in \Curve(\bbQ)$ is a marked Weierstrass point, arises from a polynomial $f(x)$ satisfying the following conditions:
\begin{enumerate}
\item The coefficients of $f(x)$ are integers and the discriminant of $f(x)$ is non-zero.
\item No polynomial of the form $n^{-10} f(n^{2} x)$ has integer coefficients, where $n \geq 2$ is an integer.
\end{enumerate}
We write $\Equations$ for the set of all polynomials $f(x) = x^5 + c_{12} x^3 + c_{18} x^2 + c_{24} x + c_{30} \in \bbZ[x]$ of non-zero discriminant, and $\Equations_\text{min} \subset \Equations$ for the subset satisfying condition 2.\ above. For $f(x) \in \Equations$, we write $\Curve_f$ for the corresponding pointed genus-2 curve and $\Jacobian_f$ for the Jacobian of $\Curve_f$, a principally polarized abelian  surface over $\bbQ$. We define the height $\Ht(f)$ of a polynomial $f(x) \in \Equations$ by the formula
\[ \Ht(f) = \sup_i | c_i(f) |^{120 / i}. \]
Note that for any $\htvar > 0$, the set $\{ f \in \Equations | \Ht(f) < \htvar \}$ is finite. We can now state our first main theorem.
\begin{theorem}[Theorem \ref{thm_main_theorem_first_version}]\label{thm_intro_bounded_selmer}
The average size of the 3-Selmer group $\Sel_3(\Jacobian_f)$ for $f \in \Equations_\text{min}$ is 4. More precisely, we have
\[ \lim_{\htvar \to \infty} \frac{ \sum_{f \in \Equations_\text{min}, \Ht(f) < \htvar} | \Sel_3(\Jacobian_f) | }{ | \{ f \in \Equations_\text{min} \mid \Ht(f) < \htvar \} |} = 4. \]
\end{theorem}
(A similar result can be proved for subsets of $\Equations_\text{min}$ defined by congruence conditions. See Remark \ref{rmk_generalized_version_of_main_theorem}.) 

Theorem \ref{thm_intro_bounded_selmer} has the following consequence for rational points, which follows from work of Poonen and Stoll \cite{Poo14}: 
\begin{theorem}[Theorem \ref{thm_poonen_stoll}]\label{thm_intro_poonen_stoll}
Let $\Family$ be the set of pairs $(\Curve, \Point)$ where $\Curve$ is a (smooth, projective, connected) curve of genus 2 and $\Point \in \Curve(\bbQ)$ is a marked Weierstrass point. Then
a positive proportion of curves $(\Curve, \Point) \in \Family$ satisfy $\Curve(\bbQ) = \{ \Point \}$. More precisely, we have
\[ \liminf_{\htvar \to \infty} \frac{ | \{ f \in \Equations_\text{min} \mid \Ht(f) < \htvar, | \Curve_f(\bbQ) | = 1 \} |}{| \{ f \in \Equations_\text{min} \mid \Ht(f) < \htvar \} |} > 0. \]
\end{theorem}

\subsection{Context and method of proof}

In the paper \cite{Bha13}, Bhargava and Gross calculated the average size of the the 2-Selmer group of the Jacobian of an odd hyperelliptic curve of fixed genus $g \geq 2$ using a connection with the arithmetic invariant theory of a graded Lie algebra; more precisely, the $\bbZ / 2 \bbZ$-graded Lie algebra arising from the element $-1$ of the automorphism group of a type-$A_{2g}$ root lattice. Their proof required studying the orbits of the representation coming from this grading: that of the group $\SO_{2g+1}$ on the space of traceless, self-adjoint $(2g + 1) \times (2g + 1)$ matrices. 

In this paper, we exploit the stable $\bbZ / 3 \bbZ$-grading of a Lie algebra of type $E_8$ in order to study the 3-Selmer groups of odd genus-2 curves. Note that we are firmly in the territory of exceptional groups! In particular, there seems to be no hope of generalizing anything in this paper to study e.g.\ the 3-Selmer groups of hyperelliptic curves of higher genus. Nevertheless, we expect the methods developed in this paper to have applications elsewhere, for reasons we will soon explain.

Let $H$ be a split reductive group over $\bbQ$ of type $E_8$ with split maximal torus $T$, and let $\check{\rho} : \bbG_m \to T$ be the sum of the fundamental coweights with respect to some choice root basis. The restriction $\theta$ of $\check{\rho}$ to $\mu_3$ determines a stable $\bbZ / 3 \bbZ$-grading 
\[ \frh = \oplus_{i \in \bbZ / 3 \bbZ} \frh(i) \]
of the Lie algebra $\frh := \Lie H$, and hence a coregular representation of $G = H^\theta$ on $V = \frh(1)$ (see e.g.\ \cite{Ree12} -- the word `stable' refers to the presence of stable $G$-orbits in $V$, i.e. orbits that are closed and have finite stabilizers).

One can identify $G$ with $\SL_9 / \mu_3$ and $V$ with the 3rd exterior power of the standard representation of $\SL_9$. The relation between this representation and 3-descent on odd genus-2 curves has been studied previously by Rains and Sam \cite{Rai17}. We do not use their work. Instead, we follow a different approach which we find more suited to studying integrality problems (of which more in a moment).

Using results of Vinberg, one can identify the geometric quotient $B = V \dquot G = \Spec \bbQ[V]^G$ with the spectrum of the polynomial algebra $\bbQ[c_{12}, c_{18}, c_{24}, c_{30}]$ in 4 indeterminates (thus $c_{12}, \dots, c_{30}$ are algebraically independent $G$-invariant polynomials on $V$). We can therefore think of $B$ as parameterizing polynomials $f(x) = x^5 + c_{12} x^3 + c_{20} x^2 + c_{24} x + c_{30}$. We write $V_f \subset V$ for the $G$-invariant closed subscheme given by the fibre of the quotient map $\pi : V \to B$ above a point $f$ of the base. 

The first step in the proof of Theorem \ref{thm_intro_bounded_selmer} is to construct for any field $k / \bbQ$ and any $f \in B(k)$ of non-zero discriminant an injection
\begin{equation}\label{eqn_intro_eta}\eta_f: \Jacobian_f(k) / 3 \Jacobian_f(k) \to G(k) \backslash V_f(k), 
\end{equation}
where $\Jacobian_f$ is the Jacobian of the curve given by the equation $y^2 = f(x)$. In fact, we go further than this, giving a construction that works over any $\bbQ$-algebra $R$ (and for any $f \in B(R)$ with discriminant that is a unit in $R$). If $R$ is a ring over which every locally free module is free, then we obtain an injection
\begin{equation}\label{eqn_intro_eta_R} \eta_f : \Jacobian_f(R) / 3 \Jacobian_f(R) \to G(R) \backslash V_f(R), 
\end{equation}
recovering the previous map (\ref{eqn_intro_eta}) in the case that $R = k$ is a field. 

This construction is based on changing our point of view from $G$-orbits in $V$ to isomorphism classes of triples $(H', \theta', \gamma')$, where $H'$ is a reductive group of type $E_8$, $\theta'$ is a stable $\bbZ / 3 \bbZ$-grading, and $\gamma' \in \frh'(1)$. We give a construction that begins with a Heisenberg group (such as the $\mu_3$-extension of $\Jacobian_f[3]$ arising from the Mumford theta group of thrice the canonical principal polarization of $\Jacobian_f$) and a representation $W$ of this Heisenberg group, and returns a Lie algebra $\frh'$ of type $E_8$ with a stable $\bbZ / 3 \bbZ$-grading $\theta'$, together with a representation of $\frg' = (\frh')^{\theta'}$ on the same space $W$. The existence of this construction, which seems to be related to twisted vertex operator realisations of affine Kac--Moody algebras \cite{Lep85}, still seems remarkable to us! The general version of this construction will be described in a future work of the first author \cite{Rom18a}.

The next step in the proof of Theorem \ref{thm_intro_bounded_selmer} is to introduce integral structures. All of the objects $H$, $\theta$, $G$, $V$ can be defined naturally over $\bbZ$, and we may choose the polynomials $c_{12}, \dots, c_{30}$ so that they lie in $\bbZ[V]^G$. If $p$ is a prime and $f(x) = x^5 + c_{12} x^3 + c_{18} x^2 + c_{24} x + c_{30} \in \bbZ_p[x]$ is a polynomial of non-zero discriminant, then our constructions so far yield a map $\Jacobian_f(\bbQ_p) / 3 \Jacobian_f(\bbQ_p) \to G(\bbQ_p) \backslash V_f(\bbQ_p)$. However, it is essential to be able to show that each $G(\bbQ_p)$-orbit in $V_f(\bbQ_p)$ that is in the image of this map admits an integral representative, i.e.\ intersects $V_f(\bbZ_p)$ non-trivially. This has been a sticking point for some time. In our earlier papers \cite{Tho15, Rom18}, our failure to construct integral representatives in full generality meant we could provide upper bounds only for the average sizes of the Selmer sets, and not the full Selmer groups, of the families of curves studied there. 

In this paper we introduce a new general technique to construct integral orbit representatives. We describe it briefly here. If $f(x) \in \bbZ_p[x]$ is a polynomial of non-zero discriminant, we choose a lifting to $\widetilde{f}(x) \in \bbZ_p[u][x]$ with favourable properties. In particular, the discriminant of $\widetilde{f}(x)$ should be non-zero in $\bbF_p[u]$ and {square-free} in $\bbQ_p[u]$. The construction giving rise to the map (\ref{eqn_intro_eta_R}) determines a triple $(H', \theta', \gamma')$ over the complement in $\Spec \bbZ_p[u]$ of the locus where the discriminant of $\widetilde{f}$ vanishes.

Using an explicit construction of integral representatives in the square-free discriminant case, we extend the objects in this triple to the complement in $\Spec \bbZ_p[u]$ of finitely many closed points. Finally, we use the fact that a reductive group on the punctured spectrum of a 2-dimensional regular local ring extends uniquely to the whole spectrum (see \cite[Theorem 6.13]{Col79}) to extend the objects in our triple further to the whole of $\Spec \bbZ_p[u]$. Specialising to $u = 0$, we find the desired integral representative. 

This argument is inspired by the proof of the fundamental lemma for Lie algebras \cite{Ngo10}. The problem of constructing integral representatives can be viewed as the problem of showing that a graded analogue of an affine Springer fibre is non-empty. From this point of view, attempting to deform the problem to a case where it can be solved directly is a natural strategy. Although we develop this technique here just in the case of the stable $\bbZ / 3 \bbZ$-grading of $E_8$ and its relation to odd genus-2 curves, it is completely general. We hope to return to this in a future work. 

Once integral representatives have been constructed, we can reduce the problem of studying the average size of the 3-Selmer groups of the curves $\Curve_f$ to the problem of studying the number of orbits of $G(\bbZ)$ in $V(\bbZ)$ of bounded height (with congruence conditions and local weights imposed). In the final step in the proof of Theorem \ref{thm_intro_bounded_selmer}, we use Bhargava's techniques and their interpretation in the framework of graded Lie algebras (as in e.g.\ \cite{Bha13}, \cite{Tho15}) to carry out this orbit count and finally prove Theorem \ref{thm_intro_bounded_selmer}.

\begin{remark}
	In the second author's thesis \cite{Tho13}, simple curve singularities and their deformations played an important role. The same is true here. The family of affine curves given by the equation $y^2 = x^5 + c_{12} x^3 + c_{18} x + c_{24} x + c_{30}$ is a versal deformation of a type-$A_4$ singularity. Here, we think of this family instead as being embedded in the family of affine surfaces
	\[ y^2 = z^3 + x^5 + c_{12} x^3 + c_{18} x + c_{24} x + c_{30}. \]
	This family is a versal deformation of the $E_8$ surface singularity $y^2 = z^3 + x^5$ and carries an action of $\mu_3$ by the formula $\zeta \cdot (x, y, z) = (x, y, \zeta^{-1} z)$. This fact plays an important role in \S \ref{sec_proof_of_prop_on_invariants}. 
\end{remark}

\subsection{Organization of this paper}

We now describe the organization of this paper. In \S \ref{sec_root_lattice_of_type_E8} we review relevant properties of the $E_8$ root lattice and its associated Weyl group. In \S \ref{sec_a_functor}, fundamental for the construction of orbits, we give our `Heisenberg group to graded Lie algebra' construction. In \S \ref{sec_stable_grading_of_E8}, we describe the invariant theory of our graded Lie algebra, and use the construction of \S \ref{sec_a_functor} to parameterize and construct orbits. An important role is played by two special transverse slices to nilpotent elements, namely the Kostant section and the subregular Slodowy slice: we use the first of these to parametrize the set of stable orbits, and the second to normalize our generators for the ring of $G$-invariant polynomials on $V$.

In \S \ref{sec_integral_orbits} we give our construction of integral orbit representatives. We treat the local case using the ideas described above, and then deduce the existence of integral orbit representatives for Selmer elements in the global case as a consequence. In \S \ref{sec_counting_points}, we give the point-counting results we need in order to prove Theorem \ref{thm_intro_bounded_selmer}. Finally, in \S \ref{sec_main_theorem} we combine all of these ingredients to prove our main theorems. 

\subsection{Acknowledgments}

Both authors received support from EPSRC First Grant EP/N007204/1. This project has received funding
from the European Research Council (ERC) under the European Union's Horizon
2020 research and innovation programme (grant agreement No. 714405).

\subsection{Notation}

If $H$ is a group scheme, then we will use a gothic letter $\frh = \Lie H$ for its Lie algebra. If $\theta : \mu_n \to \Aut(\frh)$ is homomorphism, then we write $\frh = \oplus_{i \in \bbZ / n \bbZ} \frh(\theta, i)$ for the corresponding grading; thus $\frh(\theta, i)$ is the isotypic subspace in $\frh$ corresponding to the character $\zeta \mapsto \zeta^i$ of $\mu_n$. (If the map $\theta$ is understood, we sometimes write $\frh(i)$ for $\frh(\theta, i)$.) In particular $\frh(\theta, 0) = \frh^\theta$. Similarly, if $\theta: \mu_n \to H$ is a homomorphism, we write $H^\theta$ for centralizer of $\theta$ in $H$.
%the fixed-point subgroup of $H$ for the action of $\mu_n$ via $\theta$ composed with conjugation.}

If $G$ is a group scheme over a base $S$, $X$ an $S$-scheme on which $G$ acts, $T$ an $S$-scheme, and $x \in X(T)$, then we write $Z_G(x)$ for the scheme-theoretic stabilizer of $x$, which is a $T$-scheme.  

By a Lie algebra over $S$, we mean a coherent sheaf of $\cO_S$-modules $\frg$ together with an alternating bilinear form $[\cdot, \cdot] : \frg \times \frg \to \frg$ that satisfies the Jacobi identity. Similarly, if $\frg$ is a Lie algebra that is equipped with a Lie algebra homomorphism $\frg \to \End_{\cO_S}(W)$, for some locally free sheaf $W$ of $\cO_S$-modules, and $x \in W \otimes_{\cO_S} \cO_T$, then we define $\frz_\frg(x)$ to be the Lie centralizer of $x$, which is a Lie algebra over $T$. 

If $G$ is reductive and $A \subset G$ is a maximal torus, we write $X^*(A) = \Hom(A, \bbG_m)$ for its character group, $\Phi(G, A) \subset X^*(A)$ for its set of roots, and $X_*(A)$ for its cocharacter group. We write $N_G(A)$ for the normalizer of $A$ and $W(G, A) = N_G(A) / A$ for the Weyl group of $A$ in $G$.

If $G$ is a smooth group scheme over a scheme $S$, then we write $H^1(S, G)$ for the set of isomorphism classes of $G$-torsors over $S$, which we think of as a non-abelian \'etale cohomology set. If $S = \Spec R$ is affine then we will write $H^1(R, G)$ for the same object. 

If $G$ is a smooth linear algebraic group over a field $k$ which acts on an integral affine variety $X$, and $G^0$ is reductive, then we write $X \dquot G = \Spec k[X]^G$, which is again an integral affine variety. 

\section{The $E_8$ root lattice}\label{sec_root_lattice_of_type_E8}

Throughout this paper, we will constantly make use of the properties of a certain conjugacy class of automorphisms of the $E_8$ root lattice. We therefore record some of these properties here. For us, an $E_8$ root lattice $\Lambda$ is a finite free $\bbZ$-module, equipped with a symmetric bilinear pairing $(\cdot, \cdot) : \Lambda \times \Lambda \to \bbZ$ with the following property: the set $\Phi := \{ \alpha \in \Lambda \mid (\alpha, \alpha) = 2\}$ forms a root system in $\Lambda_\bbR$ of Dynkin type $E_8$ (elements of $\Phi$ are called roots). Any two $E_8$ root lattices are isomorphic. 

Let $\Lambda$ be an $E_8$ root lattice with $\Phi \subset \Lambda$ its set of roots.
Note that because $E_8$ is simply laced, if $\al, \be \in \Phi$, then $\al + \be \in \Phi$ if and only if $(\al, \be) = -1$. 
We will make frequent use of this fact throughout the proofs in \S \ref{sec_a_functor} and \S \ref{sec_counting_points}. Given $\gamma \in \Lambda$, we define $\check\gamma$ to be the element of the dual lattice $\Lambda^\vee = \Hom(\Lambda, \bbZ)$ given by $\check\gamma(\mu) = (\gamma, \mu)$ for all $\mu \in \Lambda$. We note that the map $\Lambda \to \Lambda^\vee$ defined by $\gamma \mapsto \check\gamma$ is an isomorphism of lattices. (If $\al \in \Phi$ is a root, then $\check\al$ is called a coroot.)

We write $\Aut(\Lambda)$ for the group of automorphisms of $\Lambda$ that preserve the pairing $(\cdot, \cdot)$. Since $E_8$ has no diagram automorphisms, $\Aut(\Lambda)$ equals the Weyl group of $E_8$.
%$ W(\Lambda)$, which is generated by the reflections in the root hyperplanes $\alpha^\perp$ (for $\alpha \in \Phi)$. 
We recall that an element $w \in \Aut(\Lambda)$ is said to be elliptic if $\Lambda^w = 0$. 
\begin{lemma}\label{lem_weyl_group}
	The automorphism group $\Aut(\Lambda)$ contains a unique conjugacy class of elliptic elements of order $3$. Let $w$ be such an element, and let $\Lambda_w = \Lambda / (w-1)\Lambda$ be the group of $w$-coinvariants in $\Lambda$. Then:
	\begin{enumerate}
		\item There is an isomorphism $\Lambda_w \cong \bbF_3^4$.
		\item 
		Any choice of orbit representatives for the action of $\la w \ra$ on $\Phi$ gives a complete set of coset representatives for the non-zero elements of $\Lambda_w$. 
\item The centralizer of $w$ in $\Aut(\Lambda)$ acts transitively on $\Lambda_w - \{ 0 \}$.
	\end{enumerate}
\end{lemma}
\begin{proof}
	See \cite[Table 1]{Ree11} and \cite[Lemma 4.4]{Ree11}.
\end{proof}
We also note for later use that if $w \in \Aut(\Lambda)$ is elliptic of order 3, then $w^2\gamma + w\gamma + \gamma = 0$ for all $\gamma \in \Lambda$. 
In \cite{Ree11}, for any elliptic element $w \in \Aut(\Lambda)$, Reeder defines a symplectic pairing on $\Lambda_w$ that is invariant under the action of the centralizer of $w$ in $\Aut(\Lambda)$. We now describe a slight variant of this pairing.

Let $S$ be a $\bbZ[1/3]$-scheme.  We now let $\Lambda$ be an \'etale sheaf of $E_8$ root lattices on $S$. By this we mean that $\Lambda$ is a locally constant \'etale sheaf of finite free $\bbZ$-modules that is equipped with a pairing $\Lambda \times \Lambda \to \bbZ$ making each stalk $\Lambda_{\overline{s}}$ above a geometric point $\overline{s} \to S$ into an $E_8$ root lattice. Then $\Aut(\Lambda)$ is a finite \'etale $S$-group. 

In this setting we define an elliptic $\mu_3$-action on $\Lambda$ to be a homomorphism $\theta : \mu_3 \to \Aut(\Lambda)$ such that for any geometric point $\overline{s} \to S$ and any primitive 3rd root of unity $\zeta \in \mu_3(\overline{s})$, $\theta(\zeta) \in \Aut(\Lambda_{\overline{s}})$ is an elliptic element of order 3. 

If $\theta$ is an elliptic $\mu_3$-action on $\Lambda$, then we write $\Lambda_\theta$ for the sheaf of $\theta$-coinvariants; by Lemma \ref{lem_weyl_group}, it is a locally constant \'etale sheaf of $\bbF_3$-vector spaces of rank 4. We define a pairing $\langle \cdot, \cdot \rangle : \Lambda_\theta \times \Lambda_\theta \to \mu_3$ by the formula
\begin{equation}\label{eqn_reeders_pairing}
\langle \lambda, \mu \rangle = \zeta^{((1 - \theta(\zeta)) \lambda, \mu)},
\end{equation} 
for any primitive 3rd root of unity $\zeta$. (Despite appearances, the pairing does not depend on a choice of root of unity.)
\begin{lemma}\label{lem_reeder_pairing_is_symplectic}
	The pairing (\ref{eqn_reeders_pairing}) is symplectic and non-degenerate, and it induces an isomorphism  $\Lambda_\theta \cong \Hom(\Lambda_\theta, \mu_3)$.
\end{lemma}
\begin{proof}
	This can be checked on geometric points, in which case it reduces to \cite[Lemma 2.2, Lemma 2.3]{Ree11}.
\end{proof}
Let $H$ be a reductive group over $S$ with geometric fibres of type $E_8$. We define a stable $\bbZ / 3 \bbZ$-grading of $H$ to be a homomorphism $\theta : \mu_3 \to H$ such that for each geometric point $\overline{s} \to S$, there exists a maximal torus $\Cen \subset H_{\overline{s}}$ that is normalized by the image of $\theta$ and such that the induced map $\mu_3 \to \Aut(X^\ast(\Cen))$ is an elliptic $\mu_3$-action (this definition makes sense since $X^\ast(\Cen)$ is an $E_8$ root lattice). Note that any such $\theta$ is then automatically a closed immersion, cf. \cite[Lemma B.1.3]{Con14}.

The next lemma shows that any two stable $\bbZ / 3 \bbZ$-gradings are conjugate \'etale locally on the base. 
\begin{lemma}\label{lem_graded_groups_etale_locally_isomorphic}
	Let $S$ be a $\bbZ[1/3]$-scheme. Let $(H, \theta)$ and $(H', \theta')$ be two pairs, each consisting of a reductive group over $S$ with geometric fibres of type $E_8$ and a stable $\bbZ / 3 \bbZ$-grading. Then for any $s \in S$ there exists an \'etale morphism $S' \to S$ with image containing $s$ and an isomorphism $H_{S'} \to H'_{S'}$ intertwining $\theta_{S'}$ and $\theta'_{S'}$.
\end{lemma}
\begin{proof}
	The question is \'etale local on $S$, so we can assume that $H = H'$ are both split reductive groups. Let $T$ denote the scheme of elements $h \in H$ such that $\Ad(h) \circ \theta = \theta'$; this is a closed subscheme of $H$ that is smooth over $S$, by \cite[Proposition 2.1.2]{Con14}. Since surjective smooth morphisms have sections \'etale locally, we just need to show that $T \to S$ is surjective. Since the formation of $T$ commutes with base change, we are therefore free to assume that $S = \Spec k$ is the spectrum of an algebraically closed field. 
	
	In this case, there exist (by assumption) maximal tori $\Cen$, $\Cen' \subset H$ on which $\theta$, $\theta'$ act through elliptic automorphisms of order 3. Using the conjugacy of maximal tori, we can therefore assume that $\Cen = \Cen'$. Using Lemma \ref{lem_weyl_group}, we can assume that $\theta$, $\theta'$ define the same element of the Weyl group of $\Cen$. 
	
	We have therefore reduced the problem to the statement that if $w \in W(H, \Cen)$ is an elliptic element of order 3, then any two lifts $n, n'$ of $w$ to the normalizer $N_H(\Cen)(k)$ are $H(k)$-conjugate. In fact, they are even $\Cen(k)$-conjugate, as follows from the fact that the morphism $1 - w : \Cen \to \Cen$ is \'etale (and surjective). This completes the proof.
\end{proof}

\section{A morphism of stacks}\label{sec_a_functor}

 In this section we describe a functorial construction of a graded Lie algebra from a Heisenberg group satisfying some conditions. We will later observe that the input data can be constructed from a genus-2 curve with a Weierstrass point (see \S \ref{sec_twisting}). 

\subsection{Two stacks}
We first need to introduce some notation. Let $N = 2 \times 3 \times 5 \times 7$. Let $R$ be a $\bbZ[1/N]$-algebra.
We write $\operatorname{Heis}_R$ for the groupoid of triples $(\Lambda, \theta, \mathscr{H})$, where:
\begin{enumerate}
	\item $\Lambda$ is an \'etale sheaf of $E_8$ root lattices on $\Spec R$ in the sense described in \S \ref{sec_root_lattice_of_type_E8} with symmetric pairing $( \cdot, \cdot ) : \Lambda \times \Lambda \to \bbZ$.
	\item $\theta : \mu_3 \to \Aut(\Lambda)$ is an elliptic $\mu_3$-action on $\Lambda$. 
	\item $\mathscr{H}$ is a central extension
	\[ 1 \to \mu_3 \to \mathscr{H} \to \Lambda_\theta \to 1 \]
	of \'etale $R$-groups, with the property that the induced commutator pairing $\Lambda_\theta \times \Lambda_\theta \to \mu_3$ is the same as the pairing $\langle \cdot, \cdot \rangle: \Lambda_\theta \times \Lambda_\theta \to \mu_3$ given by (\ref{eqn_reeders_pairing}).
\end{enumerate}
Morphisms $(\Lambda, \theta, \mathscr{H}) \to (\Lambda', \theta', \mathscr{H}')$ in $\operatorname{Heis}_R$ are pairs of isomorphisms $\alpha : \Lambda \to \Lambda'$, $\beta : \mathscr{H} \to \mathscr{H}'$ such that $\al$ intertwines $\theta$ and $\theta'$ and such that the induced diagram
\[ \xymatrix{ 1 \ar[r] & \mu_3 \ar[r]\ar[d]^= & \mathscr{H}\ar[d]^\beta \ar[r] & \Lambda_\theta \ar[r]\ar[d]^{\alpha_\theta} & 1 \\ 
	1 \ar[r] & \mu_3 \ar[r] & \mathscr{H}' \ar[r] & \Lambda'_{\theta'} \ar[r] & 1	}\]
commutes. (Here $\al_\theta$ is the map naturally induced by $\al$.) The groupoids $\operatorname{Heis}_R$ fit together into a stack $\operatorname{Heis} \to \operatorname{Alg}_{\bbZ[1/N]}^{\text{op}}$ in the \'etale topology.

Now let $R$ again be a $\bbZ[1/N]$-algebra. We write $\operatorname{GrLieT}_R$ for the groupoid of triples $(H, \theta, \Cen)$, where:
\begin{enumerate}
	\item $H$ is a reductive group over $R$ with geometric fibres all of Dynkin type $E_8$.
	\item $\Cen \subset H$ is a maximal torus. 
	\item $\theta : \mu_3 \to H$ is a homomorphism whose image normalizes $\Cen$, and such that the induced map $\theta : \mu_3 \to \Aut(X^\ast(\Cen))$ is an elliptic $\mu_3$-action.
\end{enumerate}
Morphisms $(H, \theta, \Cen) \to (H', \theta', \Cen')$ in this category are given by isomorphisms $\gamma : H \to H'$ sending $\Cen$ to $\Cen'$ and intertwining $\theta$ and $\theta'$. The groupoids $\operatorname{GrLieT}_R$ again fit together into a stack $\operatorname{GrLieT} \to \operatorname{Alg}_{\bbZ[1/N]}^{\text{op}}$ in the \'etale topology.

It seems likely that both $\operatorname{Heis}$ and $\operatorname{GrLieT}$ are algebraic stacks, and even global quotient stacks. We have chosen not to pursue this point of view here since we don't need it. We will introduce variants of these stacks in \S \ref{sec_twisting}. In particular, we will introduce the stack $\operatorname{GrLie}$ of pairs $(H, \theta)$ (we forget the torus).
\subsection{Definition of the morphism}
 The main goal of \S \ref{sec_a_functor} is to prove the following theorem.
\begin{contheorem}\label{thm_equivalence_of_Heis_with_GrLieT}
There is a morphism of stacks $\operatorname{Heis} \to \operatorname{GrLieT}$.
\end{contheorem}
Let $R \in \operatorname{Alg}_{\bbZ[1/N]}$ and let $(\Lambda, \theta, \mathscr{H}) \in \operatorname{Heis}_R$. Let us say that $(\Lambda, \theta, \mathscr{H})$ is split if $\mu_3$, $\Lambda$ and $\mathscr{H}$ are constant as \'etale sheaves on $\Spec R$. Since every object of $\operatorname{Heis}$ is split locally in the \'etale topology, to prove Theorem \ref{thm_equivalence_of_Heis_with_GrLieT} it will be enough to construct the morphism of Theorem \ref{thm_equivalence_of_Heis_with_GrLieT} on the full sub-fibred category of split objects. This is what we will do. We choose for each $R \in \operatorname{Alg}_{\bbZ[1/N]}$ a primitive 3rd root of unity $\zeta_R$; our construction will be independent of this choice, up to equivalence.   We can and do assume that $\Spec R$ is connected.

Take a split triple $(\Lambda, \theta, \mathscr{H}) \in \operatorname{Heis}_R$. We will abuse language by thinking of the constant sheaves $\Lambda$ and $\mathscr{H}$ as abstract groups. Let $w = \theta(\zeta_R)$. Let $\Phi \subset \Lambda$ be the set of roots, let $\widetilde{\Lambda} = \Lambda \times_{\Lambda_\theta} \mathscr{H}$, and let $\widetilde{\Phi}$ denote the pre-image of $\Phi$ in $\widetilde{\Lambda}$. Thus $\widetilde{\Lambda}$ is a central extension
	\begin{equation*}
	1 \to \mu_3(R) \to \widetilde{\Lambda} \to \Lambda \to 1
	\end{equation*}
	that also has commutator pairing given by $\langle \cdot, \cdot \rangle$, and $\widetilde{\Phi} \to \Phi$ is a 3-fold cover. For an element $\widetilde\al \in \widetilde\Phi$, we write $\al$ for the image of $\widetilde\al$ in $\Phi$. 

Let $\Cen$ be the torus over $R$ with $X^\ast(\Cen) = \Lambda$, and let $\cen = \Lie \Cen$. For any $\alpha \in \Phi$, the coroot $\check{\alpha}$ corresponds to an element of $\cen$. In fact $\cen$ is generated by $\{\check\al \mid \al \in \Phi\}$. Let $\frh_1$ be the quotient of the free $R$-module with basis $\{X_{\widetilde{\alpha}} \mid \widetilde{\alpha} \in \widetilde{\Phi}\}$, by the relations $X_{\zeta_R \widetilde{\alpha}} = \zeta_R X_{\widetilde{\alpha}}$ (for any $\widetilde{\alpha} \in \widetilde{\Phi}$). Finally, let $\frh = \cen \oplus \frh_1$. Thus $\frh$ is a free $R$-module of rank 248 generated by $\{ \check\al, X_{\widetilde\be} \mid \al \in \Phi, \widetilde\be \in \widetilde\Phi\}$ (by abuse of notation we also write $X_{\widetilde\be}$ for the image of this vector in $\frh_1$).

To motivate our definition of the Lie bracket on $\frakh$, we note that we have constructed $\frakh$ so that $\cen$ will be a Cartan subalgebra, and for each $\widetilde{\al} \in \widetilde{\Phi}$, we would like $X_{\widetilde{\al}}$ to be a basis for the $\al$-root space with respect to $\cen$. The definitions of $[\check\al, \check\be]$ and $[\check\al, X_{\widetilde\be}]$ (for $\al \in \Phi, \wt{\be} \in \wt\Phi$) are completely determined by these choices, and the definition of $[X_{\wt\al}, X_{\wt\be}]$ (for $\wt\al, \wt\be \in \wt\Phi$) is determined up to scalar. It is a well-known difficulty in Lie theory to choose these scalars in a consistent way. Here, we take advantage of the extra structure afforded by the triple $(\Lambda, \theta, \mathscr{H})$ to define these scalars using multiplication in $\wt\Lambda$. This strategy was inspired by \cite{Lurie} and \cite{Lep85}, and the construction is generalized in \cite{Rom18a}.

Thus we define a bilinear map $[ \cdot, \cdot ] : \frh \times \frh \to \frh$ as follows. We set $[x, y] = 0$ for any $x, y \in \cen$. We let
\[ [ \check{\alpha}, X_{\widetilde{\beta}} ] = -[  X_{\widetilde{\beta}},  \check{\alpha} ] = ( \alpha, \beta ) X_{\widetilde{\beta}} \]
for any $\al \in \Phi, \wt\be \in \wt\Phi$. 
Finally, the bracket of vectors $X_{\widetilde{\alpha}}$, $X_{\widetilde{\beta}}$ is defined by the formula
\[ [ X_{\widetilde{\alpha}}, X_{\widetilde{\beta}} ] = \left\{ \begin{array}{cc} -\widetilde{\alpha} \widetilde{\beta} \check\alpha & \text{if }\alpha + \beta = 0. \\
(-1)^{(\al, w\be)}\la \al, \be \ra X_{\wt\al\wt\be} & \text{if }\alpha + \beta \in \Phi. \\
0 & \text{otherwise.} \end{array}\right. \]
We observe that the map is well-defined, i.e. it respects the defining relation $X_{\zeta_R \widetilde{\alpha}} = \zeta_R X_{\widetilde{\alpha}}$.
\begin{proposition}\label{prop_jacobi}
	With the above definition, $\frh$ is a Lie algebra (i.e.\ the bracket $[ \cdot, \cdot]$ is alternating and satisfies the Jacobi identity). 
\end{proposition}
In order to prove the proposition, we first make the following observation.

\begin{lemma}\label{lem_weyl_negative}
	If $\al, \be, \al + \be \in \Phi$, then $(-1)^{(\al, w\be)} + (-1)^{(w\al, \be)} = 0$.
\end{lemma}

\begin{proof}
	We have $(-1)^{(\al, w\be)} = (-1)^{(w^2\al, \be)} = (-1)^{(-\al -w\al, \be)} = (-1)^{(w\al, \be) + 1}$ since $(\al, \be) = -1$.
\end{proof}

We also point out the useful fact that because the pairing $\la \cdot, \cdot \ra$ is alternating, we have $\wt\al\wt\be = \wt\be\wt\al$ whenever $\al + \be = 0$.

\begin{proof}[Proof of Proposition \ref{prop_jacobi}]
	Using Lemma \ref{lem_weyl_negative} and the fact that the pairing $\la \cdot, \cdot \ra$ is alternating, it is not hard to check that the bracket $[\cdot, \cdot]$ is alternating. Thus it suffices to check the Jacobi identity.
	Consider
	\begin{equation}\label{eqn-jac1}
	[x,[y,z]] + [y,[z,x]] + [z,[x,y]]
	\end{equation}
	for generators $x, y, z$. If any of $x, y, z$ are in $\cen$, then 
	it follows easily from the definition of the bracket that (\ref{eqn-jac1}) is zero. 

	Thus we restrict our attention to the case when $x = X_{\widetilde\al}, y = X_{\widetilde\be}$ and $z = X_{\widetilde\gamma}$ for some $\widetilde\al, \widetilde\be, \widetilde\gamma \in \widetilde\Lambda$.

	First suppose $\al + \be + \gamma = 0$. 
	Then $\be + \gamma = -\al \in \Phi$, and similarly $\be + \gamma, \al + \be \in \Phi$. So we have
	\begin{eqnarray*}
		[x,[y,z]] + [y,[z,x]] + [z,[x,y]] &=&
		(-1)^{(\be, w\gamma)}\la\be, \gamma\ra[X_{\widetilde\al},X_{\widetilde\be\widetilde\gamma}] + (-1)^{(\gamma, w\al)}\la \gamma, \al\ra[X_{\widetilde\be},X_{\widetilde\gamma\widetilde\al}]\\ & &+ (-1)^{(\al, w \be)}\la\al, \be\ra[X_{\widetilde\gamma},X_{\widetilde\al\widetilde\be}]\\
		&=&
		-\wt\al\wt\be\wt\gamma\left[(-1)^{(\be, w\gamma)}\la\be, \gamma\ra \check\al + (-1)^{(\gamma, w\al)}\la\gamma, \al\ra\check\be + (-1)^{(\al, w\be)}\la\al,\be\ra\check\gamma\right].
	\end{eqnarray*}

	Replacing $\gamma$ by $-\al - \be$ and $\check\gamma$ by $-\check\al - \check\be$ and using the fact that $(\be, w\be) = (\al, w\al) = -1$, we may simplify this equation to
	\begin{equation*}
	[x,[y,z]] + [y,[z,x]] + [z,[x,y]] =
	\widetilde\al\widetilde\be\widetilde\gamma\left[(-1)^{(\be, w\al)}\la\be,-\al\ra + (-1)^{(\al, w\be)}\la\al,\be\ra\right](\check\al + \check\be),
	\end{equation*}
	which is zero by Lemma \ref{lem_weyl_negative}.
	
	For the rest of the proof we assume $\al + \be + \gamma \neq 0$. For (\ref{eqn-jac1}) to be non-zero,  at least one term in (\ref{eqn-jac1}) must be non-zero. Without loss of generality, we may assume the first term is non-zero, and so either $\be + \gamma = 0$ or $\be + \gamma \in \Phi$. We deal with each of these cases separately.
	
	Case 1: $\be + \gamma = 0$. 
	
	In this case the first term of (\ref{eqn-jac1}) is $(\al, \be)\widetilde\be\widetilde\gamma X_{\widetilde\al}$. By assumption $(\al, \be) \neq 0$. Suppose $(\al, \be) = -1$. Then $(\al, \gamma) = 1$ and $(\gamma, \al + \be) = -1$, so 
	\begin{eqnarray*}
		[x,[y,z]] + [y,[z,x]] + [z,[x,y]] &=&
		-\widetilde\be\widetilde\gamma X_{\widetilde\al} + (-1)^{(\al, w\be) + (\gamma, w\al + w\be)}\la\al, \be\ra \la \gamma, \al + \be\ra X_{\widetilde\gamma\widetilde\al\widetilde\be}\\
		&=& -\widetilde\be\widetilde\gamma X_{\widetilde\al} +\la\al, \be\ra^{-1}X_{\widetilde\gamma\widetilde\al\widetilde\be}.
	\end{eqnarray*}
	Note that
	\begin{equation*}
	X_{\wt\gamma\wt\al\wt\be} = \la\al, \be\ra X_{\wt\gamma\wt\be\wt\al} = \la\al, \be\ra \wt\gamma\wt\be X_{\wt\al},
	\end{equation*}
	where we are using the fact that $\widetilde\be\widetilde\gamma \in \mu_3$. Thus (\ref{eqn-jac1}) is zero in the case when $(\al, \be) = -1$.
	The case when $(\al, \be) = 1$ is similar.
	
	If $(\al, \be) = -2$ then $\al = \gamma = -\be$, so 
	\begin{eqnarray*}
		[x,[y,z]] + [y,[z,x]] + [z,[x,y]] 
		&=& -2\widetilde\be\widetilde\gamma X_{\widetilde\al} + 2\widetilde\al\widetilde\be X_{\widetilde\gamma}\\
		&=& -2\widetilde\be\widetilde\gamma X_{\widetilde\al} + 2\widetilde\al\widetilde\be \widetilde\gamma(\widetilde\al)^{-1}X_{\widetilde\al}.
	\end{eqnarray*}
	Since $\widetilde\al$ commutes with $\widetilde\be$ and $\widetilde\gamma$ commutes with $(\widetilde\al)^{-1}$, we see that this is zero. The case when $(\al, \be) = 2$ is similar. 
	
	Case 2: $\be + \gamma \in \Phi$. 
	
	Since we are assuming that the first term in (\ref{eqn-jac1}) is non-zero, we have that $\al + \be + \gamma \in \Phi$, and so $(\be, \gamma) = (\al, \be + \gamma) = -1$.
	Thus $(\al, \be) + (\al, \gamma) = -1$, and at least one of $(\al, \be)$ or $(\al, \gamma)$ is less than 0. If $(\al, \be)$ or $(\al, \gamma)$ is $-2$, then the proof reduces to that of Case 1. 

	Suppose $(\al, \be) = -1$. Then 
	\begin{eqnarray*}
		[x,[y,z]] + [y,[z,x]] + [z,[x,y]] &=& (-1)^{(\al, w\be + w\gamma) + (\be, w\al)}\la\al, \be + \gamma\ra\la\be, \gamma\ra X_{\widetilde\al\widetilde\be\widetilde\gamma}\\& & + (-1)^{(\al, w\be) + (\gamma, w\al + w\be)}\la\al, \be\ra\la\gamma, \al + \be\ra X_{\widetilde\gamma\widetilde\al\widetilde\be}\\
		&=& (-1)^{(\al, w\be)}\la\al, \be\ra\left[(-1)^{(\al + \be, w\gamma)}\la\al + \be, \gamma\ra + (-1)^{(\gamma, w\al + w\be)}\langle\gamma, \al + \be\rangle^2\right] X_{\widetilde\al\widetilde\be\widetilde\gamma},
	\end{eqnarray*}
	which is zero by Lemma \ref{lem_weyl_negative} since $\la \al + \be, \gamma \ra = \la \gamma, \al + \be \ra^{2}$. If $(\al, \gamma) = -1$, the proof is similar.
\end{proof}
Thus $\frh$ is a Lie algebra over $R$. Our assumption that $N = 2 \times 3 \times 5 \times 7$ is a unit in $R$ implies that the Killing form of $\frh$ is non-degenerate, and therefore (by the main theorem of \cite{Vas16}) that $H = \Aut(\frh)$ is a reductive group over $R$ with geometric fibres of Dynkin type $E_8$. Moreover, we have $\Lie H = \frh$ and we can identify $Z_H(\cen) = \Cen$. 

To describe the homomorphism $\mu_3 \to H$, we first extend the action of $\mu_3$ on $\Lambda$ to $\widetilde{\Lambda} = \Lambda \times_{\Lambda_\theta} \mathscr{H}$ by letting $\mu_3$ act trivially on $\mathscr{H}$. We define a map $\theta_{\mathscr{H}} : \mu_3 \to H$ using the action on $\cen$ induced from $\theta$, together with the formula
\[ \theta_{\sH}(\zeta_R)(X_{\widetilde{\alpha}}) = X_{\theta(\zeta_R)(\widetilde{\alpha})}. \]
\begin{lemma}
	The action of $\theta_{\sH}(\zeta_R)$ on $\frh$ 
	preserves the Lie bracket. Consequently, the map $\theta_{\sH} : \mu_3 \to H$ is defined.
\end{lemma}
\begin{proof}
	We again write $w$ for the extension of $w = \theta(\zeta_R)$ to $\wt\Lambda$, and we let $w' = \theta_{\sH}(\zeta_R)$. In order to prove the lemma, it suffices to show that 
	\begin{equation}\label{eqn-lieaut}
	[w'(x), w'(y)] = w'([x,y])
	\end{equation}
	for any generators $x, y$ of $\frakh$. If $x, y \in \cen$, then both sides of (\ref{eqn-lieaut}) are zero. If $x = \check\al$ and $y= X_{\wt\be}$ for some $\al \in \Phi, \wt\be \in \wt\Lambda$, then equality in (\ref{eqn-lieaut}) follows from the fact that $(w\al, w\be) = (\al, \be)$. Suppose $x = X_{\wt\al}$ and $y = X_{\wt\be}$ for some $\wt\al, \wt\be \in \wt\Lambda$. If $(\al, \be) \geq 0$, then both sides of (\ref{eqn-lieaut}) are zero, and if $\al + \be \in \Phi$, then equality in (\ref{eqn-lieaut}) follows from the fact that
	$\la w\al, w\be \ra = \la \al, \be \ra$. If  $\al + \be = 0$, then 
	\begin{eqnarray*}
		[w'(X_{\wt\al}), w'(X_{\wt\be})] &=& -(w\wt\al)(w\wt\be)w(\check\al)\\
		&=& -w(\wt\al\wt\be)w(\check\al)\\
		&=& -\wt\al\wt\be w(\check\al),
	\end{eqnarray*}
	where we are using that $\wt\al\wt\be \in \mu_3(R)$. Thus we again have equality in (\ref{eqn-lieaut}). 
\end{proof}
This completes the construction of the triple $(H, \theta_{\sH}, \Cen)$ associated to a split object $(\Lambda, \theta, \mathscr{H}) \in \operatorname{Heis}_R$. We must next consider morphisms. Let $(\Lambda, \theta, \mathscr{H}) \in \operatorname{Heis}_R$, $(\Lambda', \theta', \mathscr{H}') \in \operatorname{Heis}_{R'}$ be split, and suppose that $(f, g) : (\Lambda', \theta', \mathscr{H}') \to (\Lambda, \theta, \mathscr{H}) $ is a morphism in $\operatorname{Heis}$; equivalently, a ring map $f : R \to R'$ and an isomorphism $g : (\Lambda', \theta', \mathscr{H}') \to f^\ast(\Lambda, \theta, \mathscr{H})$. Let $(H, \theta, A) \in \operatorname{GrLieT}_R$ and $(H', \theta', A') \in \operatorname{GrLieT}_{R'}$ be the corresponding triples. We must construct the morphism $(H', \theta', A') \to (H, \theta, A)$ in $\operatorname{GrLieT}$ corresponding to $(f, g)$. It will suffice to construct a Lie algebra isomorphism $\frh' \to \frh \otimes_R R'$. We can again assume that both $\Spec R$ and $\Spec R'$ are connected. We can define an element $x_f \in \{ \pm 1 \} \subset (R')^\times$ defined by the formula $f(\zeta_R) = \zeta_{R'}^{x_f}$; equivalently, $f(\zeta_R / (1-\zeta_R^{-1})) = x_f \zeta_{R'} / (1 - \zeta_{R'}^{-1})$.  Since $\cen = \Hom(\Lambda, R)$, there is an obvious homomorphism $\cen' \to \cen \otimes_R R'$.  We extend this to a map $\sigma_{(f, g)} = \sigma : \frh' \to \frh \otimes_R R'$ by sending $X_{\widetilde{\alpha}}$ to $x_f \cdot X_{g(\widetilde{\alpha})}$. 
\begin{lemma}
The homomorphism $\sigma : \frh' \to  \frh \otimes_R R'$ is an isomorphism of Lie algebras; in other words, it satisfies the identity
	\begin{equation}\label{eqn_gamma_action}
	\sigma([x, y]) = [\sigma(x), \sigma(y)]
	\end{equation}
for all $x, y \in \frh'$.
\end{lemma}
\begin{proof}
It suffices to check this identity in the case $x = X_{\wt\al}$ and $y = X_{\wt\be}$ for some $\wt\al, \wt\be \in \wt\Phi$. If $x_f = 1$, then this is clear. Suppose instead that $x_f = -1$. We split into cases. If $\alpha + \beta = 0$, the both sides of (\ref{eqn_gamma_action}) are equal to $-g(\check{\alpha}) \otimes \widetilde{\alpha} \widetilde{\beta} \in \cen \otimes_R R'$.
	
	If $\alpha + \beta$ is a root, then the left-hand side of (\ref{eqn_gamma_action}) equals $ X_{g(\widetilde{\alpha} \widetilde{\beta})} \otimes (-1)^{(\alpha, \theta'(\zeta_{R'})(\beta)) + 1} \langle \alpha, \beta \rangle$,  while the right-hand side equals $ X_{g(\widetilde{\alpha} \widetilde{\beta})} \otimes f((-1)^{(g(\alpha),  \theta(\zeta_R) g(\beta))}\la g(\al), g(\be)\ra)$. Since $g$ intertwines $\theta'(\zeta_{R'})$ and $\theta(\zeta_R)^{-1}$, we have
	\begin{equation*}
	(-1)^{(g(\alpha), \theta(\zeta_R)(g(\beta)))} = (-1)^{(\theta'(\zeta_R')( \al), \be)} =  (-1)^{(\alpha, \theta'(\zeta_{R'})(\beta)) + 1},
	\end{equation*}
	showing that both sides of (\ref{eqn_gamma_action}) are equal.
\end{proof}
This completes the proof of Theorem \ref{thm_equivalence_of_Heis_with_GrLieT}. 

We observe that if $(\Lambda, \theta, \mathscr{H}) \in \operatorname{Heis}_R$, then there is a morphism $f_R : H^0(R, \Lambda_\theta) \to \Aut_{\operatorname{Heis}_R}(\Lambda, \theta, \mathscr{H})$ defined as follows: if $\lambda \in H^0(R, \Lambda_\theta)$, then $f_R(\lambda)$ acts as the identity on $\Lambda$ and as $\mu \mapsto \langle \lambda, \mu \rangle \mu$ on $\mathscr{H}$. Let $(H, \theta_{\sH}, \Cen) \in \operatorname{GrLieT}_R$ be the tuple corresponding to $(\Lambda, \theta, \mathscr{H})$ under the construction of Theorem \ref{thm_equivalence_of_Heis_with_GrLieT}. Varying $R$ and taking into the account the functorial nature of our construction, we obtain a morphism of group schemes over $R$:
\begin{equation}\label{eqn_action_of_coinvariants_by_automorphisms} \Lambda_\theta \to \Aut(H, \theta_{\sH}, A) = N_H(A)^{\theta_{\sH}}. \end{equation}
We can describe this explicitly:
\begin{lemma}\label{lem_action_by_automorphisms}
	Let notation be as in the above discussion. Then there is a canonical isomorphism $\Lambda_\theta \cong \Cen^\theta$, under which the morphism (\ref{eqn_action_of_coinvariants_by_automorphisms}) corresponds to the adjoint action of $\Cen^\theta = Z_H(A)^{\theta_{\sH}} \subset N_H(A)^{\theta_{\sH}} = \Aut(H, \theta_{\sH}, A)$.
\end{lemma}
\begin{proof}
	By definition, we have $X^\ast(\Cen) = \Lambda$, hence $X_\ast(\Cen) \cong \Lambda^\vee = \Hom(\Lambda, \bbZ)$. There is a canonical isomorphism $\Cen^\theta \cong (\Lambda^\vee \otimes \mu_3)^\theta$. 
There is also an isomorphism $\Lambda_\theta \cong (\Lambda^\vee \otimes \mu_3)^\theta$, given by the formula $\lambda \mapsto (1 - \theta(\zeta))\check\lambda \otimes \zeta$ for $\lambda \in \Lambda$; this does not depend on the choice of $\zeta$ and also depends only on the image of $\lambda$ in $\Lambda_\theta$.
	Composing the above two isomorphisms gives the desired isomorphism $\Lambda_\theta \cong \Cen^\theta$. 
	
	We now need to check that under this isomorphism, the action of $ \Cen^\theta$ on the triple $(H, \theta_{\mathscr{H}}, A)$ is the adjoint action. It suffices to check this in the case that the triple $(\Lambda, \theta, \mathscr{H})$ is split. Let us therefore assume this and fix $\lambda \in \Lambda_\theta$. By definition, $\lambda$ acts as the identity on $\cen$ and sends the vector $X_{\widetilde{\alpha}}$ to $X_{\langle \lambda, \alpha \rangle \widetilde{\alpha}}$. In other words, it leaves invariant the $\alpha$-root space and acts by the scalar $\langle \lambda, \alpha \rangle$ there. 
	
	On the other hand, the element $(1 - \theta(\zeta_R))\check\lambda(\zeta_R)$ in $\Cen^\theta$ also acts as the identity on $\Cen$ and leaves invariant each root space, acting on the $\alpha$-root space by the scalar
	\[ \zeta_R^{((1 - \theta(\zeta_R))\lambda, \alpha)} = \langle \lambda, \alpha \rangle. \]
	This completes the proof.
\end{proof}

\subsection{Identifying $\frh(0)$}

Theorem \ref{thm_equivalence_of_Heis_with_GrLieT} associates to any triple $(\Lambda, \theta, \mathscr{H}) \in \operatorname{Heis}_R$ a triple $(H, \theta_{\sH}, \Cen) \in \operatorname{GrLieT}_R$. In the proof of our next result, we show that if $W$ is an irreducible representation of $\mathscr{H}$ on which the central $\mu_3$ acts by scalar multiplication through its tautological character, then we can identify $\frh(0)$ with $\frs\frl(W)$ and $\mathscr{H}$ with a certain subgroup of $\SL(W)$. (Note that over an algebrically closed field of characteristic prime to 3, any two such representations are isomorphic, by the Stone--von Neumann theorem, and have dimension 9.) This result will play an essential role in the construction of orbits in \S \ref{sec_twisting}.
\begin{theorem}\label{thm_embedding_Heis_in_Gsc}
	Let $(\Lambda, \theta, \mathscr{H}) \in \operatorname{Heis}_R$. Let $W$ be a locally free $R$-module of rank 9, and suppose $\rho : \mathscr{H} \to \Aut_R(W)$ is a homomorphism such that the central $\mu_3$ acts on $W$ through its tautological character. Let $(H, \theta_{\sH}, \Cen)$ denote the image of $(\Lambda, \theta, \mathscr{H})$ under the functor of Theorem \ref{thm_equivalence_of_Heis_with_GrLieT} and let $G = H^{\theta_{\sH}}$. Then:
	\begin{enumerate}
		\item $G$ is a semisimple reductive group, and the simply connected cover $G^\text{sc}$ of $G$ is isomorphic to $\SL(W)$.
\item	There is a commutative diagram of $R$-groups with exact rows:
		\[ \xymatrix{ 1 \ar[r] & \mu_3 \ar[r] & G^\text{sc} \ar[r] & G \ar[r] & 1 \\ 1 \ar[r] & \mu_3 \ar[r] \ar[u]^= & \mathscr{H} \ar[u] \ar[r] & \Lambda_\theta \ar[u] \ar[r] & 1, } \]
		where the map $\Lambda_\theta \to G$ is induced by the canonical isomorphism $\Lambda_\theta \cong \Cen^\theta$ of Lemma \ref{lem_action_by_automorphisms} and the map $\sH \to G^\text{sc}$ is induced by $\rho$.
	\end{enumerate}
\end{theorem}
To prove the theorem, we can again assume that $\Spec R$ is connected and that the triple $(\Lambda, \theta, \mathscr{H})$ is split. The group $G$ is reductive with geometric fibres of type $\SL_9 / \mu_3$, so its simply connected cover $G^\text{sc}$ exists, and the kernel of the natural covering map $G^\text{sc} \to G$ is a group of multiplicative type over $R$ of order 3. Let $\frg = \Lie G$. Then $\frg = \frh^{\theta_{\sH}}$ (see e.g. \cite[Lemma 2.2.4]{Con14}).

The first step in the proof of Theorem \ref{thm_embedding_Heis_in_Gsc} is to define an action of the Lie algebra $\frg$ on $W$; equivalently, to define a map $\frg \to \End_{R}(W_{R})$. 
Recall that we have extended $\theta$ to a map $\mu_3 \to \Aut(\wt\Lambda)$. If $\widetilde{\alpha} \in \widetilde{\Phi}$, define $Z_{\widetilde{\alpha}} = X_{\widetilde{\alpha}} + X_{\theta(\zeta_R)(\widetilde{\alpha})} + X_{\theta(\zeta_R^2)(\widetilde{\alpha})} \in \frg$. Elements of this form span $\frg$.  

\begin{proposition}
Let $\pi : \widetilde{\Lambda} \to \mathscr{H}$ denote the canonical projection. Define a map $\rho' : \frg \to \End_{R}(W)$ by the formula 
\[ \rho'(Z_{\widetilde{\alpha}}) = \zeta_R(1-\zeta_R^{-1})^{-1} \rho(\pi(\widetilde{\alpha})). \]
	Then $\rho'$ is a well-defined Lie algebra homomorphism.
\end{proposition}
\begin{proof}
	We see that $\rho'$ is well defined exactly because $\rho(\zeta_R) = \zeta_R \cdot 1_W$. The key point is to check that the Lie bracket is preserved, or in other words that the relation
	\begin{equation}\label{eqn_lie_bracket_preserved}
	\rho'([Z_{\widetilde{\alpha}}, Z_{\widetilde{\beta}}]) = [ \rho'(Z_{\widetilde\al}), \rho'(Z_{\widetilde\be})]
	\end{equation}
	holds for all $\widetilde{\alpha}, \widetilde{\beta} \in \widetilde{\Phi}$. 
	We give a case-by-case-proof depending on the value of $(\alpha, \beta )$.  Let $w = \theta(\zeta_R)$.
	Before beginning, we note again the useful fact that if $\alpha \in \Phi$, then $\alpha + w(\alpha) + w^2(\alpha) = 0$. In particular, $\alpha + w(\alpha)$ and $\alpha + w^2(\alpha)$ are roots. 
	
	Case 1: $(\al, \be) = \pm 2$. 

If $( \alpha, \beta) = \pm 2$, then $\alpha = \pm \beta$. If $\alpha = \beta$ then both sides of (\ref{eqn_lie_bracket_preserved}) vanish. If $\alpha = - \beta$, then the right-hand side vanishes because $\pi(\widetilde{\alpha})$ and $\pi(\widetilde{\beta})$ commute in $\mathscr{H}$. On the other hand, $[Z_{\widetilde{\alpha}}, Z_{\widetilde{\beta}}]$ equals
	\begin{equation}\label{eqn_heisenberg_sum} \begin{split}   &  [ X_{\widetilde{\alpha}}, X_{\widetilde{\beta}}] + [ X_{w\widetilde{\alpha}}, X_{w\widetilde{\beta}}] + [ X_{w^2\widetilde{\alpha}}, X_{w^2\widetilde{\beta}}]  \\ +& [ X_{\widetilde{\alpha}}, X_{w\widetilde{\beta}}] + [ X_{w\widetilde{\alpha}}, X_{w^2\widetilde{\beta}}] + [ X_{w^2\widetilde{\alpha}}, X_{\widetilde{\beta}}]  \\ + &[ X_{\widetilde{\alpha}}, X_{w^2\widetilde{\beta}}] + [ X_{w\widetilde{\alpha}}, X_{\widetilde{\beta}}] + [ X_{w^2\widetilde{\alpha}}, X_{w\widetilde{\beta}}] . \end{split} \end{equation}
	The first line of (\ref{eqn_heisenberg_sum}) is zero because it is an element of $\frg \cap \cen = 0$. The second line vanishes because $\alpha - w \alpha$ is not a root. The third line vanishes because $\alpha - w^2 \alpha$ is not a root. Therefore both sides of (\ref{eqn_lie_bracket_preserved}) are zero in this case. 
	
	We note that if any of $(\al, w\be) = \pm 2, (\al, w^2\be) = \pm 2, (w\al, \be) = \pm 2$, or $(w^2\al, \be) = \pm 2$, then because $Z_{\wt\al} = Z_{w\wt\al} = Z_{w^2\wt\al}$ and similarly for $Z_{\wt\be}$, we still have that both sides of (\ref{eqn_lie_bracket_preserved}) are zero, so for the rest of the proof we can, and do, assume that this is not the case.

	Case 2: $(\al, \be) = -1$. 

If $(\alpha, \beta ) = -1$, then $\alpha + \beta$ is a root. Note that the equation $(\al, \be) + (w\al, \be) + (w^2\al, \be) = 0$ implies that $(w\al, \be)$ and $(w^2\al, \be)$ are nonnegative, and similarly for $(\al, w\be)$ and $(\al, w^2\be)$. This implies that $[Z_{\wt\al}, Z_{\wt\be}] = (-1)^{(\al, w\be)}\la\al,\be\ra Z_{\wt\al\wt\be}$ and that $\la \al, \be \ra = \zeta_R^{-(w\al, \be) - 1} \neq 1$. Thus
	\begin{eqnarray*}
		[ \rho'(Z_{\widetilde\al}), \rho'(Z_{\widetilde\be})] &=& 
		\zeta_R^{-1}(1-\zeta_R^{-1})^{-2}\left[\rho(\pi(\widetilde\al )), \rho(\pi(\widetilde{\beta}))\right]\\ 
		&=& \zeta_R^{-1}(1- \zeta_R^{-1})^{-2} (1 - \rho(\pi(\widetilde\be\widetilde\al\widetilde{\be}^{-1} \widetilde{\al}^{-1})))\rho(\pi(\widetilde\al \widetilde\be))\\ 
		&=& \zeta_R^{-1}(1 - \zeta_R^{-1})^{-2}(1 - \la \be, \al \ra)\rho(\pi(\widetilde\al \widetilde\be)).
	\end{eqnarray*}
	If $\la \al, \be \ra = \zeta_R$, then $(w\al, \be) = 1$ and $(\al, w\be) = 0$, so 
	\begin{equation*}
	\zeta_R^{-1}(1 - \zeta_R^{-1})^{-2}(1 - \la \be, \al \ra)\rho(\pi(\widetilde\al \widetilde\be)) = \zeta_R^{-1}(1 - \zeta_R^{-1})^{-1} = \rho'(\left[Z_{\widetilde{\alpha}}, Z_{\widetilde{\beta}}\right]).
	\end{equation*}
	If $\la \al, \be \ra = \zeta_R^{-1}$, then $(w\al, \be) = 0$, $(\al, w\be) = 1$, and we again have equality in (\ref{eqn_lie_bracket_preserved}).

	Case 3: $(\alpha, \beta) \in \{0, 1\}$.

If $(\al, \be) = (w \alpha, \beta) = (w^2 \alpha, \beta) = 0$, then both sides of (\ref{eqn_lie_bracket_preserved}) are  0. Otherwise the equation $(\al, \be) + (w\al, \be) + (w^2\al, \be) = 0$ implies that either $(w\al, \be) = -1$ or $(w^2\al, \be) = -1$, and thus we may reduce to Case 2.
	This completes the proof that $\rho'$ is a Lie algebra homomorphism.
%	
%	It remains to check that $\rho'$ is equivariant for the action of $\Gamma$. We just need to note the formulae
%	\[ \sigma(Z_{\widetilde{\alpha}}) = \left\{ \begin{array}{rl} Z_{\sigma(\widetilde{\alpha})} & \text{ if }\sigma(\zeta) = \zeta; \\ - Z_{\sigma(\widetilde{\alpha})} & \text{ if }\sigma(\zeta) = \zeta^{-1} \end{array} \right. \]
%	and
%	$-\zeta / (1 - \zeta^{-1}) = \zeta^{-1} / (1-\zeta)$.
\end{proof}
The induced map $\frg \to \frs\frl(W)$ is an isomorphism and thus induces an isomorphism $G^\text{ad} \to \PGL(W)$, where $G^\text{ad}$ is the adjoint group of $G$ (by \cite{Vas16} once again). There is a unique isomorphism $G^\text{sc} \to \SL(W)$ that is compatible with the given map $\frg \to \frs\frl(W)$ (\cite[Exercise 6.5.2]{Con14}). Let $\mathscr{H}'$ denote the pre-image of $\Cen^\theta \cong \Lambda_\theta$ in $G^\text{sc}$ under the covering map $G^\text{sc} \to G$. Identifying the centre of $G^\text{sc}$ with $\mu_9$ via its action on $W$ (and therefore the kernel $G^\text{sc} \to G$ with $\mu_3$), we find that $\mathscr{H}'$ fits into a diagram
\[ \xymatrix{ 1 \ar[r] & \mu_3 \ar[r] & G^\text{sc} \ar[r] & G \ar[r] & 1 \\ 1 \ar[r] & \mu_3 \ar[r] \ar[u]^= & \mathscr{H}' \ar[u] \ar[r] & \Lambda_\theta \ar[u] \ar[r] & 1, } \]
where the vertical maps are induced by inclusion and the isomorphism $\Lambda_\theta \cong A^\theta$.
To complete the proof of Theorem \ref{thm_embedding_Heis_in_Gsc}, we must show that there is an isomorphism $\mathscr{H} \cong \mathscr{H}'$ of central extensions of $\Lambda_\theta$ by $\mu_3$. We will show that in fact the image of $\mathscr{H}$ under $\rho$ corresponds to $\mathscr{H}'$ under the isomorphism $G^\text{sc} \cong \SL(W)$. 

Let $p : G^\text{sc} \to G^\text{ad}$ denote the projection to the adjoint group. A straightforward calculation shows $\rho$ is injective, and thus we can characterize $\mathscr{H}$ as 
\[ \mathscr{H} = \{ g \in G^\text{sc} \mid p(g) \in p(\rho(\mathscr{H})), g^3 = 1 \}. \]
Identifying $A^\theta$ with $\Lambda_\theta$, we can characterize $\mathscr{H}'$ as 
\[ \mathscr{H}' = \{ g \in G^\text{sc} \mid p(g) \in p(\Lambda_\theta), g^3 = 1 \}. \]
To prove the theorem, it is therefore enough to show that the two homomorphisms $\Lambda_\theta \to \Aut(\frg)^\circ = G^\text{ad}$, one derived from $\Ad \circ \rho$, the other derived from $\Ad|_{\mathscr{H}'}$, are the same. Since $\frg$ is spanned by the elements $Z_{\widetilde{\beta}}$, and the non-trivial elements of $\Lambda_\theta$ are all of the form $\alpha \text{ mod}(\theta - 1)$, for some $\alpha \in \Phi$, it is enough to show that the two possible actions of $\alpha$ on $Z_{\widetilde{\beta}}$ are the same for all $\alpha \in \Phi$, $\widetilde{\beta} \in \widetilde{\Phi}$. 

For the first action, we see that, by definition, $\rho'(Z_{\widetilde{\beta}}) \in \frg = \frs\frl(W) \subset \End_{R}(W)$ is a scalar multiple of $\rho(\pi(\widetilde{\beta}))$. Therefore we have 
\begin{equation}\label{eqn_first_heis_action}
\Ad(\rho(\pi(\widetilde{\alpha})))(Z_{\widetilde{\beta}}) = \rho(\pi(\widetilde{\alpha} \widetilde{\beta} \widetilde{\alpha}^{-1} \widetilde{\beta}^{-1}))Z_{\widetilde{\beta}} = \langle \alpha, \beta \rangle Z_{\widetilde{\beta}}. 
\end{equation}
For the second action, we use the fact that $Z_{\widetilde{\beta}} = X_{\widetilde{\beta}} + X_{\theta(\zeta_R)(\widetilde{\beta})} + X_{\theta(\zeta_R^2)(\widetilde{\beta})}$. The isomorphism $\Lambda_\theta \to \Cen^\theta$, defined by Lemma \ref{lem_action_by_automorphisms}, sends a root $\alpha$ to the element $(1-w)\check{\alpha}(\zeta_R)$. We calculate the corresponding action on $Z_{\widetilde{\beta}}$ as
\begin{equation}\label{eqn_second_heis_action}
\Ad((1-w)\check{\alpha}(\zeta_R))(Z_{\widetilde{\beta}}) = \zeta_R^{((1 - w)\check{\alpha}, \beta)}Z_{\widetilde{\beta}} = \langle \alpha, \beta \rangle Z_{\widetilde{\beta}}.
\end{equation} 
The equality of the expressions (\ref{eqn_first_heis_action}) and (\ref{eqn_second_heis_action}) concludes the proof of Theorem \ref{thm_embedding_Heis_in_Gsc}.

\section{A stable grading of $E_8$}\label{sec_stable_grading_of_E8}

In the previous section we constructed a functor from certain Heisenberg groups to $\bbZ / 3 \bbZ$-graded Lie algebras. In order to count points, we need to have a `reference' algebra in which to do explicit calculations. In this section we introduce such an algebra using a principal grading as defined in \cite{Ree12} and give rigidifications of orbits and invariant polynomials using two special transverse slices to nilpotent orbits. The main results of this section, in \S \ref{sec_twisting}, combine this work with the work done in \S \ref{sec_a_functor} to define the map $\eta_f$ described in the introduction (in other words, to construct rational orbits from rational points of Jacobians).

\subsection{Definition of the grading}\label{sec_definition_of_grading}

Let $\intH$ be a split reductive group of type $E_8$ over $\bbZ$. Let $\intT \subset \intH$ be a split maximal torus, and let $\Phi_H \subset X^\ast(\intT)$ be the corresponding set of roots. Let $S_H \subset \Phi_H$ be a fixed choice of root basis.  Let $\Phi_H^+$ be the corresponding set of positive roots. We suppose that $\intH$ comes with a pinning $\{ X_\alpha \}_{\alpha \in S_H}$. 

Let $\check{\rho} \in X_\ast(\intT)$ be the sum of the fundamental coweights with respect to $S_H$, and let $\theta = \check{\rho}|_{\mu_3} : \mu_3 \to \intH$. Let $\intfrh = \Lie(\intH)$. Then $\theta$ defines an action of $\mu_3$ on $\intfrh$ and thus determines a $\bbZ / 3 \bbZ$-grading
\begin{equation} \intfrh = \intfrh(0) \oplus \intfrh(1) \oplus \intfrh(2). 
\end{equation}
We let $\intG = \intH^\theta$, the centralizer of the image of $\theta$ in $\intH$. We write $\intV = \intfrh(1)$; it is a representation of $\intG$, free over $\bbZ$ of rank 84.
\begin{proposition}
The group $\intG$ is a split reductive group isomorphic to $\SL_9 / \mu_3$. The subgroup $\intT \subset \intG$ is a split maximal torus. Over $\bbZ[1/3]$, $\theta$ is a stable $\bbZ / 3 \bbZ$-grading of $\intH$, in the sense of \S \ref{sec_root_lattice_of_type_E8}.
\end{proposition}
\begin{proof}
	It follows from the discussion in \cite[Remark 3.1.5]{Con14} that $\intG$ is smooth over $\bbZ$, and moreover that the connected component $\intG^0$ (which agrees in each fibre $\intG_s$ with the connected component of $\intG_s$) is reductive. Moreover, $\intT \subset \intG^0$ is a split maximal torus. It remains therefore to check that $\intG = \intG^0$ and that its root datum is that of $\SL_9 / \mu_3$. The quotient $\intG / \intG^0$ is \'etale over $\bbZ$, so both of these last points can be checked at the generic point, in which case they follow from the general theory over $\bbC$ (see e.g.\ \cite{Ree10}). The final statement can be checked in geometric fibres, in which case it is \cite[Corollary 14]{Ree12}.
\end{proof}
Let $\Phi_G = \Phi(\intG, \intT)$. There exists a unique choice $S_G$ of root basis for $\intG$ such that $\Phi_G^+ = \Phi_G \cap \Phi_H^+$. 

We write $H$ for the $\bbQ$-fibre of $\intH$, and similarly for $T$, $G$, and $V$. In the coming sections we will describe the invariant theory of the pair $(G, V)$ and its relation to 3-descent on odd genus-2 curves. We will re-introduce integral structures into our discussion in \S \ref{sec_spreading_out} below.

We let $B = V \dquot G = \Spec \bbQ[V]^G$, and write $\pi : V \to B$ for the quotient map. For a detailed summary of the properties of the pair $(G, V)$, and their analogues over fields of sufficiently large positive characteristic, see e.g. \cite{Lev09}. We invite the reader to become familiar at least with the results in the introduction to that paper before proceeding; in particular, we will make frequent use of the existence of Jordan decomposition of elements in $V$ and of the fact that, if $k$ is algebraically closed, then two semisimple elements of $V(k)$ are $G(k)$-conjugate if and only if they have the same image in $B(k)$.

The discriminant of $\frh$ is the image under the Chevalley isomorphism $\bbQ[T]^{W(H, T)} \to \bbQ[\frh]^H$ of the product of all roots $\al \in \Phi_H$.
We write $B^\text{rs} \subset B$ for the open subscheme defined by the non-vanishing of the restriction of this discriminant to $V$. The preimage $\pi^{-1}(B^\text{rs})$ is the open subscheme $V^\text{rs} \subset V$ of regular semisimple elements. We write $V^\text{reg} \subset V$ for the open subscheme of regular elements, i.e.\ those with finite stabilizers in $G$. 
We will generally use the superscripts $(?)^\text{rs}$ and $(?)^\text{reg}$ to denote intersection with these open subschemes of regular semisimple and regular elements, respectively.

\subsection{Kostant section}\label{section_kostant_section}
Let $E = \sum_{\alpha \in S_H} X_\alpha \in \frh$. Then $E$ is a regular nilpotent element.  Let $(E, X, F)$ be the unique normal $\frs\frl_2$-triple containing it. By definition, this means that $(E, X, F)$ is an $\frs\frl_2$-triple with $E \in \frh(1)$, $X \in \frh(0)$, and $F \in \frh(-1)$ (cf. \cite[\S 3.1]{deG11}; the uniqueness follows from \cite[Lemma 8]{deG11}, together with the fact that $Z_G(E)$ is trivial by \cite[Theorem 3.3]{Pan05}). In fact, $(E, X, F)$ is the $\frs\frl_2$-triple naturally associated to the pinning of $H$ by the construction of \cite[Lemma 5.2]{KostantBetti}.

We define an affine linear subspace $\kappa = ( E + \frz_{\frh}(F) ) \cap V\subset V$.
\begin{proposition}
	The restriction of the map $\pi : V \to B$ to $\kappa$ induces an isomorphism $\pi|_\kappa : \kappa \to B$. Moreover, $\kappa$ is contained in the open subscheme $V^\text{reg} \subset V$ of regular elements.
\end{proposition}
\begin{proof}
	See \cite[Theorem 3.5]{Pan05} and \cite[Theorem 8]{KostantPoly}.
\end{proof}
We write $\Kostant : B \to V$ for the inverse $\Kostant = \pi|_\kappa^{-1}$, and call it the Kostant section. We define an action of $\bbG_m$ on $\kappa$ by the formula $t \cdot x = t \Ad(\check{\rho}(t^{-1}))(x)$. This is a contracting action with $E$ as its unique fixed point, and the morphism $\pi|_{\kappa}$ is $\bbG_m$-equivariant (when $\bbG_m$ acts on $B = V \dquot G$ in the natural way, compatibly with its action on $V$ by scalar multiplication). Hence $\sigma$ is also $\bbG_m$-equivariant. 

If $k / \bbQ$ is a field extension, and $f \in B^\text{rs}(k)$, then we can use the Kostant section to organise the set $G(k) \backslash V_f(k)$, where $V_f = \pi^{-1}(f)$. Indeed, we write $\mu_f : G \to V_f$ for the action map $g \mapsto g \cdot \Kostant(f)$. Then $\mu_f$ is a torsor for the group $Z_G(\Kostant(f))$ and determines a bijection
\[ G(k) \backslash V_f(k) \cong \ker(H^1(k, Z_G(\Kostant(f))) \to H^1(k, G)). \]
(We will generalize this statement in Lemma \ref{lem_cohomological_description_of_orbits}.)
The group scheme $Z_G(\Kostant(f))$ can be described explicitly as follows: $\Cen_f := Z_H(\Kostant(f))$ is a maximal torus of $H$. The image of $\theta: \mu \to H$ normalizes $A_f$,  inducing a homomorphism $\mu_3 \to N_H(\Cen_f) / \Cen_f = W(H, \Cen_f)$ that is an elliptic $\mu_3$-action on $X^*(\Cen_f)$ in the sense of \S \ref{sec_root_lattice_of_type_E8}. Thus $Z_G(\Kostant(f)) = \Cen_f^\theta$ is a finite \'etale $k$-group of order $3^4$.

More generally, the centralizer $\Cen := Z_H(\Kostant|_{B^\text{rs}})$ is a maximal torus in $H_{B^\text{rs}}$. We define $\Lambda = X^\ast(\Cen)$ and $\Lambda^\vee = \Hom(\Lambda, \bbZ)$.  
We define a pairing $( \cdot, \cdot ) : \Lambda \times \Lambda \to \bbZ$ by the formula $(\lambda, \mu) = \check\lam(\mu)$. 
Then $\Lambda$ is an \'etale sheaf of $E_8$ root lattices on $B^\text{rs}$. The grading $\theta_{B^\text{rs}}: \mu_3 \to H_{B^\text{rs}}$ determines a homomorphism $\mu_3 \to \Aut(\Lambda)$ that we also denote by $\theta_{B^\text{rs}}$, and which is an elliptic $\mu_3$-action on $\Lambda$. (For ease of notation, we will write $\Lambda_\theta$ for $\Lambda_{\theta_{B^\text{rs}}}$.) The stabilizer scheme $Z_G(\Kostant|_{B^\text{rs}})$ is finite \'etale over $B^\text{rs}$, and can be identified with $\Lambda_{\theta}$ (cf. Lemma \ref{lem_action_by_automorphisms}). Moreover, $\Lambda_\theta$ admits a symplectic, non-degenerate pairing $\langle \cdot, \cdot \rangle : \Lambda_\theta \times \Lambda_\theta \to \mu_3$ (Lemma \ref{lem_reeder_pairing_is_symplectic}).

\begin{proposition}\label{prop_normalizing_invariant_polynomials}
	We can choose polynomials $c_{12}, c_{16}, c_{24}, c_{30} \in \bbQ[V]^G$ with the following properties:
	\begin{enumerate}
		\item Each polynomial $c_i$ is homogeneous of degree $i$, and $\bbQ[V]^G = \bbQ[c_{12}, c_{16}, c_{24}, c_{30}]$. Consequently, there is an isomorphism $B \cong \bbA_\bbQ^4$. If $\Delta_0 \in \bbQ[V]^G$ denotes the discriminant of the polynomial $f(x) = x^5 + c_{12} x^3 + c_{18} x^2 + c_{24} x + c_{30}$, then $\Delta_0^2$ is (up to scalar) equal to the restriction to $V$ of the Lie algebra discriminant of $\frh$ (as defined in the previous section). 
		\item Let $\Curve^0 \to B$ be the family of affine curves given by the equation 
		\begin{equation} \Curve^0 : y^2 = x^5 + c_{12} x^3 + c_{18} x^2 + c_{24} x + c_{30}, \end{equation} and let $\Curve \to B$ be its completion inside weighted projective space $\bbP_B(1, 1, 3)$, projective over $B$. Let $\Jacobian \to B^\text{rs}$ be the Jacobian of its smooth part. Then there is an isomorphism $\Lambda_\theta \cong \Jacobian[3]$ of \'etale sheaves that sends the pairing $\langle \cdot, \cdot \rangle$ on $\Lambda_\theta$ to the Weil pairing on $\Jacobian[3]$. 
	\end{enumerate}
\end{proposition}
The proof of this proposition will be given in \S \ref{sec_proof_of_prop_on_invariants}. 

Let $\Point : B \to \Curve$ denote the section at infinity (which is a Weierstrass point in each smooth fibre of $\Curve$). The choice of $\Point$ determines a symmetric line bundle $\prin =  \cO_{\Jacobian}(\Curve - \Point)$ on $\Jacobian$. We write $\tprin = \prin^{\otimes 3}$, and define $\extJ$ to be the 3-torsion subgroup of the Mumford theta group $\mathscr{G}(\tprin)$. Thus $\extJ$ is a central extension
\[ 1 \to \mu_3 \to \extJ \to \Jacobian[3] \to 1 \]
of \'etale group schemes over $B^\text{rs}$, and $(\Lambda, \theta_{B^\text{rs}}, \extJ)$ is an object of the category $\operatorname{Heis}_{B^\text{rs}}$ defined in \S \ref{sec_a_functor}. We will soon show (Proposition \ref{prop_Jacobian_and_Kostant_section}) that the image of $(\Lambda, \theta_{B^\text{rs}}, \extJ)$ under the functor defined in Theorem \ref{thm_equivalence_of_Heis_with_GrLieT} is isomorphic in the groupoid $\operatorname{GrLieT}_{B^\text{rs}}$ to the triple $(H_{B^\text{rs}}, \theta_{B^\text{rs}}, A)$.

\subsection{Twisting}\label{sec_twisting}

We can now explain our construction of orbits. We recall that in \S \ref{sec_a_functor} we have defined stacks $\operatorname{Heis}$ and $\operatorname{GrLieT}$ over $\operatorname{Alg}_{\bbQ}^\text{op}$ in the \'etale topology (in fact, over $\operatorname{Alg}_{\bbZ[1 / N]}^\text{op}$), and a morphism $\operatorname{Heis} \to \operatorname{GrLieT}$. We now define some related stacks over $\operatorname{Alg}_{\bbQ}^\text{op}$. 

We define $\operatorname{GrLie}$ to be the stack of pairs $(H', \theta')$, where $H'$ is a reductive group over a $\bbQ$-algebra $R$ of type $E_8$ and $\theta' : \mu_3 \to H'$ is a stable $\bbZ / 3 \bbZ$-grading. Morphisms $(H', \theta') \to (H'', \theta'')$ are given by isomorphisms $H' \to H''$ intertwining $\theta'$ and $\theta''$.
\begin{lemma}\label{lem_forms_of_H_theta}
	The following sets are in canonical bijection:
	\begin{enumerate}
		\item The set of isomorphism classes of objects in $\operatorname{GrLie}_R$.
		\item The set $H^1(R, G)$.
	\end{enumerate}
\end{lemma}
\begin{proof}
Note that $\operatorname{GrLie}_R$ always contains the object $(H_R, \theta_R)$, where $H$ and $\theta$ are as defined in \S \ref{sec_definition_of_grading}. We have proved (Lemma \ref{lem_graded_groups_etale_locally_isomorphic}) that any two objects of $\operatorname{GrLie}_R$ are isomorphic \'etale locally on $\Spec R$. Since $\Aut(H, \theta) = G$, the result follows by descent.
\end{proof}
We define $\operatorname{GrLieE}$ to be the stack of tuples $(H', \theta', \gamma')$, where $(H', \theta') \in \operatorname{GrLie}_R$ and $\gamma' \in \frh'(1)$. Morphisms $(H', \theta', \gamma') \to (H'', \theta'', \gamma'')$ in $\operatorname{GrLie}_R$ are given by isomorphisms $H' \to H''$ intertwining $\theta'$ and $\theta''$ and such that the induced map $\frh' \to \frh''$ sends $\gamma'$ to $\gamma''$. 

There is an obvious morphism $V \to \operatorname{GrLieE}$ given by $\gamma \in V(R) \mapsto (H_R, \theta_R, \gamma)$. We may extend the map $\pi: V \to B$ to a morphism $\operatorname{GrLieE} \to B$, also denoted by $\pi$, as follows: 
Let $(H', \theta', \gamma') \in \operatorname{GrLieE}_R$. After passing to a faithfully flat extension $R \to R'$, we can find an isomorphism $\alpha : (H'_{R'}, \theta'_{R'}) \cong (H_{R'}, \theta_{R'})$. Then $\pi(\gamma'):= \pi(\alpha(\gamma')) \in B(R')$ in fact lies in $B(R)$ and is independent of the choice of $\alpha$. 

%If $(H', \theta', \gamma') \in \operatorname{GrLieE}_R$, then we define an element $\pi(\gamma') \in B(R)$ as follows: after passage to a faithfully flat extension $R \to R'$, we can find an isomorphism $\alpha : (H'_{R'}, \theta'_{R'}) \cong (H_{R'}, \theta_{R'})$, and $\pi(\alpha(\gamma')) \in B(R')$ in fact lies in $B(R)$ and is independent of the choice of $\alpha$. Thus there is a functor $\pi : \operatorname{GrLieE}_R \to B(R)$, compatible with arbitrary base change on $R$. (We view the set $B(R)$ as a discrete category, i.e.\ as a category with no non-identity morphisms.)

If $f \in B(R)$, then we define $\operatorname{GrLieE}_{f, R}$ to be the full subcategory of $\operatorname{GrLieE}_R$ consisting of tuples $(H', \theta', \gamma')$ where $\pi(\gamma') = f$. 
\begin{lemma}\label{lem_cohomological_description_of_orbits}
	Let $R$ be a $\bbQ$-algebra, and let $f \in B^\text{rs}(R)$. Then any two objects of $\operatorname{GrLieE}_{f, R}$ are isomorphic \'etale locally on $\Spec R$. Consequently, the following sets are in canonical bijection:
	\begin{enumerate}
		\item The set of $G(R)$-orbits in $V_f(R)$.
		\item The set $\ker(H^1(R, Z_G(\Kostant(f))) \to H^1(R, G))$.
		\item The set of isomorphism classes of objects $(H', \theta', \gamma') \in \operatorname{GrLieE}_{f, R}$ such that $(H', \theta') \cong (H_R, \theta_R)$ in $\operatorname{GrLie}_R$. 
	\end{enumerate}
\end{lemma}
\begin{proof}
We first show that any two objects of $\operatorname{GrLieE}_{f, R}$ are isomorphic \'etale locally on $\Spec R$. It suffices to show that any object $(H', \theta', \gamma') \in \operatorname{GrLieE}_{f, R}$ is isomorphic \'etale locally to $(H_R, \theta_R, \Kostant(f))$. By Lemma \ref{lem_graded_groups_etale_locally_isomorphic}, we can assume that $(H', \theta') = (H_R, \theta_R)$, so that $\gamma' \in V_f(R)$. The group scheme $Z_G(\Kostant(f))$ is a finite \'etale $R$-scheme, and the action map $G \to V_f$ of $\Kostant(f)$ is a $Z_G(\Kostant(f))$-torsor. This shows that $\gamma', \Kostant(f)$ become conjugate after a finite \'etale ring extension $R \to R'$, and therefore that the triples $(H_R, \theta_R, \Kostant(f))$ and $(H', \theta', \gamma')$ become isomorphic after passage to $R'$.

The existence of the bijection between the first and second sets is a consequence of e.g.\ \cite[Exercise 2.4.11]{Con14}. The category $\operatorname{GrLieE}_{f, R}$ contains the triple $(H_R, \theta_R, \Kostant(f))$. The automorphisms of this triple may be identified with the sections over $R$ of $Z_G(\Kostant(f))$. Moreover, any two objects of $\operatorname{GrLieE}_{f, R}$ are isomorphic \'etale locally on $\Spec R$. This implies the existence of the bijection between the second and third sets.
\end{proof}
If $f \in B^\text{rs}(R)$, then we define $\operatorname{Heis}_{f}$ to be the subcategory of $\operatorname{Heis}_R$ whose objects are triples of the form $(f^\ast \Lambda, \theta_R, \mathscr{H})$ and whose morphisms $(f^\ast\Lambda, \theta_R, \mathscr{H}) \to (f^\ast\Lambda, \theta_R, \mathscr{H}')$ are the morphisms in $\operatorname{Heis}_{R}$ that restrict to the identity on $f^\ast \Lambda$. (Recall that $\Lambda = X^\ast(\Cen)$ is an \'etale sheaf of $E_8$ root lattices on $B^\text{rs}$, so $f^\ast \Lambda$ is an \'etale sheaf of $E_8$ root lattices on $\Spec R$.)

Let $f_\tau \in B^\text{rs}(B^\text{rs})$ be the tautological section. As stated in \S \ref{section_kostant_section}, 
%Proposition \ref{prop_normalizing_invariant_polynomials}, 
$(\Lambda, \theta_{B^\text{rs}}, \extJ)$ defines an element of the groupoid $\operatorname{Heis}_{B^\text{rs}}$ (in fact, $(\Lambda, \theta_{B^\text{rs}}, \extJ) \in \operatorname{Heis}_{f_\tau, B^\text{rs}}$). Let $(H_\tau, \theta_\tau, A_\tau) \in \operatorname{GrLieT}_{B^\text{rs}}$ be its image under the functor of Theorem \ref{thm_equivalence_of_Heis_with_GrLieT}. By definition we have $A_\tau = A$. The Lie algebra $\fra_\tau$ has a tautological section $\gamma_\tau$ (which is in fact none other than $\Kostant$). Thus $(H_\tau, \theta_\tau, \gamma_\tau) \in \operatorname{GrLieE}_{f_\tau, B^\text{rs}}$.
\begin{proposition}\label{prop_Jacobian_and_Kostant_section}
	The objects $(H_\tau, \theta_\tau, \gamma_\tau)$ and $(H_{B^\text{rs}}, \theta_{B^\text{rs}}, \Kostant(f_\tau))$ of $\operatorname{GrLieE}_{ f_\tau, B^\text{rs}}$ are isomorphic.
\end{proposition}
\begin{proof}
	The proof relies upon the fact that for each triple, the associated $\bbZ/3\bbZ$-grading can be naturally extended to a $\bbZ / 6 \bbZ$-grading. The structure of $\bbZ / 6 \bbZ$-grading plays the role of a rigidification, in the presence of which rational orbits and geometric orbits coincide.
	
	A $\bbZ/6\bbZ$-grading of $\frh_{B^{\text{rs}}}$ extending the grading of $(H_{B^\text{rs}}, \theta_{B^\text{rs}}, \Kostant(f_\tau))$ is given by $\check{\rho}|_{\mu_6}$. Let $V' = \frh_{B^\text{rs}}(\check\rho |_{\mu_6}, 1)$, 
%denote the 1-part of this $\bbZ / 6 \bbZ$-grading, 
and note that $\Kostant(f_\tau) \in V'(B^\text{rs})$. To define a $\bbZ / 6 \bbZ$-grading of $\frh_\tau$ extending the grading of $(H_\tau, \theta_\tau, \gamma_\tau)$, observe that $\prin$ is a symmetric line bundle. A choice of isomorphism $\prin \cong [-1]^\ast \prin$ determines an automorphism $(\omega, \alpha) \mapsto (-\omega, [-1]^\ast \alpha)$ of $\extJ$ that restricts to the identity on the central $\mu_3$ and such that the induced map on the quotient $\Jacobian[3]$ is multiplication by $-1$. 
%Since this automorphism is compatible with the action of $-1$ on $\Lambda$, 
Since taking $-1$ on $\Lambda$ with this automorphism on $\extJ$ gives a automorphism in the category $\operatorname{Heis}_{B^\text{rs}}$,
Theorem \ref{thm_equivalence_of_Heis_with_GrLieT} implies the existence of an involution $\iota: H_\tau \to H_\tau$ that restricts to the identity on the image of $\theta_\tau$, that restricts to an isomorphism of $Z_{H_\tau}(\gamma_\tau)$, and that induces the map $-1$ on the character group of this torus. Let $\theta': \mu_2 \to H_\tau$ be defined by mapping $-1$ to $d\iota \in \Aut(\frh_\tau) = H_\tau$. Then $\theta' \cdot \theta_\tau$ defines a $\bbZ / 6 \bbZ$-grading of $\frh_\tau$ such that $\gamma_\tau \in \frh_\tau(\theta'\cdot \theta, 1)$. 
%if $V'_\tau \subset \frh_\tau$ is the 1-part of the grading, then $\gamma_\tau \in V'_\tau$. 
	
	To complete the proof, consider the $B^\text{rs}$-scheme $\mathscr{T}$ of isomorphisms $H \to H_\tau$ intertwining $\check{\rho}|_{\mu_6}$ and $\theta' \cdot \theta_\tau$ and such that the induced map $\frh \to \frh_\tau$ sends $\Kostant(f_\tau)$ to $\gamma_\tau$. Then $\mathscr{T}$ is an \'etale $B^\text{rs}$-scheme, which is in fact a torsor for $Z_H(\Kostant(f_\tau))^{\check{\rho}|_{\mu_6}}$ (the argument is the same as in the proof of Lemma \ref{lem_graded_groups_etale_locally_isomorphic}). Since $Z_H(\Kostant(f_\tau))^{\check{\rho}|_{\mu_6}} \to B^\text{rs}$ is the trivial group scheme, we have $\mathscr{T} = B^\text{rs}$ and it follows that there is a unique isomorphism $(H_\tau, \theta_\tau, \gamma_\tau) \cong (H_{B^\text{rs}}, \theta_{B^\text{rs}}, \Kostant(f_\tau))$ in $\operatorname{GrLieE}_{B^\text{rs}, f_\tau}$ that is compatible with the given $\bbZ / 6 \bbZ$-gradings.
\end{proof}
\begin{theorem}\label{thm_heisenberg_grlieE_equivalence}
	Let $f \in B^\text{rs}(R)$. Then there is an equivalence of categories $\operatorname{Heis}_{f, R} \to \operatorname{GrLieE}_{f, R}$.
\end{theorem}
\begin{proof}
Let $(f^\ast \Lambda, \theta_R, \mathscr{H}) \in \operatorname{Heis}_{f, R}$, and let $(H', \theta', A') \in \operatorname{GrLieT}_{R}$ be its image under the functor of Theorem \ref{thm_equivalence_of_Heis_with_GrLieT}. Since $A' = f^\ast A_\tau$, we have $f^\ast(\gamma_\tau) \in \frh'(\theta', 1)$. We define a functor $\operatorname{Heis}_{f, R} \to \operatorname{GrLieE}_{f, R}$ by sending $(f^\ast \Lambda, \theta_R, \mathscr{H})$ to the triple $(H', \theta', f^\ast(\gamma_\tau))$.
It is fully faithful, by Lemma \ref{lem_action_by_automorphisms}. The category $\operatorname{Heis}_{f, R}$ contains the object $f^\ast (\Lambda, \theta_{B^\text{rs}}, \extJ) = (f^\ast\Lambda, \theta_R, f^\ast\extJ)$, 
whose image in $\operatorname{Heis}_{f, R}$ under the functor just defined is
%which corresponds to the object 
$(H_R, \theta_R, f^\ast (\gamma_\tau))$ by Proposition \ref{prop_Jacobian_and_Kostant_section}. The objects of both categories are therefore classified by the group $H^1(R, \Jacobian_f[3]) = H^1(R, Z_G(\Kostant(f)))$. This shows that our functor is essentially surjective, and completes the proof of the lemma. 
\end{proof}

\begin{corollary}\label{cor_construction_of_orbits_over_general_base}
	Let $R$ be a $\bbQ$-algebra over which every locally free module of finite rank is free. Let $f \in B^\text{rs}(R)$. Then there is a canonical injection $\eta_{f} : \Jacobian_f(R) / 3 \Jacobian_f(R) \to G(R) \backslash V_f(R)$.
\end{corollary}
\begin{proof}
	The group $\extJ$ acts on $H^0(\Jacobian_f, \tprin)$, which is a locally free $R$-module of rank 9. By Proposition \ref{prop_Jacobian_and_Kostant_section} and Theorem \ref{thm_embedding_Heis_in_Gsc},  there is a diagram of $R$-groups with exact rows:
	\[ \xymatrix{ 1 \ar[r]  & \mu_3 \ar[r] & G^\text{sc}_R \ar[r] & G_R \ar[r] & 1 \\
		1 \ar[r] & \mu_3 \ar[r]\ar[u]^= & \extJ \ar[r]\ar[u] & \Jacobian_f[3] \ar[r]\ar[u] & 1.  }\]
	Let $P \in \Jacobian_f(R)$. Let $\tprin_P = t_P^\ast \prin \otimes \prin \otimes \prin$, and let $\extJ_P$ denote the 3-torsion subgroup of the Mumford theta group $\mathscr{G}(\tprin_P)$. Then $(f^\ast\Lambda, \theta_R, \extJ_P) \in \operatorname{Heis}_{f, R}$. Let $(H_P, \theta_P, \gamma_P) \in \operatorname{GrLieE}_{f, R}$ denote the tuple corresponding to $(f^\ast\Lambda, \theta_R, \extJ_P)$ under the equivalence of Theorem \ref{thm_heisenberg_grlieE_equivalence}. Then the class $\varphi \in H^1(R, \Jacobian_f[3])$ corresponding to $(H_P, \theta_P, \gamma_P)$ under the bijection of Lemma \ref{lem_cohomological_description_of_orbits} is the Kummer class of the point $P$ (as follows from Lemma \ref{lem_action_by_automorphisms}).
	
	To prove the corollary, we will show that this class lifts to $H^1(R, \extJ)$. This will imply that the image of $\varphi$ in $H^1(R, G)$ lies in the image of the map $H^1(R, G^\text{sc}) \to H^1(R, G)$, which is trivial (by our assumption on $R$). To show that the class lifts, it will even suffice to show that it lifts to $H^1(R, \mathscr{G}(\tprin))$, where $\mathscr{G}(\tprin)$ is the Mumford theta group of $\tprin$, sitting in the short exact sequence of $R$-groups
	\[ \xymatrix@1{ 1 \ar[r] & \bbG_m \ar[r] & \mathscr{G}(\tprin) \ar[r] & \Jacobian_f[3] \ar[r] & 1. } \]
	Indeed, the map $H^2(R, \mu_3) \to H^2(R, \bbG_m)$ is injective, again by our assumption on $R$.
	We define $\mathscr{T}_P$ to be the scheme of pairs $(\omega, \alpha)$, where $\omega \in \Jacobian_f$ and $\alpha : \tprin_P \to t_\omega^\ast \tprin$ is an isomorphism. Note that forgetting $\omega$ leads to a surjective map $\mathscr{T}_P \to [3]^{-1}(P)$. Thus $\mathscr{T}_P$ is a torsor for $\mathscr{G}(\tprin)$, defining a class in $H^1(R, \mathscr{G}(\tprin))$ that lifts the class $\varphi \in H^1(R, \Jacobian_f[3])$. This completes the proof of the corollary.
\end{proof}
\begin{remark}
	A stronger version of Corollary \ref{cor_construction_of_orbits_over_general_base} is true: we may replace our assumption that every locally free module is free with the assumption that $H^1(R, G^\text{sc})$ is trivial, by refining the torsor for $\mathscr{G}(\tprin)$ constructed in the proof to a torsor for $\extJ$ using the canonical isomorphism $\prin^{\otimes 9} \cong [3]^\ast \prin$ afforded by the theorem of the cube. However, since we don't need this extra generality we have chosen not to include the details here.
\end{remark}
We restate the corollary in the case that $R = k$ is a field extension of $\bbQ$.
\begin{corollary}\label{cor_construction_of_orbits_over_field}
	Let $k / \bbQ$ be a field, and let $f \in B^\text{rs}(k)$. Then there is a canonical injection $\eta_{f} : \Jacobian_f(k) / 3 \Jacobian_f(k) \to G(k) \backslash V_f(k)$.
\end{corollary}
We also record for future use the fact that when $R = \bbQ$, we can extend the above construction of orbits from rational points to 3-Selmer elements:
\begin{proposition}\label{prop_existence_of_selmer_orbits}
Let $f \in B^\text{rs}(\bbQ)$. Then the map $\eta_f : \Jacobian_f(\bbQ) / 3 \Jacobian_f(\bbQ) \to G(\bbQ) \backslash V_f(\bbQ)$ naturally extends to a map $\eta_f : \Sel_3(\Jacobian_f) \to G(\bbQ) \backslash V_f(\bbQ)$.
\end{proposition}
\begin{proof}
	Taking into account Lemma \ref{lem_cohomological_description_of_orbits}, we just need to show that the map $H^1(\bbQ, G) \to \prod_{v} H^1(\bbQ_v, G)$ has trivial kernel, where the product runs over the set of all places $v$ of $\bbQ$. This is an exercise using class field theory.
\end{proof}

\subsection{Proof of Proposition \ref{prop_normalizing_invariant_polynomials}}\label{sec_proof_of_prop_on_invariants}

We now prove Proposition \ref{prop_normalizing_invariant_polynomials}. We will use a special transverse slice to the orbit of a subregular nilpotent element in $V$ in order to form the bridge between the group $H$ and the family of abelian surfaces $\Jacobian$. 

We begin by fixing a subregular nilpotent element $e_0 \in V$ (the existence of such an element can be read off from the tables in \cite{Vin78}, which also show that there is a unique $G$-orbit of subregular nilpotents in $V$). We can complete $e_0$ to a normal $\frs\frl_2$-triple $(e_0, h_0, f_0)$ in $\frh$. We define $\OpenSurface_0 = e_0 + \frz_\frh(f_0)$, an affine linear subspace of $\frh$. We define a $\mu_3$-action on $\OpenSurface_0$ by the formula $\zeta \cdot x = \zeta^{-1} \Ad(\theta(\zeta))(x)$. We define a $\bbG_m$-action on $\OpenSurface_0$ by the formula $t \cdot x = t^2 \Ad(\lambda(t^{-1}))(x)$, where $\lambda : \bbG_m \to G$ is the cocharacter with $d \lambda(1) = h_0$. These two actions commute, giving a $\mu_3 \times \bbG_m$-action on $\OpenSurface_0$. Let $B_0 = \frh \dquot H$, and let $p_0 : \OpenSurface_0 \to B_0$ denote the restriction of the adjoint quotient $\pi_0 : \frh \to \frh \dquot H$ to $\OpenSurface_0$. Define an action of $\bbG_m \times \mu_3$ on $\bbQ[\frh]^H$ by setting $(\zeta, t)\cdot f = \zeta^{-d}t^{2d}f$ if $f$ is a homogeneous polynomial of degree $d$. With the induced action on $B_0$, the map $p_0$ is equivariant for the action of $\mu_3 \times \bbG_m$ on source and target. We identify $X^\ast(\mu_3 \times \bbG_m) = \bbZ / 3 \bbZ \times \bbZ$. If $v$ is an eigenvector for an action of $\mu_3 \times \bbG_m$, then we define its weight to be the image in $\bbZ / 3 \bbZ \times \bbZ$ of the character by which $\mu_3 \times \bbG_m$ acts on $v$.
\begin{proposition}\label{prop_slodowy_slice}
	We can choose polynomials $c_2, c_8, c_{12}, c_{14}, c_{18}, c_{20}, c_{24}, c_{30} \in \bbQ[\frh]^H$ and $x, y, z \in \bbQ[\OpenSurface_0]$ with the following properties:
	\begin{enumerate}
		\item Each polynomial $c_i$ is homogeneous of degree $i$. The polynomials $c_i$ are algebraically independent and generate $\bbQ[\frh]^H$. The restriction of $c_i$ to $\OpenSurface_0$ has weight $(-i,  2i)$. The elements $x$, $y$, $z$ have weights $(0, 12)$, $(0, 30)$ and $(-1, 20)$, respectively.
		\item The restrictions of the 7 elements $c_2, c_8, c_{12}, c_{14}, c_{18}, c_{20}, c_{24}$ to $\OpenSurface_0$, together with $x, y, z$, form an algebraically independent set and generate $\bbQ[\OpenSurface_0]$. Moreover, the morphism $p_0 : \OpenSurface_0 \to B_0$ is given by the formula
		\[ y^2 = z^3 + x^5 + z (c_2 x^3 + c_8 x^2 + c_{14} x + c_{20}) + (c_{12} x^3 + c_{18} x^2 + c_{24} x + c_{30}). \]
	\end{enumerate}
\end{proposition}
\begin{proof}
View $\OpenSurface_0$ as a vector space with origin $e_0$; then the action of $\mu_3 \times \bbG_m$ on $\OpenSurface_0$ is linear. By direct calculation, the weights of $\mu_3 \times \bbG_m$ in $\OpenSurface_0$ are as follows:
\[  (1, 4), (0, 12), (1, 16), (-1, 20), (0, 24), (1, 28), (0, 30), (0, 36), (1, 40),   (0, 48). \]
The weights of $\mu_3 \times \bbG_m$ on $\frh \dquot H$ are as follows:
\[ (1, 4), (1, 16), (0, 24), (1, 28), (0, 36), (1, 40), (0, 48), (0, 60). \]
By comparison with the results of \cite[\S 8.7]{Slo80}, we see that the differential $dp_{0, e_0}$ has rank 7, mapping the subspace where $\bbG_m$ acts with weights $4, 16, 24, 28, 36, 40, 48$ isomorphically into the Zariski tangent space $T_0(\frh \dquot H)$ and annihilating the subspace where $\bbG_m$ acts with weights $12, 20$, and $30$. 

Following through the argument of \cite[\S 8.7, Theorem]{Slo80} with this $(\mu_3 \times \bbG_m)$-action now shows that there is a $(\mu_3 \times \bbG_m)$-equivariant isomorphism $(\OpenSurface_0 \to B_0) \to (\OpenSurface'_0 \to B'_0)$, where $\OpenSurface_0' \to B_0'$ is the semi-universal $(\mu_3 \times \bbG_m)$-deformation of the singularity $y^2 = z^3 + x^5$ given by the formula
\[ y^2 = z^3 + x^5 + z (c_2 x^3 + c_8 x^2 + c_{14} x + c_{20}) + (c_{12} x^3 + c_{18} x^2 + c_{24} x + c_{30}), \]
where $x, y, z$ have weights $(0, 12)$, $(0, 30)$ and $(-1, 20)$, respectively, and each $c_i$ has weight $(-i, 2i)$. We fix our invariant polynomials $c_2, \dots, c_{30} \in \bbQ[\frh]^H$ to be the images under this isomorphism of the elements with the same names in the affine ring of $B_0'$. This completes the proof of the proposition.
\end{proof}
We fix elements $c_2, c_8, c_{12}, c_{14}, c_{18}, c_{20}, c_{24}, c_{30}$ and $x, y, z$ as in Proposition \ref{prop_slodowy_slice}. Thus we have identified $\OpenSurface_0$ explicitly as given by the equation
\begin{equation}\label{eqn_weierstrass_equation} y^2 = z^3 + x^5 + z (c_2 x^3 + c_8 x^2 + c_{14} x + c_{20}) + (c_{12} x^3 + c_{18} x^2 + c_{24} x + c_{30}).  \end{equation}
We view (\ref{eqn_weierstrass_equation}) as an affine Weierstrass equation over $\bbA^1_{B_0}$. This allows us to compactify $\OpenSurface_0$ to obtain a projective Weierstrass fibration (in the sense of \cite{Mir81}) $\Surface_0 \to \bbP^1_{B_0}$ which contains $\OpenSurface_0$ as an open subscheme. More precisely, $\OpenSurface_0$ is the complement in $\Surface_0$ of the zero section $\OO$ and the fibre $\FF$ above the point $x = \infty$ of $\bbP^1_{B_0}$. 

Let $\kappa_0 = E + \frz_\frh(F)$, and let $\Kostant_0 = \pi_0|_{\kappa_0}^{-1} : B_0 \to \frh$ denote the Kostant section for $\frh$. (Thus $\kappa = \kappa_0 \cap V$.) Let $\Cen_0 =  Z_H(\Kostant_0|_{B_0^\text{rs}})$, and let $\Lambda_0 =  X^\ast(\Cen_0)$. Then $\Lambda_0$ is an \'etale sheaf of $E_8$ root lattices on $B_0^\text{rs}$.
Observe that there is a $\Cen_0$-torsor $\Torsor_0 \to \OpenSurface_0^\text{rs}$ given by the formula
\[ \Torsor_0 = \{ (h, x) \in H \times B_0^\text{rs} \mid h \cdot \Kostant_0(x) \in \OpenSurface^\text{rs} \}. \]

Let $\eta_0$ be the generic point of $B_0$, and let $\overline{\eta}_0$ be a geometric point above it.
The existence of $\Torsor_0$ determines a $\pi_1(\eta_0, \overline{\eta}_0)$-equivariant map $X^\ast(\Cen_{0,\overline{\eta}_0}) \to \Pic(\OpenSurface_{0,\overline{\eta}_0})$. (Note that this \'etale fundamental group can be identified with $\Gal(k(\overline{\eta}_0) / k(\eta_0))$.) We endow $\Pic(\OpenSurface_{0,\overline{\eta}_0})$ with an intersection pairing as follows. There is a perfect intersection pairing on $\Pic(\Surface_{0,\overline{\eta}_0})$, which is a free $\bbZ$-module of rank 10 ($\Surface_{0,\overline{\eta}_0}$ is a rational elliptic surface, cf. \cite[\S 8]{Sch10}). Let $\Lattice_0 = \langle \OO, \FF \rangle \subset \Pic(\Surface_{0,\overline{\eta}_0})$ be the free $\bbZ$-module generated by $\OO$ and $\FF$, and let $\Lattice_0^\perp$ denote its orthogonal complement. Then $\Pic(\Surface_{0,\overline{\eta}_0}) = \Lattice_0 \oplus \Lattice_0^\perp$, and so the morphism $\Lattice_0^\perp \to \Pic(\Surface_{0,\overline{\eta}_0}) \to \Pic(\OpenSurface_{0,\overline{\eta}_0})$ induced by the open immersion $\OpenSurface_{0,\overline{\eta}_0} \to \Surface_{0,\overline{\eta}_0}$ is an isomorphism. We give $\Pic(\OpenSurface_{0,\overline{\eta}_0})$ the perfect, negative-definite pairing induced from that of $\Lattice_0^\perp$. 
\begin{lemma}\label{lem_isomorphism_of_lattices}
	The $\pi_1(\eta_0, \overline{\eta}_0)$-equivariant morphism $X^\ast(\Cen_{0, \overline{\eta}_0}) \to \Pic(\OpenSurface_{0, \overline{\eta}_0}) \cong \Lattice_0^\perp$ just constructed is an isomorphism that intertwines the pairing $( \cdot, \cdot ) : \Lambda_0 \times \Lambda_0 \to \bbZ$ with minus the intersection pairing on $\Pic(\OpenSurface_{0, \overline{\eta}_0}) \cong \Lattice_0^\perp$.
\end{lemma}
\begin{proof}
	To show that this morphism is an isomorphism, we use the existence of the Springer resolution. Recall that $T \subset H$ is a maximal torus with Lie algebra $\frt$ and root basis $S_H \subset \Phi(H, T)$. We write $P \subset H$ for the Borel subgroup corresponding to this choice of root basis. Let $\OpenSurface_{0, \frt}$ denote the pullback of $\OpenSurface_0 \to B_0$ along the finite map $\frt \to B_0 = \frt \dquot W(H, T)$, and define
	\[ \widetilde{\OpenSurface}_{0, \frt} = \{ (h P, x) \in H / P \times \OpenSurface_{0, \frt} \mid x \in \Ad(h)(\Lie P) \}. \]
	Then $\widetilde{\OpenSurface}_{0, \frt} \to \frt$ is the Springer resolution of the transverse slice $\OpenSurface_0$: it is smooth, and the natural morphism $\widetilde{\OpenSurface}_{0, \frt} \to \OpenSurface_{0,\frt}$ is proper and is an isomorphism away from the singular points in each fibre of $\OpenSurface_{0, \frt} \to \frt$ (cf. \cite[\S 5.3]{Slo80}). In particular, we can glue $\widetilde{\OpenSurface}_{0,\frt}$ with $\Surface_{0, \frt}$ to obtain a smooth proper surface $\widetilde{\Surface}_{0, \frt} \to \frt$ which is itself a resolution of $\Surface_{0, \frt} \to \frt$.
	
	We have $\Pic(H / P) = X^\ast(T)$. The projection $\widetilde{\OpenSurface}_{0, \frt} \to H / P$ therefore induces a map $X^\ast(T) \to \Pic(\widetilde{\OpenSurface}_{0,\frt})$. Let $\xi$ denote the generic point of $\frt$, and let $\overline{\xi}$ denote a lift of $\overline{\eta}_0$ to a geometric point above $\xi$. We claim that the composite 
	\[ X^\ast(T) \to \Pic(\widetilde{\OpenSurface}_{0, \frt}) \to \Pic(\widetilde{\OpenSurface}_{0, \frt, \overline{\xi}}) \]
	is an isomorphism. In fact, it suffices to prove the analogous claim above the central point $0$ of $\frt$, as we now explain. In order to avoid confusing notation, let us write $s$ for $0 \in \frt$ and $\overline{s}$ for a geometric point above it. Then there is a commutative diagram
	\[ \xymatrix{ X^\ast(T) \ar[d]^= \ar[r] & \Pic(\widetilde{\OpenSurface}_{0, \frt, \overline{\xi}}) \ar[d] \\
		X^\ast(T)  \ar[r] & \Pic(\widetilde{\OpenSurface}_{0, \frt, \overline{s}}), } \]
	where the right vertical map is given by specialisation and is injective with torsion-free cokernel (cf. \cite[Proposition 3.6]{Mau12}, which gives a specialisation morphism for N\'eron--Severi groups of the fibres of the proper morphism $\widetilde{\Surface}_{0, \frt} \to \frt$; we are using that the fibres of $\widetilde{\Surface}_{0, \frt} \to \frt$ are rational elliptic surfaces, and that in each geometric fibre the free rank-2 subgroup $\langle \OO, \FF \rangle$ of the Picard group splits off as a direct summand in a way that is compatible with specialisation). Since the source and target of the right vertical map are both free $\bbZ$-modules of rank 8, this map is an isomorphism. It follows that if the bottom map in the above diagram is an isomorphism, then so is the top map. 
	
	Here we can use the explicit description of the exceptional fibre of the morphism $\widetilde{\OpenSurface}_{0,\frt, \overline{s}} \to \OpenSurface_{0, \frt, \overline{s}}$ given in \cite{Hin91}. This exceptional divisor is a union of projective lines $C_\alpha$ indexed by simple roots $\alpha \in S_H$, the classes of which freely generate $\Pic(\widetilde{\OpenSurface}_{0, \frt, \overline{s}})$; and if $\alpha \in S_H$ is a simple root, then the image of $\alpha \in X^\ast(T)$ in the Picard group is the class of the curve $C_\alpha$ (this is \cite[5.3, Lemma]{Hin91}). 
	
	The same argument shows that the morphism $X^\ast(T) \to \Pic(\widetilde{\OpenSurface}_{0, \frt, \overline{\xi}}) \cong \Lattice_0^\perp$ intertwines the pairing on the root lattice $X^\ast(T)$ with the negative of the intersection pairing on $\Lattice_0^\perp$. Indeed, this can be checked in the central fibre of the Springer resolution, where it follows from the fact that the intersection paring between $C_\al$ and $C_\be$ is given by $-(\al, \be)$, itself a consequence of \cite[Proposition 5.2]{Hin91}. 
	
	It remains to show why the above results on the map $X^\ast(T) \to \Pic(\widetilde{\OpenSurface}_{0, \frt, \overline{\xi}})$ imply the desired properties of the map $X^\ast(\Cen_{\overline{\eta}_0}) \to \Pic(\OpenSurface_{0, \overline{\eta}_0})$. Since the points $\overline{\xi}$ and $\sigma_0(\overline{\eta}_0)$ of $\frh(k(\overline{\eta}_0))$ are conjugate under the action of $H(k(\overline{\eta}_0))$, what we really need to check is that the two maps $X^\ast(T) \to \Pic(\OpenSurface_{0, \frt, \overline{\xi}})$, one arising by pullback from $\Pic(H / P)$, and the other arising from the existence of the $T_{k(\overline{\eta}_0)}$-torsor
	\[ \Torsor' = \{ h \in H_{k(\overline{\eta}_0)} \mid \Ad(h)(\xi) \in \OpenSurface_{0, k(\overline{\eta}_0)} \}, \]
	are the same. This follows from the definitions. 
\end{proof}
As noted above, the morphism $p_0: \OpenSurface_0 \to B_0$ is $\mu_3$-equivariant, where $\mu_3$ acts on $B_0 = \Spec \bbQ[c_2, c_8, c_{12}, c_{14},\\ c_{18}, c_{20}, c_{24}, c_{30}]$ by the formula $\zeta \cdot c_i = \zeta^{-i} c_i$. If $\OpenSurface$ denotes the restriction of $\OpenSurface_0$ to $B = \Spec \bbQ[c_{12}, c_{18}, c_{24}, c_{30}]\\ = B_0^{\mu_3}$, then the action of $\mu_3$ on the fibres of $\OpenSurface \to B$ induced by $\theta$ is given by the formula $\zeta\cdot(x, y, z) = ( x, y, \zeta^{-1} z)$. If we write $\Surface \to B$ for the pullback of $\Surface_0$ to $B$, then $\mu_3$ acts on the fibres of $\Surface \to B$ by the same formula, and we can identify the fixed locus $\Surface^{\mu_3}$ with the projective completion $\Curve$ of the family of affine curves
\begin{equation}
\Curve^0 : y^2 = x^5 + c_{12} x^3 + c_{18} x^2 + c_{24} x + c_{30}
\end{equation}
that is the object of our study in this paper.

We now come back to Proposition \ref{prop_normalizing_invariant_polynomials}. We recall that we have defined $\Lambda = X^\ast(\Cen)$, where $\Cen = Z_H(\Kostant|_{B^\text{rs}})$ is a torus over $B^\text{rs}$. Thus $\Lambda$ is an \'etale sheaf of $E_8$ root lattices and may be identified with the pullback of $\Lambda_0$ along the closed immersion $B \to B_0$. The image of the stable $\bbZ / 3 \bbZ$-grading $\theta_{B^\text{rs}} : \mu_3 \to H_{B^\text{rs}}$ normalizes $\Cen$ and determines an elliptic $\mu_3$-action $\mu_3 \to \Aut(\Lambda)$ that we also denote by $\theta_{B^\text{rs}}$. To prove the proposition, we must show that there is an isomorphism $\Lambda_\theta \to \Jacobian[3]$ intertwining the pairing $\langle \cdot, \cdot \rangle$ on $\Lambda_\theta$, described in \S \ref{sec_root_lattice_of_type_E8}, with the Weil pairing on $\Jacobian[3]$.

We define the morphism $\Lambda_\theta \to \Jacobian[3]$ in stages. Since both $\Lambda$ and $\Jacobian[3]$ are locally constant \'etale sheaves over $B^\text{rs}$, it is enough to define this morphism at the generic point $\eta$ of $B^\text{rs}$. Let $\eta$ now denote the generic point of $B^\text{rs}$, and let $\overline{\eta}$ be a geometric point above it. Let $\Lattice = \langle \OO, \FF \rangle \subset \Pic(\Surface_{\overline{\eta}})$, and let $\Lattice^\perp \subset \Pic(\Surface_{\overline{\eta}})$ be its orthogonal complement. We have the following corollary of Lemma \ref{lem_isomorphism_of_lattices}:
\begin{corollary}\label{cor_intertwining_in_V}
	There is a $\pi_1(\eta, \overline{\eta})$-equivariant isomorphism $X^\ast(\Cen_{\overline{\eta}}) \to \Pic(\OpenSurface_{\overline{\eta}})$ that intertwines the Weyl-invariant pairing $(\cdot, \cdot )$ on the source with the negative of the intersection pairing of $\Lattice^\perp$ on the target.
\end{corollary}
We define a morphism $X^\ast(\Cen_{\overline{\eta}}) \to \Pic^0(\Curve_{\overline{\eta}})$ as the composite
\[ X^\ast(\Cen_{\overline{\eta}}) \to \Pic(\OpenSurface_{\overline{\eta}}) \to \Pic(\Curve_{\overline{\eta}}^0) \cong \Pic^0(\Curve_{\overline{\eta}}), \]
where the first arrow represents the map given by Corollary \ref{cor_intertwining_in_V}, the second denotes pullback along $\Curve^0 \to \OpenSurface$, and the third map is the natural isomorphism $\Pic(\Curve^0_{\overline{\eta}}) \cong \Pic(\Curve_{\overline{\eta}}) / \langle \Point_{\overline{\eta}} \rangle \cong  \Pic^0(\Curve_{\overline{\eta}})$. (Recall that $\Point$ is the section at infinity as defined in Section \ref{section_kostant_section}.) Equivalently, we can define the morphism as the composite
\[ X^\ast(\Cen_{\overline{\eta}}) \to \Pic(\OpenSurface_{\overline{\eta}}) \cong \Lattice^\perp \subset \Pic(\Surface_{\overline{\eta}}) \to \Pic(\Curve_{\overline{\eta}}), \]
where the last arrow is now pullback along $\Curve \to \Surface$.

The map $X^\ast(\Cen_{\overline{\eta}}) \to \Pic(\Curve_{\overline{\eta}})$ is $\pi_1(\eta, \overline{\eta})$-equivariant and factors through the $\theta_{B^\text{rs}}$-coinvariants in $X^\ast(\Cen_{\overline{\eta}})$, which are 3-torsion. We obtain a morphism $\Lambda_\theta \to \Jacobian[3]$ of locally constant \'etale sheaves on $B^\text{rs}$.
\begin{proposition}\label{prop_E8_pairing_equals_Weil_pairing}
	The morphism $\Lambda_\theta \to \Jacobian[3]$ just described is an isomorphism that intertwines the pairing $\langle \cdot, \cdot \rangle$ on $\Lambda_\theta$ with the Weil pairing on $\Jacobian[3]$.
\end{proposition}
\begin{proof}
	It suffices to check this statement at the geometric generic point of $B^\text{rs}$. Let us write $\langle \cdot, \cdot \rangle_W$ for the Weil pairing on $\Jacobian[3]_{\overline{\eta}} = \Pic(\Curve_{\overline{\eta}})[3]$. We will consider the factorization
		\[ X^\ast(\Cen_{\overline{\eta}}) \to \Pic(\OpenSurface_{\overline{\eta}}) \to \Pic(\Curve^0_{\overline{\eta}}). \]
	Let us write $( \cdot, \cdot)_{\Lattice^\perp}$ for the (negative-definite) intersection pairing on $\Pic(\OpenSurface_{\overline{\eta}})$. 
Applying Corollary \ref{cor_intertwining_in_V}, we think of $\Pic(\OpenSurface_{\overline{\eta}})$ as an $E_8$ root lattice with pairing
$( \cdot, \cdot)$ given by the negative of $( \cdot, \cdot)_{\Lattice^\perp}$. 
%The pairing $( \cdot, \cdot)$ corresponds under the isomorphism $X^\ast(\Cen_{\overline{\eta}}) \cong \Pic(\OpenSurface_{\overline{\eta}})$ of Corollary \ref{cor_intertwining_in_V} to the pairing on $X^\ast(\Cen_{\overline{\eta}})$ which is also denoted by $(\cdot, \cdot)$. 
To prove the proposition, it is enough to show that the map $\Pic(\OpenSurface_{\overline{\eta}}) \to \Pic(\Curve^0_{\overline{\eta}})$ factors through an isomorphism $\psi : \Pic(\OpenSurface_{\overline{\eta}})_\theta \to \Pic^0(\Curve_{\overline{\eta}})[3]$ that satisfies the identity
	\begin{equation}\label{eqn_identity_of_mod_3_pairings} \zeta^{((1-\theta_{B^\text{rs}}(\zeta)) \alpha, \beta)} = \zeta^{-((1-\theta_{B^\text{rs}}(\zeta)) \alpha, \beta)_{\Lattice^\perp}}= \langle \psi(\alpha), \psi(\beta) \rangle_W 
	\end{equation}
	for all $\alpha, \beta \in \Pic(\OpenSurface_{\overline{\eta}})$. 	We will show this via an explicit calculation. Recall (cf. \S \ref{sec_root_lattice_of_type_E8}) that $\Pic(\OpenSurface_{\overline{\eta}})_\theta$ is isomorphic as an abelian group to $(\bbZ / 3 \bbZ)^4$. Its 80 non-trivial elements are in bijective correspondence with the $\theta_{B^\text{rs}}$-orbits of root vectors in $\Pic(\OpenSurface_{\overline{\eta}})$. To prove the result, it will suffice to show that these 80 non-trivial elements are in bijection with the non-trivial elements of $\Pic(\Curve_{\overline{\eta}})[3]$, and that if $\alpha, \beta \in \Pic(\OpenSurface_{\overline{\eta}})$ are two root vectors, then they satisfy the identity (\ref{eqn_identity_of_mod_3_pairings}).

	The root vectors $\alpha$ in $\Pic(\OpenSurface_{\overline{\eta}})$ correspond exactly to the sections $s_\alpha : \bbP^1_{\overline{\eta}} \to \Surface_{\overline{\eta}}$ of $\Surface_{\overline{\eta}} \to \bbP^1_{\overline{\eta}}$ that do not meet the zero section $\OO$ (see e.g.\ \cite{Shi10}). If $\ell_\alpha \subset  \Surface_{\overline{\eta}}$ denotes the image of $s_\alpha$, then the element of $\Lattice^\perp \subset \Pic(\Surface_{\overline{\eta}})$ corresponding to $\alpha$ is $\ell_\alpha - \OO - \FF$. Each of the sections $s_\alpha$ admits a unique expression $(a(x), b(x))$, where $a(x), b(x) \in k(\overline{\eta})[x]$ have degrees $2$ and $3$ respectively and satisfy $b(x)^2 = a(x)^3 + f(x)$. Suppose $s_\alpha = (a(x), b(x))$ and $s_\beta = (c(x), d(x))$ are two such sections. If $\alpha \neq \pm \beta$ then $b(x) - d(x)$ has 3 zeroes in $k(\overline{\eta})$, counted with multiplicity, and we have the formula
	\[ \begin{split} ( \alpha, \beta ) & = - (\ell_\alpha - \OO - \FF, \ell_\beta - \OO - \FF)_{\Lattice^\perp} = 1 - ( \ell_\alpha, \ell_\beta )_{\Lattice^\perp} \\ & = 1 - | \{ \gamma \in k(\overline{\eta}) \mid b(\gamma) = d(\gamma), a(\gamma) = c(\gamma) \}|, \end{split} \]
	(see e.g.\ \cite[\S 8.7]{Sch10}).
	
	Note that $\theta_{B^\text{rs}}(\zeta)(s_\alpha) = (\zeta^{-1} a(x), b(x))$. Therefore we have the formula
	\begin{equation}\label{eqn_putative_weil_pairing} \zeta^{((1-\theta_{B^\text{rs}}(\zeta))\alpha, \beta)} = \zeta^{-| \{ \gamma \in k(\overline{\eta}) \mid b(\gamma) = d(\gamma), a(\gamma) = c(\gamma) \} | + | \{ \gamma \in k(\overline{\eta}) \mid b(\gamma) = d(\gamma), \zeta^{-1} a(\gamma) = c(\gamma) \} |}. 
	\end{equation}
	Note that $\psi(\alpha)$ is the class of the divisor $P_1 + P_2 - 2 \infty$, where $a(x)$ has roots $\gamma_1, \gamma_2$ in $k(\overline{\eta})$ and $P_i = (\gamma_i, b(\gamma_i))$. 
	
	We compare this with the Weil pairing $\langle \psi(\alpha), \psi(\beta) \rangle_W$ on a case-by-case basis as follows. Assuming as we may that $\alpha \neq \theta_{B^\text{rs}}(\zeta^i) \beta$ for any $i \in \bbZ$, we see that $a(x) - c(x)$ is not the zero polynomial.  Let $\Sigma(\alpha, \beta)$ denote the set of zeroes in $k(\overline{\eta})$ of $b(x) - d(x)$; it has 3 elements. For each $\gamma \in \Sigma(\alpha, \beta)$, we have $a(\gamma)^3 = c(\gamma)^3$, hence $\omega(\gamma) = c(\gamma) / a(\gamma)$ is a 3rd root of unity. We now divide into 2 cases.
	
	For the first case, assume the values $\omega(\gamma)$, $\gamma \in \Sigma(\alpha, \beta)$, are pairwise distinct. In this case we see that both the Weil pairing $\langle \psi(\alpha), \psi(\beta) \rangle_W$ and the value given by (\ref{eqn_putative_weil_pairing}) are equal to 1. Indeed, the Weil pairing can be computed using \cite[Lemma 5]{Bru14}, while for (\ref{eqn_putative_weil_pairing}) this is obvious.
	
	For the second case, assume that some value $\omega(\gamma)$ occurs exactly twice. Since the pairing $\langle \cdot, \cdot \rangle$ does not depend on the choice of $\zeta$, we can suppose without loss of generality that it is $\zeta$ that appears twice. Then we see that (\ref{eqn_putative_weil_pairing}) gives a value of $\zeta^{-1}$ (if the other value of $\omega(\gamma)$ is 1) or $\zeta$ (if the other value of $\omega(\gamma)$ is $\zeta^{-1}$). Again, this agrees with the result of \cite[Lemma 5]{Bru14}. This concludes the proof.
\end{proof}
The only part of Proposition \ref{prop_normalizing_invariant_polynomials} that remains to be proved is that $\Delta_0 = \disc( x^5 + c_{12} x^3 + c_{16} x^2 + c_{24} x + c_{30})$ has the property that $\Delta_0^2$ is (up to scalar) the restriction to $V$ of the Lie algebra discriminant $\Delta$. Note that $\Delta_0^2$ and $\Delta$ both have degree 240, that $\Delta_0$ is irreducible, and that $\Delta_0^2$ and $\Delta$ vanish along the same points (as follows from e.g.\ \cite[\S 6.6]{Slo80}). This implies that they are equal up to scalar, as desired.
\subsection{Spreading out}\label{sec_spreading_out}

So far we have described the structure of the pair $(G, V)$ over $\bbQ$. We recall (see \S \ref{sec_definition_of_grading}) that this pair has a natural extension $(\intG, \intV)$ over $\bbZ$. We now observe that the above results hold if we work over $\bbZ[1/N]$ for an appropriate choice of integer $N \geq 1$. 

Indeed, we can choose the invariant polynomials $c_{12}, c_{18}, c_{24}, c_{30} \in \bbQ[V]^G$ to lie in $\bbZ[\intV]^{\intG}$ (by using the $\bbG_m$-action on $\OpenSurface_0$ described at the beginning of \S \ref{sec_proof_of_prop_on_invariants} to clear denominators). We set $\intB = \Spec \bbZ[c_{12}, \dots, c_{30}]$ and write $\pi : \intV \to \intB$ for the corresponding morphism (which extends the morphism $V \to B$ on $\bbQ$-fibres already denoted by $\pi$).  Note that this implies that $\Delta_0 = \disc(x^5 + c_{12} x^3 + c_{18} x^2 + c_{24} x + c_{30})  \in \bbZ[\intV]^{\intG}$ too. We define $\intB^\text{rs} = \Spec \bbZ[c_{12}, c_{18}, c_{24}, c_{30}][\Delta_0^{-1}]$. We extend $\Curve$ to a family of projective curves $\Curve \to \intB$ given by the same equation as before.

We can now find an integer $N \geq 1$ satisfying the following properties:
\begin{enumerate}
\item Let $S = \bbZ[1/N]$. Then each prime $p$ dividing the order of the Weyl group of $H$ (i.e.\ $p \in \{ 2, 3, 5, 7\})$ is a unit in $S$. In particular, the morphism $\Curve_S \to \intB_S$ is smooth exactly above $\intB_S^\text{rs}$. 
\item $S[\intV]^{\intG} = S[c_{12}, c_{18}, c_{24}, c_{30}]$. The Kostant section extends to a section $\Kostant : \intB_S \to \intV_S$ of $\pi$ that satisfies the following property: for any $f \in \intB(\bbZ) \subset \intB(S)$, $\Kostant(N \cdot f) \in \intV(\bbZ)$. We write $\kappa_S \subset \intV_S$ for the image of the Kostant section.
\item There exist open subschemes $\intV^\text{rs} \subset \intV^\text{reg} \subset \intV_S$ such that if $S \to k$ is a map to a field and $v \in \intV(k)$, then $v$ is regular if and only if $v \in \intV^\text{reg}(k)$ and $v$ is regular semisimple if and only if $v \in \intV^\text{rs}(k)$. Moreover, $\intV^\text{rs}$ is the locus in $\intV_S$ where $\Delta_0$ does not vanish. We note that this condition implies that $\Delta_0^2 = \Delta$ up to a unit in $S$.
\item Let $\intCen = Z_{\intH}(\Kostant_S|_{\intB_S^\text{rs}})$, a maximal torus in $\intH_{\intB_S^\text{rs}}$. We now write $\Lambda$ for $X^\ast(\intCen)$. Then $\Lambda$ is an \'etale sheaf of $E_8$ root lattices on $\intB_S^\text{rs}$, equipped with a pairing $( \cdot, \cdot ) : \Lambda \times \Lambda \to \bbZ$ and an elliptic $\mu_3$-action induced by $\theta_{\intB_S^\text{rs}}: \mu_3 \to \intH_{B_S^\text{rs}}$. (We write $\theta_{\intB_S^\text{rs}}$ for this elliptic $\mu_3$-action, but for ease of notation, we again write $\Lambda_\theta$ for the $\theta_{\intB_S^\text{rs}}$-covariants in $\Lambda$.)
\item There is a perfect pairing $\langle \cdot, \cdot \rangle : \Lambda_\theta \times \Lambda_\theta \to \mu_3$ given by the formula $\langle \lambda, \mu \rangle = \zeta^{((1 - \theta_{\intB_S^\text{rs}}(\zeta))(\lambda), \mu)}$ for any primitive 3rd root of unity $\zeta$. We can therefore extend  $\operatorname{GrLieE}_{f}$ and $\operatorname{Heis}_{f}$ to stacks over $\intB_S^\text{rs}$.
\item Let $\Jacobian \to \intB_S^\text{rs}$ denote the Jacobian of $\Curve_{\intB_S^\text{rs}}$. Then there is an isomorphism $\Lambda_\theta \cong \Jacobian[3]$ of locally constant \'etale sheaves on $\intB_S^\text{rs}$ that intertwines the pairing $\langle \cdot, \cdot \rangle$ on $\Lambda_\theta$ and the Weil pairing of $\Jacobian[3]$.
\item Let $\prin = \cO_\Jacobian( \Curve - \Point)$, a symmetric non-degenerate line bundle on $\Jacobian$, and let $\tprin = \prin^{\otimes 3}$. Let $\extJ \subset \mathscr{G}(\tprin)$ be the subgroup of 3-torsion elements (here $\Point$ and $\extJ (\tprin)$ are defined as in \S \ref{section_kostant_section}). Then $\extJ$ is an extension
\[ 1 \to \mu_3 \to \extJ \to \Jacobian[3] \to 1, \]
and the analogue of Proposition \ref{prop_Jacobian_and_Kostant_section} holds in $\operatorname{GrLieE}_{f_\tau, \intB^\text{rs}}$.
\end{enumerate}
With these data in hand, we can extend our constructions of orbits from sections of Jacobians. We can therefore apply the results of \S \ref{sec_twisting} for $S$-algebras $R$ (and not just $\bbQ$-algebras). We mention in particular:
\begin{enumerate}
\item Let $R$ be an $S$-algebra and let $f \in \intB^\text{rs}(R)$. Suppose we are given a tuple $(H', \theta', \gamma') \in \operatorname{GrLieE}_{f, R}$. If $(H', \theta') \cong (H_R, \theta_R)$, then $(H', \theta', \gamma')$ determines an element of $\intG(R) \backslash \intV_f(R)$, a set that is in turn in canonical bijection with the set $\ker(H^1(R, Z_{\intG}(\Kostant(f))) \to H^1(R, \intG))$. 
\item Let $R$ be an $S$-algebra, and let $f \in \intB^\text{rs}(R)$. Suppose that every locally free $R$-module is free. Then there is an injective map $\eta_f : \Jacobian_f(R) / 3 \Jacobian_f(R) \to \intG(R) \backslash \intV_f(R)$ that is compatible with base change on $R$.
\end{enumerate}

\subsection{Measures}\label{sec_measures}

The results of this section are used in the calculations of \S \ref{sec_counting_points} and \S \ref{sec_main_theorem}. Let $\omega_G$ be a generator for the (free rank-one $\bbZ$-module of) left-invariant top forms on $\intG$. Then $\omega_G$ is uniquely determined up to sign, and it determines Haar measures $dg$ on $G(\bbR)$ and on $G(\bbQ_p)$ for each prime $p$. 
\begin{proposition}\label{prop_tamagawa_number_of_G}
	The product $\vol(\intG(\bbZ) \backslash \intG(\bbR)) \cdot \prod_p \vol(\intG(\bbZ_p))$ converges absolutely, and equals 3.
\end{proposition}
\begin{proof}
	Note that $\intG$ has class number 1 (i.e. $\intG(\bbQ) \backslash \intG(\bbA^\infty) / \intG(\widehat{\bbZ})$ has 1 element). Therefore the product expresses the Tamagawa number of the simple group $G = \SL_9 / \mu_3$, which equals 3 (apply the results of \cite{Lan66} and \cite{Ono65}).
\end{proof}
Let $\omega_V$ be a generator for the $\bbZ$-module of left-invariant top forms on $\intV$. Then $\omega_V$ is also determined up to sign, and determines Haar measures $dv$ on $V(\bbR)$ and on $V(\bbQ_p)$ for every prime $p$. We write $\omega_B$ for the form $dc_{12} \wedge d c_{18} \wedge dc_{24} \wedge dc_{30}$ on $\intB$. It determines measures $df$ on $B(\bbR)$ and on $B(\bbQ_p)$ for every prime $p$.
\begin{proposition}\label{prop_integration_in_fibres}
	There exists a constant $W_0 \in \bbQ^\times$ such that for any prime $p$, the following properties hold:
	\begin{enumerate}
		\item Let $\intV(\bbZ_p)^\text{rs} = \intV(\bbZ_p) \cap \intV^\text{rs}(\bbQ_p)$, and define a function $m_p : \intV(\bbZ_p)^\text{rs} \to \bbR_{\geq 0}$ by the formula
		\[ m_p(v) = \sum_{v' \in \intG(\bbZ_p) \backslash (\intG(\bbQ_p) \cdot v \cap \intV(\bbZ_p))} \frac{| Z_{\intG}(v)(\bbQ_p)|}{|Z_{\intG}(v')(\bbZ_p)|}. \]
		Then $m_p(v)$ is locally constant.
		\item Let $\intB(\bbZ_p)^\text{rs} = \intB(\bbZ_p) \cap \intB^\text{rs}(\bbQ_p)$, and let $\psi_p : \intV(\bbZ_p)^\text{rs} \to \bbR_{\geq 0}$ be a bounded, locally constant function that is $\intG(\bbQ_p)$-invariant, in the sense that if $v, v' \in \intV(\bbZ_p)$ are conjugate under the action of $\intG(\bbQ_p)$, then $\psi_p(v) = \psi_p(v')$. Then we have the formula
		\[ \int_{v \in \intV(\bbZ_p)^\text{rs}} \psi_p(v) \, dv = | W_0 |_p \vol(\intG(\bbZ_p)) \int_{f \in \intB(\bbZ_p)^\text{rs}} \sum_{g \in \intG(\bbQ_p) \backslash \intV_f(\bbZ_p)} \frac{m_p(v) \psi_p(v)}{| Z_{\intG}(v)(\bbQ_p) |} \, df. \]
		\item Let $U_0 \subset G(\bbR)$ and $U_1 \subset B^\text{rs}(\bbR)$ be open subsets such that the morphism $\mu: U_0 \times U_1 \to V^\text{rs}(\bbR)$ defined by $(g, f) \mapsto g \cdot \Kostant(f)$ is injective. Then we have the formula
		\[ \int_{v \in \mu(U_0 \times U_1)} \, dv = | W_0 | \int_{g \in U_0} \, dg \int_{f \in U_1} \, df. \] 
	\end{enumerate}
\end{proposition}
Here we write $| \cdot |_p$ for the usual $p$-adic absolute value on $\bbQ_p$ (with $| p |_p = p^{-1}$).
\begin{proof}
	All of these identities can be proved in the same way as in \cite[Proposition 3.3]{Rom18} and \cite[Proposition 2.16]{Tho15}. The key input in the proof is the equality $\dim_\bbQ V = \sum_i \deg c_i$, which holds here since $84 = 12 + 18 + 24 + 30$. 
\end{proof}

\section{Constructing integral orbit representatives} \label{sec_integral_orbits}

We continue with the notation of \S \ref{sec_stable_grading_of_E8}. Let $\Equations$ denote the set of polynomials $f(x) = x^5 + c_{12} x^3 + c_{18} x^2 + c_{24} x + c_{30} \in \bbZ[x]$ of non-zero discriminant. If $p$ is a prime, let $\Equations_p$ denote the set of polynomials $f(x) = x^5 + c_{12} x^3 + c_{18} x^2 + c_{24} x + c_{30} \in \bbZ_p[x]$ of non-zero discriminant. Thus we can identify $\Equations_p = \intB(\bbZ_p)^\text{rs} :=\intB(\bbZ_p) \cap B^\text{rs}(\bbQ_p)$. We endow $\mathscr{E}_p$ with the $p$-adic topology.

This section is devoted to the proof of the following theorem concerning the map $\eta_f$ of Corollary \ref{cor_construction_of_orbits_over_field}.
\begin{theorem}\label{thm_existence_of_integral_representatives_for_soluble_orbits}
	Let $N$ be an integer satisfying the properties listed in \S \ref{sec_spreading_out}. Then for each prime $p \nmid N$, for each polynomial $f(x) \in \Equations_p$, and for each $P \in \Jacobian_f(\bbQ_p)$, the orbit $\eta_f(P) \in G(\bbQ_p) \backslash V_f(\bbQ_p)$ intersects $\intV_f(\bbZ_p)$.
\end{theorem}
Most of \S \ref{sec_integral_orbits} is devoted to the proof of this theorem. We first prove the theorem for polynomials of square-free discriminant in \S \ref{sec_square_free_disc_case}. This special case is then used as an ingredient in the proof of the theorem in the general case in \S \ref{sec_general_case}. 

\subsection{The case of square-free discriminant}\label{sec_square_free_disc_case}

In this section we establish Theorem \ref{thm_existence_of_integral_representatives_for_soluble_orbits} for polynomials $f(x) \in \Equations_p$ of square-free discriminant. We first prove two useful lemmas.

\begin{lemma}\label{lem_trivial_class_group}
	Let $R$ be a Noetherian regular integral domain such that every locally free $R$-module of finite rank is free, and let $K = \Frac(R)$. Then the map $H^1(R, G) \to H^1(K, G)$ has trivial kernel.
\end{lemma}
\begin{proof}
	The existence of the short exact sequence of smooth $R$-groups
	\[ \xymatrix@1{ 1 \ar[r]& \mu_3 \ar[r] & \SL_9 \ar[r] & G \ar[r] & 1,} \]
	together with the triviality of $H^1(R, \SL_9)$, reduces the problem to showing that $H^2(R, \mu_3) \to H^2(K, \mu_3)$ is injective, or even that $H^2(R, \bbG_m) \to H^2(K, \bbG_m)$ is injective. This follows from \cite[1.8]{Gro68}.
\end{proof}

\begin{lemma}\label{lem_inflation_restriction}
	Let $R$ be a complete discrete valuation ring, let $K = \Frac(R)$, and let $\Gamma$ be a quasi-finite \'etale commutative $R$-group that satisfies the `N\'eron mapping property': $\Gamma(R') = \Gamma(\Frac(R'))$ for any \'etale extension $R \to R'$ of discrete valuation rings. Then the natural map $H^1(R, \Gamma) \to H^1(K, \Gamma)$ is injective.
\end{lemma}
\begin{proof}
	Let $j : \Spec K \to \Spec R$ be the natural open immersion. The `N\'eron mapping property' says that $A = j_\ast j^\ast A$. The map $H^1(R, A) \to H^1(K, A)$ is therefore injective because it is the first map in the 5-term exact sequence associated to the spectral sequence $H^p(R, R^q j_\ast j^\ast A) \Rightarrow H^{p + q}(K, A)$. 
\end{proof}
The following proposition contains a special case of Theorem \ref{thm_existence_of_integral_representatives_for_soluble_orbits} when $R = \bbZ_p$.
\begin{proposition}\label{prop_existence_of_orbits_of_squarefree_discriminant}
	Let $R$ be a discrete valuation ring in which $N$ is a unit. Let $K = \Frac(R)$, and let $\ord_K : K^\times \twoheadrightarrow \bbZ$ be the normalized discrete valuation. Let $f \in \intB(R)$. Suppose that $\ord_K \Delta_0(f) \leq 1$. Then:
	\begin{enumerate}
		\item If $x \in \intV_f(R)$, then $Z_{\intG}(x)(K) = Z_{\intG}(x)(R)$.
		\item The  natural map $\alpha : \intG(R) \backslash \intV_f(R) \to \intG(K) \backslash \intV_f(K)$ is injective and its image contains $\eta_f(\Jacobian_f(K) / 3 \Jacobian_f(K))$.
		\item  If further $R$ is complete and has finite residue field then the image of $\alpha$ equals $\eta_f(\Jacobian_f(K) / 3 \Jacobian_f(K))$.
	\end{enumerate}
\end{proposition}
\begin{proof}
	We first note that the first two parts of the proposition hold for $R$ if they hold for the completion $\widehat{R}$ of $R$.  To see this, we use the equality $\intG(\widehat{K}) = \intG(K) \intG(\widehat{R})$, where $\widehat{K} = \Frac \widehat{R}$ (see \cite[Th\'eor\`eme 3.2]{Nis84}). We therefore assume that $R$ is complete.
	
	If $\ord_K \Delta_0 (f) = 0$, then $\Jacobian_f$ is a smooth projective $R$-scheme and $\Jacobian_f(R) = \Jacobian_f(K)$. In particular the map $\Jacobian_f(K) / 3 \Jacobian_f(K) \to \intG(K) \backslash \intV_f(K)$ factors through $\intG(R) \backslash \intV_f(R)$. We therefore just need to check that $\intG(R) \backslash \intV_f(R) \to \intG(K) \backslash \intV_f(K)$ is injective. In fact, the map $H^1(R, \Jacobian_f[3]) \to H^1(K, \Jacobian_f[3])$ is injective (a special case of Lemma \ref{lem_inflation_restriction}), so this follows from Lemma \ref{lem_cohomological_description_of_orbits}.
	
	If the residue field of $R$ is finite, then $H^1(R, \intG) = \{ 1 \}$ and the map $\Jacobian_f(R) / 3 \Jacobian_f(R) \to H^1(R, \Jacobian_f[3])$ is an isomorphism, as $H^1(R, \Jacobian_f)$ is trivial, by Lang's theorem.
	
	Now suppose that $\ord_K \Delta_0 (f) = 1$. Roughly the same principles apply. Let $\cJ_f$ denote the N\'eron model of $\Jacobian_f$. Then $\cJ_f$ is a smooth group scheme over $R$ with generic fibre $\Jacobian_f$, and $\Jacobian_f(K) = \cJ_f(R)$. Our assumptions imply that $\cJ_f$ has connected fibres and that the special fibre of $\cJ_f$ is an extension of an elliptic curve by a rank 1 torus. (Indeed, $\Curve_f$ is projective over $R$ and regular. Its special fibre is integral and has a unique singularity, which is a node. Now one can compute using the results of \cite[Ch. 9]{Bos90}.) In particular, the quasi-finite \'etale group scheme $\cJ_f[3]$ over $R$ has generic fibre of order $3^4$ and special fibre of order $3^3$. 
	
	We claim that $\intV_f^\text{reg}(R) = \intV_f(R)$. To prove this, we must show that any element $x$ of $\intV_f(R)$ has regular image $x_k \in \intV_f(k)$, where $k$ is the residue field of $R$. Consider the direct sum decomposition $\intfrh_R = \intfrh_{0, R} \oplus \intfrh_{1, R}$, where the restriction of $\ad(x)$ to $\intfrh_{0, R}$ is topologically nilpotent and the restriction of $\ad(x)$ to $\intfrh_{1, R}$ is invertible over $R$. Let $\varpi$ be a uniformizer of $R$. 
	Reducing the components of this decomposition modulo $\varpi$ yields the direct sum decomposition $\intfrh_k = \intfrh_{0, k} \oplus \intfrh_{1, k}$, where $\ad(x_k)$ acts as a nilpotent operator on $\intfrh_{0, k}$ and acts invertibly on $\intfrh_{1, k}$. In fact, if $x_k = y_s + y_n$ is the Jordan decomposition of $x_k$ as a sum of its semisimple and nilpotent parts, then $\intfrh_{0, k} = \frz_{\intfrh}(y_s)$. To show $x_k$ is regular, we must show that $y_n$ is a regular nilpotent element of $\frz_{\intfrh}(y_s)$.
	
	To see this, we first observe that there exists a unique closed subgroup $\underline{L} \subset \intH_R$ such that $\Lie \underline{L} = \intfrh_{0, R}$ and such that $\underline{L}$ is smooth over $R$ with connected fibres. Moreover, we have $\underline{L}_k = Z_{\underline{H}}(y_s)$. The uniqueness follows from \cite[Exp. XIV, Proposition 3.12]{SGA3II}. To show existence, choose a regular semisimple element $\overline{r} \in \frz_\intfrh(y_s)$ and an arbitrary lift $r \in \intfrh_{R, 0}$. The centralizer $Z_{\intH}(r)$ is a maximal torus of $\intH_R$ with Lie algebra contained in $\intfrh_{R, 0}$, and we can construct a Levi subgroup of $\intH_R$ with Lie algebra $\frh_{R, 0}$ after passage to an \'etale extension $R \to R'$ where $Z_{\intH}(r)$ is split.

	Let $f_k$ denote the reduction modulo $\varpi$ of $f$. Our condition on the discriminant implies that $f_k$ has exactly 1 repeated root. It follows that the fibre $y^2 = z^3 + f_k(x)$ in $\OpenSurface_{0, k}$ over $f_k \in \intfrh_k\dquot\intH_k$ has a unique singularity, which is a simple singularity of type $A_2$.
	Using the results of \cite[\S 6.5]{Slo80}, we see that the derived group of $\underline{L}$ has type $A_2$ and that the centre $Z_{\underline{L}}$ has rank 6. Moreover, the action of $\mu_3$ on $\intH_R$ determined by $\theta$ restricts to an action on $\underline{L}$, and the induced morphism $\theta_{\underline{L}} : \mu_3 \to \Aut(\underline{L})$ has the property that in each geometric fibre, there is a maximal torus of $\underline{L}$ on which $\theta_{\underline{L}}$ defines an elliptic $\mu_3$-action.
	
	%To show that $x_k$ is regular in $\intfrh_k$, we must show that $y_n$ is a regular nilpotent element in $\frh_{0, k}$. 
	Let $\intfrh_{0, R}^\text{der}$ denote the derived subalgebra of $\intfrh_{0, R}$. 
	After passage to an \'etale extension $R \to R'$ of discrete valuation rings, we can assume that there is an isomorphism $\intfrh_{0, R}^\text{der} \cong \frs\frl_{3, R}$ under which $\theta_{\underline{L}}$ corresponds to the homomorphism $\zeta \mapsto \Ad(\diag(1, \zeta, \zeta^2))$. (The proof is the same as the proof of Lemma \ref{lem_graded_groups_etale_locally_isomorphic}, using that the automorphism group of the $A_2$ root lattice contains a unique conjugacy class of elements of order 3.) Let $\Delta'$ denote the Lie algebra discriminant of $\intfrh_{0, R}^\text{der}$. Then $\ord_K \Delta(x) = \ord_K \Delta'(x) = 2$ by property 3 in \S \ref{sec_spreading_out}. 
	
	Let $x'$ denote the projection of $x$ to $\intfrh_{0, R}^\text{der}$ (this projection exists because of our assumption on the residue characteristic of $R$).
	The image of $x'$ under the isomorphism $\intfrh_{0, R}^\text{der} \to \frs\frl_{3, R}$ is given by a matrix of the form
	\[ x' = \left( \begin{array}{ccc} 0 & a & 0 \\ 0 & 0 &  b \\ c & 0 & 0 \end{array}\right), \]
	and the discriminant $\Delta'(x)$ equals $(abc)^2$. If $\ord_K (abc)^2 = 2$ then exactly one of $a, b$ or $c$ is divisible by $\varpi$, and so the reduction modulo $\varpi$ of $x'$ (which coincides with $y_n$) is a regular nilpotent element. This proves our claim that $\intV_f^\text{reg}(R) = \intV_f(R)$.
	
	We next claim that $Z_{\intG}(\Kostant(f))$ satisfies the `N\'eron mapping property': $Z_{\intG}(\Kostant(f))(R') = Z_{\intG}(\Kostant(f))(\Frac(R'))$ for any \'etale extension $R \to R'$ of discrete valuation rings. By Lemma \ref{lem_action_by_automorphisms} and property 6 of \S \ref{sec_spreading_out}, we may identify  $Z_{\intG}(\Kostant(f))_K$ with $\Jacobian_f[3]$. Thus it suffices to show that the isomorphism $Z_{\intG}(\Kostant(f))_K \cong \Jacobian_f[3]$ extends uniquely to an isomorphism $Z_{\intG}(\Kostant(f)) \cong \cJ_f[3]$. This will follow if we can show that the special fibre of $Z_{\intG}(\Kostant(f))$ has order $3^3$. This is the case. Taking $x$ to be $\sigma(f)$ in the above computation, we see that $Z_{\intG}(\Kostant(f))_k$ can be identified with the $\theta_{\underline{L}}$-fixed points in the centre $Z(\underline{L}_k)$ of the group $\underline{L}_k$. Since the centre is a rank-6 torus on which $\theta_{\underline{L}}$ defines an elliptic $\mu_3$-action, this group indeed has order $3^3$. (See \cite[Proposition 2.8]{Tho13} for a similar calculation. Note that while we have not defined an elliptic $\mu_3$-action in this setting, the definition is the same: the map on $X^*(Z(\underline{L}_k))$ induced by $\theta_{\underline{L}}(\zeta)$ is an elliptic automorphism of order 3.)	
	
	We claim that the map $\intG \to \intV_f^\text{reg}$, $g \mapsto g \cdot \Kostant(f)$ is surjective and \'etale and is in fact a torsor for the \'etale group scheme $Z_{\intG}(\Kostant(f))$. The only part of this claim that we have not already established is the fact that this map is surjective in the special fibre $\intG_k \to \intV_{f, k}^\text{reg}$. Thus the claim follows from the fact that if $y_s \in \intV_k$ is a semisimple element such that $Z_{\intH}(y_s)$ has derived group of type $A_2$, then $Z_{\intG}(y_s) = Z_{\intH}(y_s)^\theta$ acts transitively on the regular nilpotent elements of $\frz_{\intfrh}(y_s)(\theta, 1)$. (Note that in the $(\bbZ / 3 \bbZ)$-grading of $\frs\frl_{3, k}$ given by $\xi : \zeta \mapsto \Ad(1, \zeta, \zeta^2)$, $\SL_3^\xi$ does not act transitively on the regular nilpotent elements, but $\PGL_3^\xi$ does. In the present situation the group $Z_{\intH}(y_s)$ fits into a $\theta$-equivariant short exact sequence
	\[ 1 \to C \to Z_{\intH}(y_s) \to \PGL_3 \to 1, \]
	where $C$ is a $\theta$-elliptic torus. This implies that the map $Z_{\intH}(y_s)^\theta \to \PGL_3^\theta$ is surjective.)
	
	By the same logic as in Lemma \ref{lem_forms_of_H_theta}, it  follows that the set $\intG(R) \backslash \intV_f(R)$ is in bijection with $\ker(H^1(R, Z_{\intG}(\Kostant(f))) \to H^1(R, \intG))$. By Lemma \ref{lem_inflation_restriction}, the map $H^1(R, Z_{\intG}(\Kostant(f))) \to H^1(K, Z_{\intG}(\Kostant(f)))$ is injective, implying that the map $\alpha : \intG(R) \backslash \intV_f(R) \to \intG(K) \backslash \intV_f(K)$ is injective (cf. \cite[Exercise 2.4.11]{Con14}). To show that the image of $\alpha$ contains the image of $\eta_f$, we observe that we have a commutative diagram
	\[ \xymatrix{ \cJ_f(R) / 3 \cJ_f(R) \ar[r] \ar[d] & \Jacobian_f(K) / 3 \Jacobian_f(K) \ar[d] \\
		H^1(R, \cJ_f[3]) \ar[r] & H^1(K, \Jacobian_f[3]).} \]
	We therefore just need to show that each class in the image of $\cJ_f(R) / 3 \cJ_f(R)$ in $H^1(R, \cJ_f[3]) \cong H^1(R, Z_{\intG}(\Kostant(f)))$ has trivial image in $H^1(R, \intG)$. This follows from the fact that the map $H^1(R, \intG) \to H^1(K, \intG)$ is injective (Lemma \ref{lem_trivial_class_group}).
	
	Finally, suppose once more that $R$ has finite residue field. Lang's theorem once again implies that $H^1(R, \intG) = \{ 1 \}$ and $H^1(R, \cJ_f) = \{ 1 \}$. This completes the proof.
\end{proof}

\begin{corollary}\label{cor_global_integral_orbits_in_squarefree_discriminant_case}
	Let $R$ be a PID in which $N$ is a unit, and let $f \in \intB(R)$ be a polynomial such that the discriminant $\Delta_0(f)$ is square-free (as an element of $R$). Let $K = \Frac(R)$. Let $P \in \Jacobian_f(K)$, and let $\gamma_P \in \intV_f(K)$ be a representative of the orbit $\eta_f(P)$. Then there exists $g \in \intG(K)$ such that $g \cdot \gamma_P \in \intV_f(R)$. In particular, the triple $(H_K, \theta_K, \gamma_P) \in \operatorname{GrLieE}_{f, K}$ extends to an element of $\operatorname{GrLieE}_{f, R}$.
\end{corollary}
\begin{proof}
Note that the second statement follows immediately from the first: if $g \cdot \gamma_P \in \intV_f(R)$, then $g$ defines an isomorphism between $(H_K, \theta_K, \gamma_P)$ and $(H_K, \theta_K, g \cdot \gamma_P)$, and the latter triple extends naturally to $(\intH_R, \theta_R, g \cdot \gamma_P) \in \operatorname{GrLieE}_{f, R}$. For the first statement, we can reduce immediately (using \cite[Th\'eor\`eme 2.1]{Nis84} and Lemma \ref{lem_trivial_class_group}) to the case where $R$ is a discrete valuation ring, which is treated above.

%	After localizing, we can assume that $R$ is a DVR. In this case, Proposition \ref{prop_existence_of_orbits_of_squarefree_discriminant} implies that we can find $g \in \intG(K)$ such that $g \cdot \gamma_P \in \intV_f(R)$. Note that $g$ defines an isomorphism between $(H_K, \theta_K, \gamma_K)$ and $(H_K, \theta_K, g \cdot \gamma_P)$, and the latter triple extends naturally to $(\intH_R, \theta_R, g \cdot \gamma_P) \in \operatorname{GrLieE}_{f, R}$.
%	
%	original: 	We show that the tuple $(H_K, \theta_K, \gamma_P) \in \operatorname{GrLieE}_{K, f}$ extends to a tuple $(H_0, \theta_0, \gamma_0) \in \operatorname{GrLieE}_{f, R}$. By Lemma \ref{lem_trivial_class_group} and Lemma \ref{lem_forms_of_H_theta}, we have $(H_0, \theta_0) \cong (\intH_R, \theta_R)$, and the corollary will follow from this.
\end{proof}

\subsection{The general case}\label{sec_general_case}

We now use the results just established in \S \ref{sec_square_free_disc_case} to complete the proof of Theorem \ref{thm_existence_of_integral_representatives_for_soluble_orbits}. Let us therefore fix a prime $p > N$, a polynomial $f(x) \in \Equations_p$, and a point $P \in \Jacobian_f(\bbQ_p)$. We must show that the orbit $\eta_f(P) \subset V_f(\bbQ_p)$ contains an element of $\intV_f(\bbZ_p)$.

We first give an explicit representation of the point $P$. Arguing as in the proof of \cite[Proposition 19]{Bha13}, we can assume (after possibly changing $P$ without changing its image in $\Jacobian_f(\bbQ_p) / 3 \Jacobian_f(\bbQ_p)$) that $P$ corresponds to a decomposition $f(x) = u_0(x) v_0(x) + r_0(x)^2$, where for some $\nu \in \{ 0, 1, 2 \}$, we have that $u_0(x), v_0(x) \in \bbZ_p[x]$ are monic of degrees $\nu$ and $5 - \nu$, respectively, and $r_0(x)$ has degree at most $\nu - 1$. (This is the Mumford representation of $P$: thus $P$ corresponds to the linear equivalence class of the divisor $D - \nu \infty$, where $D \subset \Curve^0_{f, \bbQ_p}$ is the effective divisor of degree $\nu$ determined by the equations $y = r_0(x), u_0(x) = 0$.)

Let $D_\nu$ denote the scheme (over $\bbZ_p$) of tuples of polynomials $(u(x), v(x), r(x))$, where $u(x), v(x)$ are monic of degrees $\nu$ and $5 - \nu$, respectively, $r(x)$ has degree at most $\nu - 1$, and $u(x) v(x) + r(x)^2 = x^5 + a_1 x^4 + a_2 x^3 + a_3 x^2 + a_4 x + a_5$ satisfies $a_1 = 0$. Thus the tuple $(u_0(x), v_0(x), r_0(x))$ determines a point of $D_\nu(\bbZ_p)$. Let $\delta = \delta(u, v, r) \in H^0(D_\nu, \cO_{D_\nu})$ denote the discriminant of the (monic, degree-5) polynomial $u(x) v(x) + r(x)^2$, and let $D_\nu^\delta \subset D_\nu$ denote the closed subscheme defined by the vanishing of $\delta$. Then $D_\nu^\delta$ has codimension 1 in each fibre of $D_\nu$ over $\bbZ_p$. (In fact, $D_\nu^\delta$ is flat over $\intB_{\bbZ_p}$.)

Let $\lambda$ be a formal variable. We can find a point $(u_1(x), v_1(x), r_1(x)) \in D_\nu(\bbZ_p[\lambda])$ with the following properties:
\begin{itemize}
\item The reduction mod $\lambda$ of  $(u_1(x), v_1(x), r_1(x))$ is  $(u_0(x), v_0(x), r_0(x))$.
\item Let $f_1(x) = u_1(x) v_1(x) + r_1(x)^2 \in \bbZ_p[\lambda][x]$. Then $\disc f_1 = \delta(u_1, v_1, r_1)$ is square-free when viewed as an element of the ring $\bbQ_p[\lambda]$, and its image in $\bbF_p[\lambda]$ is non-zero. 
\end{itemize}
(We can find such a point by choosing  $u_1 = u_0 + \lambda u_0'$, $v_1 = v_0 + \lambda v_0', r_1 = r_0 + \lambda r_1'$ for some polynomials $u_0', v_0', r_0' \in \bbZ_p[x]$ such that the discriminant of $f_1(x)$ is not zero in $\bbF_p[\lambda]$. If the discriminant is not already square-free in $\bbQ_p[\lambda]$ then by Bertini's theorem we can choose a small $p$-adic perturbation to make it so.)

Let $U_1 = \Spec \bbZ_p[\lambda][\disc(f_1)^{-1}]$. We have constructed a smooth projective curve $\Curve_{f_1} \to U_1$, together with a section $P_1 \in \Jacobian_{f_1}(U_1)$. Using the same logic as in the proof of Corollary \ref{cor_construction_of_orbits_over_general_base}, we obtain a tuple $(H_1, \theta_1, \gamma_1) \in \operatorname{GrLieE}_{U_1, f_1}$. The pullback of this tuple to $\operatorname{GrLieE}_{\bbQ_p, f}$ along the point $\{ \lambda = 0 \} \in U_1(\bbQ_p)$ corresponds to the orbit $\eta_f(P)$ under the bijection of Lemma \ref{lem_cohomological_description_of_orbits}.

Let $U_2 = \Spec \bbQ_p[\lambda]$. Using that $\disc(f_1)$ is square-free when viewed as an element of $\bbQ_p[\lambda]$, we can apply Corollary \ref{cor_global_integral_orbits_in_squarefree_discriminant_case} to find that there is an extension of the triple $(H_{1}[1/p], \theta_1[1/p], \gamma_1[1/p])$ to a triple $(H_2, \theta_2, \gamma_2)$ over $U_2$. We can glue these triples to obtain a triple $(H_0, \theta_0, \gamma_0)$ over $U_0 := U_1 \cup U_2 \subset \Spec \bbZ_p[\lambda]$. Observe that by construction, $\theta_0: \mu_3 \to H_0$ is a stable $\bbZ / 3 \bbZ$-grading.

Note that the complement of $U_0$ in $\Spec \bbZ_p[\lambda]$ is a union of finitely many closed points in the special fibre. We now apply the following lemma.
\begin{lemma}\label{lem_reductive_groups_are_reflexive}
Let $S$ be an integral regular scheme of dimension 2, and let $Z \subset S$ be a closed subset of dimension 0. Let $U = S - Z$. Then restriction $M \mapsto M_U$ defines an equivalence between the following two categories:
\begin{enumerate}
		\item The category of reductive groups over $S$, with morphisms given by isomorphisms of group schemes. 
	\item The category of reductive groups over $U$, with morphisms given by isomorphisms of group schemes.
\end{enumerate}
Moreover, if $M$ is a reductive group over $S$, then $H^0(S, \ffrm) = H^0(U, \ffrm_U)$, where $\ffrm = \Lie(M)$ and $\ffrm_U = \Lie(M_U)$.
\end{lemma}
\begin{proof}
The essential surjectivity is \cite[Theorem 6.13]{Col79}. If $M, M'$ are reductive group schemes over $S$, then the scheme of isomorphisms between $M$ and $M'$ is $S$-affine; this shows that the functor is fully faithful (cf. \cite[Lemma 2.1]{Col79}).
\end{proof}
Applying Lemma \ref{lem_reductive_groups_are_reflexive}, we see that $H_0$ extends uniquely to a reductive group $H_3$ over $\Spec \bbZ_p[\lambda]$, that $\theta_0$ extends uniquely to a grading $\theta_3 : \mu_3 \to H_3$, and that $\gamma_0$ determines a unique section $\gamma_3 \in \frh_3(\theta_3, 1)$. Note that $\theta_3$ is a stable $\bbZ / 3 \bbZ$-grading. It follows that $(H_3, \theta_3, \gamma_3)$ is an object of the category $\operatorname{GrLieE}_{\bbZ_p[\lambda], f_1}$. By construction, its pullback to $\operatorname{GrLieE}_{\bbQ_p, f}$ along the map $\lambda = 0$ corresponds under the bijection of Lemma \ref{lem_cohomological_description_of_orbits} to the orbit $\eta_f(P) \in G(\bbQ_p) \backslash V_f(\bbQ_p)$. 

Let $(H_4, \theta_4, \gamma_4) \in \operatorname{GrLieE}_{\bbZ_p, f}$ denote the pullback of $(H_3, \theta_3, \gamma_3)$ to $\bbZ_p$. Since $H^1(\bbZ_p, \intG) = \{ 1 \}$, this triple determines an orbit in $\intG(\bbZ_p) \backslash \intV_f(\bbZ_p)$ mapping to $\eta_f(P)$ under the natural map $\intG(\bbZ_p) \backslash \intV_f(\bbZ_p) \to G(\bbQ_p) \backslash V_f(\bbQ_p)$. This completes the proof of Theorem \ref{thm_existence_of_integral_representatives_for_soluble_orbits}.

\subsection{Complements}

We conclude \S \ref{sec_integral_orbits} with a weak result that holds for every prime (not just primes $p > N$). The $\bbG_m$-action on $\intB$ here is the standard one (where $t \cdot c_i = t^i c_i$).
\begin{proposition}\label{prop_weak_existence_of_integral_representatives}
Let $p$ be a prime, and let $f_0(x) \in \Equations_p$. Then there exists an integer $n \geq 1$ and an open neighbourhood $W_p$ of $f_0$ in $\Equations_p$ such that for all $f \in W_p$ and for all $y \in \Jacobian_{p^n \cdot f}(\bbQ_p)$, the orbit $\eta_{p^n \cdot f}(y) \in G(\bbQ_p) \backslash V_{p^n \cdot f}(\bbQ_p)$ contains an element of $\intV_{p^n \cdot f}(\bbZ_p)$.
\end{proposition}
\begin{proof}
Choose $n \geq 1$ such that each orbit in the image of $\eta_{p^n \cdot f_0}$ intersects $\intV_{p^n \cdot f_0}(\bbZ_p)$. Let $\sigma_1, \dots, \sigma_r \in \intV_{p^n \cdot f_0}(\bbZ_p)$ be representatives for the distinct $G(\bbQ_p)$-orbits in the image of $\eta_{p^n \cdot f_0}$. For each $i = 1, \dots, r$, we can find an open neighbourhood $U'_{p, i} \subset \intV(\bbZ_p)$ of $\sigma_i$ with the following properties:
\begin{enumerate}
	\item The image $\pi(U'_{p, i}) = U_p \subset \intB(\bbZ_p)$ is independent of $i$, and $\pi|_{U'_{p, i}} : U'_{p, i} \to U_p$ is a homeomorphism. 
	\item We have $U_p \subset \intB(\bbZ_p)^\text{rs} = \intB(\bbZ_p) \cap B^\text{rs}(\bbQ_p)$. 
	\item Let $s_i = \pi|_{U'_{p, i}}^{-1}$. Then for each $g \in U_p$, the elements $s_i(g)$ represent the distinct $G(\bbQ_p)$-orbits in the image of $\eta_g$.
\end{enumerate}
This essentially follows from the fact that given $f \in B^\text{rs}(\bbQ_p)$, the action map $G \to V_f$ attached to any $x \in V_f(\bbQ_p)$ is \'etale. After possibly shrinking $U_p$, we can assume that it has the form $p^n \cdot W_p$ for some open compact subset $W_p \subset \intB(\bbZ_p)^\text{rs}$ that contains $f_0$. This completes the proof.
\end{proof}

\begin{corollary}\label{cor_existence_of_integrable_representatives_for_selmer_orbits}
Let $f_0(x) \in \Equations$. Then for each prime $p \leq N$ we can find a compact open neighbourhood $W_p$ of $f_0(x)$ in $\Equations_p$ and an integer $n_p \geq 0$ with the following property. Let $M = \prod_{p \leq N} p^{n_p}$. Then for all $f \in \Equations \cap ( \prod_{p \leq N} W_p)$, and for all $y \in \Sel_3(\Jacobian_{M \cdot f})$, the orbit $\eta_{M \cdot f}(y) \in G(\bbQ) \backslash V_{M \cdot f}(\bbQ)$ contains an element of $\intV_{M \cdot f}(\bbZ)$. 
\end{corollary}
\begin{proof}
We have $G(\bbA^\infty) = G(\bbQ) G(\widehat{\bbZ})$. It follows that for a given element $v \in V(\bbQ)$, finding $g \in G(\bbQ)$ such that $g \cdot v \in \intV(\bbZ)$ is equivalent to finding for each prime $p$ an element $g_p \in G(\bbQ_p)$ such that $g \cdot v \in \intV(\bbZ_p)$. The result therefore follows upon combining Theorem \ref{thm_existence_of_integral_representatives_for_soluble_orbits} and Proposition \ref{prop_weak_existence_of_integral_representatives}.
\end{proof}

\section{Counting points}\label{sec_counting_points}

We retain the notation of \S \ref{sec_stable_grading_of_E8}. In particular, we have a reductive group $\intG$ over $\bbZ$ acting on a free $\bbZ$-module $\intV$, and a $\intG$-equivariant morphism $\pi : \intV \to \intB = \Spec \bbZ[c_{12}, c_{18}, c_{24}, c_{30}]$ (where $\intG$ acts trivially on $\intB$). For $f(x) = x^5 + c_{12} x^3 + c_{18} x^2 + c_{24} x + c_{30} \in \intB(\bbZ)$, we define $\Ht(f) = \sup_i | c_i |^{120 / i}$. If $v \in \intV(\bbZ)$, then we define $\Ht(v) = \Ht(\pi(v))$.

\subsection{Counting points with finitely many congruence conditions}\label{section_counting_finitely_many}

For any $\intG(\bbZ)$-invariant subset $X \subset \intV(\bbZ)$, define
\[ N(X, \htvar) = \sum_{\substack{v \in \intG(\bbZ) \backslash X\\ \Ht(v) < \htvar}} \frac{1}{| Z_{\intG}(v)(\bbZ) |}. \]
Suppose we are given an integer $M \geq 1$ and a $\intG(\bbZ / M \bbZ)$-invariant function $w : \intV(\bbZ / M \bbZ) \to \bbR_{\geq 0}$. We define
\[ N_w(X, \htvar) = \sum_{\substack{v \in \intG(\bbZ) \backslash X\\ \Ht(v) < \htvar}} \frac{w(v \text{ mod }M)}{| Z_{\intG}(v)(\bbZ) |}. \]
We define $\mu_w$ to be the average value of $w$ (with respect to the uniform probability measure on $\intV(\bbZ / M\bbZ)$). 

For a field $k / \bbQ$, we say that $v \in V(k)$ is $k$-reducible if $v$ has zero discriminant or if $v$ is $G(k)$-conjugate to the element $\Kostant(\pi(v)) \in V(k)$ in the image of the Kostant section. Otherwise we say that $v$ is $k$-irreducible. If $X \subset V(\bbQ)$ is any subset, then we write $X^\irr$ for its intersection with the set of $\bbQ$-irreducible elements. 
The first main result of this section concerns the number of $\intG(\bbZ)$-orbits of $\bbQ$-irreducible elements of $\intV(\bbZ)$ of bounded height:
\begin{theorem}\label{thm_point_count}
	We have
	\[ N_w(\intV(\bbZ)^\irr,\htvar) = \frac{2^4|W_0|}{9} \mu_w \vol(\intG(\bbZ) \backslash G(\bbR)) \htvar^{7/10} + o(\htvar^{7/10}), \]
	where $W_0$ denotes the constant of Proposition \ref{prop_integration_in_fibres}.
\end{theorem}
Bhargava, Shankar, and Gross have developed general techniques for proving theorems like Theorem \ref{thm_point_count} when the pair $(G, V)$ is a Vinberg representation arising from a stable grading on a Lie algebra (see for example \cite{Bha15}, \cite{Bha15a}, \cite[Theorem 36]{Bha13} and \cite[\S 3]{Tho15}). These techniques apply equally well here. To avoid repetition, we have chosen not to provide all of the details for the proof: instead we give the key propositions, which can be inserted into the arguments at the appropriate points. In comparing what we prove here with the proof of \cite[Theorem 36]{Bha13}, it's useful to note that because $\Kostant(B^\text{rs}(\bbR))$ contains exactly one representative for each orbit of $G(\bbR)$ on $V(\bbR)^{\text{rs}}$, it may be used to construct a fundamental set for the action of $\bbR_{>0} \times G(\bbR)$ on $V(\bbR)^{\text{rs}}$ (cf. \cite[Section 9.1]{Bha13}), and also that the stabilizer in $G(\bbR)$ of every element in $V(\bbR)^{\text{rs}}$ has order $9$ (because for any $f \in B^\text{rs}(\bbR)$, $\Jacobian_f(\bbR)[3]$ has order 9).

Theorem \ref{thm_point_count} will follow by combining Theorem \ref{prop_killing_cusp_data} and Proposition \ref{prop_bigstab} below. Theorem \ref{prop_killing_cusp_data} is a combinatorial result that allows one to bound the contribution from the cusp region in a fundamental domain for the action of $\intG(\bbZ)$ on $V(\bbR)$; it is the analogue of \cite[Proposition 29]{Bha13}. Proposition \ref{prop_bigstab}  is the analogue of the results in \cite[\S 10.7]{Bha13}.

We start by defining some notation. We write $\Phi_V \subset \Phi_H$ for the set of roots of $H$ that occur as weights for the action of $T$ on $V$. We write $\Phi_V^+$ for $\Phi_H^+ \cap \Phi_V$. We recall that we have fixed a root basis $S_G$ for the group $G$. If $\beta \in S_G$, then we write $\check{\omega}_\beta \in X_\ast(T) \otimes \bbQ$ for the corresponding fundamental coweight. 
We define a partial ordering $\leq_G$ on $X^\ast(T)$ (and hence on $\Phi_V$, by restriction) by $\alpha \leq_G \gamma$ if and only if $\langle \gamma - \alpha,   \check{\omega}_\beta \rangle \in \bbZ_{\geq 0}$ for all $\beta \in S_G$. Let $\cC$ be the collection of nonempty subsets of $\Phi_V$ that are closed under $\leq_G$, i.e. $\cC = \{ M \subset \Phi_V \mid \text{ if } \al \in M, \gamma \in \Phi_V, \text{ and } \al \leq_G \gamma, \text{ then } \gamma \in M \}$. 
Given a vector $v \in V$, we write $v = \sum_{\al \in \Phi_V} v_\al$ for its decomposition as a sum $T$-eigenvectors. Given a subset $M \subset \Phi_V$, we define $V(M) = \{v \in V \mid v_\al = 0 \text{ for all } \al \text{ in } M\}$. 
%The following lemma, which  is a variant of \cite[Proposition 2.15]{Rom18}, gives criteria for the vectors in $\intV(M)$ to be reducible. 

\begin{theorem}\label{prop_killing_cusp_data}
If $M_0 \in \cC$ and $V(M_0)(\bbQ)$ contains $\bbQ$-irreducible vectors, then there exists a subset $M_1 \subset \Phi_V - M_0$ and a function 
$f: M_1 \to \bbR_{\geq 0}$ with $\sum_{\al \in M_1} f(\al) < \lvert M_0\rvert$ and
\begin{equation*}
\langle \sum_{\al \in \Phi_G^+} \al - \sum_{\al \in M_0} \al + \sum_{\al \in M_1} f(\al) \al, \check{\omega}_\beta \rangle > 0
\end{equation*}
for all $\beta \in S_G$.
\end{theorem}
Before giving the proof of Theorem \ref{prop_killing_cusp_data}, we give some useful lemmas.
\begin{lemma}\label{lem-reducible}
	Let $k / \bbQ$ be a field. Given a subset $M \subset \Phi_V$, suppose one of the following three conditions is satisfied:
	\begin{enumerate}
\item We have $\Phi_V^+ - S_H \subset M$.
\item 
There exists a character $\lam \in X^*(T)$ such that if $\al \in \Phi_V$ and $( \lam, \al ) > 0$, then $\al \in M$. Equivalently, there exists a cocharacter $\mu \in X_\ast(T)$ such that if $\alpha \in \Phi_V$ and $\langle \al, \mu \rangle > 0$, then $\al \in M$.
\item For every element $v \in V(M)(k)$, there exists a non-zero nilpotent element $e_v \in \frh_k$ such that $[v, e_v] = 0$.
\item There exists a root $\gamma \in \Phi_H$ such that if $\al \in \Phi_V$ and $\al + \gamma \in \Phi_H$, then $\al \in M$.  
\item There exist $\be \in \Phi_G$ and $\al \in \Phi_V - M$ such that the following conditions hold:
\begin{enumerate}
\item We have $\{\gamma \pm \be \mid \gamma \in M\} \cap \Phi_V \subset M$.
\item $\al - \be \in \Phi_V - M$.
\item Every element of $V(M \cup \{\al\})(k)$ is $k$-reducible.
\end{enumerate}
\end{enumerate}
Then every element of $V(M)(k)$ is $k$-reducible.
\end{lemma}
\begin{proof}
If one of the first two conditions or the fifth condition is satisfied, the fact that the elements of $V(M)(k)$ are $k$-reducible is given by a proof identical to that of \cite[Proposition 2.15]{Rom18}.
For the third item, note that if $v$ is irreducible, then $v$ is regular semisimple, and so every element of the centralizer $\frz_{\frh_k}(v)$ of $v$ in $\frh$ is semisimple.
The fourth condition is a special case of the third one: let $e_\al$ be a non-zero element in $\frh_\al$, and suppose $v \in V(M)(k)$. Then $[e_\al, v] \in \underset{\gamma \in \Phi_V - M}\sum \frh_{\gamma + \al, k} = 0$, so $e_\al \in \frz_{\frh_k}(v)$. Since $e_\al$ is nilpotent, $v$ is not regular semisimple.
% To prove that the fourth criterion implies reducibility, suppose $v \in \intV(M)(k)$. If $v_\al = 0$ or $v_{\al + \be} = 0$, then $v$ is in $\intV(M \cup \{\al\})(k)$ or $\intV(M \cup \{\al + \be\})(k)$ and so is reducible by condition 2 of the lemma. Thus we may assume $v_\al \neq 0$ and $v_{\al + \be} \neq 0$. 
% Let $U_{-\be} \subset G$ be the root subgroup corresponding to $-\be \in \Phi_G$. Note that there exists $u \in U_{-\be}$ such that $u \cdot (v_\al + v_{\al + \be}) = cv_{\al + \be}$ for some constant $c$. By condition (a), we have $u \cdot v \in \intV(M)(k)$, and so by our choice of $u$ we have $u \cdot v \in \intV(M \cup \{\al\})$. Thus $v$ is $k$-reducible as desired.
\end{proof}

In order to prove Theorem \ref{prop_killing_cusp_data}, it is helpful to have an explicit realisation of the Lie algebra $\frh$, together with its grading $\theta$. We may identify $\frg = \frh(0)$ with $\frs\frl_9$, $V = \frh(1)$ with $\wedge^3 \bbQ^9$ and $\frh(2)$ with $\wedge^6 \bbQ^9$, so that
\[ \frh = \frh(0) \oplus \frh(1) \oplus \frh(2) = \frs\frl_9 \oplus \wedge^3 \bbQ^9 \oplus \wedge^6 \bbQ^9. \]
With this identification, the bracket maps $\frh(0) \times \frh(1) \to \frh(1)$ and $\frh(0) \times \frh(2) \to \frh(2)$ are given by the natural action of $\frg$ on $\wedge^3\bbQ^9$ and $\wedge^6\bbQ^9$ respectively, and we may assume (possibly after scaling our basis) that the bracket $\frh(1) \times \frh(1) \to \frh(2)$ is given by wedge product. We may identify $\frt$ with the usual diagonal Cartan subalgebra of $\frs\frl_9$ and $T \subset H$ with the corresponding split maximal torus. Let $e_1, \dots, e_9$ denote the standard basis of $\bbQ^9$, and let $t_i \in X^\ast(T)$ denote the character through which $T$ acts on $e_i$. We write $(ijk) = t_i t_j t_k \in \Phi_V$ (thus the $(i j k)$-weight space in $V$ is spanned by $e_i \wedge e_j  \wedge e_k$). As described in \cite[Section 4.2]{Vin78}, the restriction to $\Phi_V$ of the pairing $(\cdot, \cdot)$ on $\Lambda$ may be described explicitly as follows:
\begin{eqnarray*}
((i j k), (l m n)) =
\begin{cases}
-1 \text{ if } \lvert \{i, j, k\} \cap \{l, m, n\}\rvert = 0\\
0 \text{ if } \lvert\{i, j, k\} \cap \{l, m, n\}\rvert = 1\\
1 \text{ if } \lvert \{i, j, k\} \cap \{l, m, n\}\rvert = 2\\
2 \text{ if } \lvert \{i, j, k\} \cap \{l, m, n\}\rvert = 3.
\end{cases}
\end{eqnarray*}
\begin{lemma}
	The subset $S_0 = \{(2 6 7), (2 5 8), (3 4 8), (1 6 9), (3 5 7), (2 4 9), (1 7 8), (4 5 6)\} \subset \Phi_V$ forms a root basis for $\Phi_H$. Taking $S_H$ to be this choice of root basis, we have $S_G = \{ t_2 / t_1, t_3 / t_2, \dots, t_9 / t_8 \}$.
\end{lemma}
\begin{proof}
Using the explicit description of the pairing $(\cdot, \cdot)$ above, we see that the elements of $S_0$ satisfy the relations for a root basis of $E_8$ (e.g. the pairings between them yield a Dynkin diagram of type $E_8$), and thus $S_0$ forms a root basis.
Alternatively, let $x = 3 \diag(0, 2, 3, 4, 5, 6, 7, 8, 9) - 44/3 \cdot 1_9 \in \frt$ (here $1_9$ denotes the $9 \times 9$ identity matrix). It is easy to see that for all $\alpha \in \Phi_H$, $\alpha(x)$ is a non-zero integer; moreover, $\alpha(x) = 1$ if and only if $\al \in S_0$. Thus $x$ determines a set of positive roots $\Phi_{H, x}^+$ for $\Phi_H$, i.e. $\Phi_{H, x}^+ = \{ \al \in \Phi_H \mid \al(x) > 0\}$. Since $\al(x)$ takes its minimal positive value when $\al \in S_0$, $S_0$ must be the root basis determined by $\Phi_{H, x}^+$. Note that $x$ corresponds to the (differential of) the sum of the fundamental coweights for $H$ with respect to $S_0$. The claimed form of $S_G$ then follows immediately from its definition as the unique root basis such that $\Phi_G^+ = \Phi_H^+ \cap \Phi_G$.
\end{proof}
We now assume that $S_H$ has been chosen to be the set $S_0$ described in the statement of the lemma. Note that $(i j k) \leq_G (l m n)$ if and only if $i \leq l, j \leq m, k \leq n$. Thus $(7 8 9)$ is the highest weight of $V$, $(1 2 3)$ is the lowest weight, and  
$\Phi_V^+$ is given by $\{\al \mid \gamma \leq_G \al \text{ for some } \gamma \in S_H\}$. We define $\beta_i = t_{i+1} / t_{i}$, so that $S_G = \{ \beta_1, \dots, \beta_8 \}$, and for $\al \in X^\ast(T)$, we  set $n_i(\alpha) = \langle \alpha, \check{\omega}_{\beta_i} \rangle$.
\begin{proposition}\label{lem-simple-not-stable}
Let $k / \bbQ$ be a field, and let $M \in \cC$ be nonempty. If any of the following conditions holds, then every element of $V(M)(k)$ is $k$-reducible.
\begin{enumerate}
\item 
$\{(2 6 7), (2 5 8), (3 4 8), (3 5 7), (1 7 8), (4 5 6) \} \cap M \neq \emptyset$.
\item $M \cap (\Phi_V^- - \{(1 5 9)\}) \neq \emptyset$.
\item $\{(1 6 9), (2 6 8) \} \subset M$.
\item $\{(159), (5 6 7)\} \subset M$.
\item $\{ (1 6 9), (3 4 9), (3 6 7) \} \subset M$.
\item $\{ (1 7 9), (2 4 9), (4 5 7) \} \subset M$.
\item $\{(1 6 9), (2 4 9), (2 7 8), (4 6 7)\} \subset M$.
\end{enumerate}
\end{proposition}

\begin{proof}

Case 1: $\{(2 6 7), (2 5 8), (3 4 8), (3 5 7), (1 7 8), (4 5 6) \} \cap M \neq \emptyset$.

First suppose $(267) \in M$. Let $\lam = - (134) - (125) \in X^*(T)$. Note that if $\al = (i j k) \in \Phi_V$ and $(\lam, \al) > 0$, then $(267) \leq_G \al$. Indeed, if $(\lam, \al) > 0$, then we must have  $((1 3 4), \al) < 0$ or $((1 2 5), \al) < 0$. By the definition of the pairing, this implies that $\{1, 3, 4\} \cap \{i, j, k\} = \emptyset$ or $\{1, 2, 5\} \cap \{i, j, l \}= \emptyset$. In particular, $i \geq 2$. Suppose $i = 2$. Then if $j = 5$, we have $(\lam, \al) = 1 - 1 = 0$, so $j \geq 6$ and thus $(2 6 7) \leq_G \al$. If $i = 3$, similar logic shows $(2 6 7) \leq_G \al$. Otherwise $i \geq 4$, and in this case it is easy to see that $(2 6 7) \leq_G \al$. Thus by the second part of Lemma \ref{lem-reducible}, every element of $V(M)(k)$ is $k$-reducible.

Next suppose $(1 7 8) \in M$, and let $\gamma = -(7 8 9)$. Suppose $\al = (i j k) \in \Phi_V$ and $\al + \gamma \in \Phi_H$. Then $(\al, \gamma) = -1$, so $\{i, j, k\} \cap \{7, 8, 9\} = 2$, and thus $(1 7 8) \leq_G \al$. In particular, we have $\al \in M$, and so the result follows from part 4 of Lemma \ref{lem-reducible}. Similarly, if $(4 5 6) \in M$, then for $\gamma = (1 2 3)$ and $\al \in \Phi_V$, if $\al + \gamma \in \Phi_H$ then $(4 5 6) \leq_G \al$, so the result follows by the same logic.

Suppose $(3 5 7) \in M$. Note that if $\al = (i j k) \in \Phi_V - M$ and $\al + (1 2 3) \in \Phi_H$, then $\al = (4 5 6)$ (indeed, since $\{i, j, k\} \cap \{1, 2, 3\} = \emptyset$, we must have $i \geq 4$ and $j \geq 5$, but since $\al \notin M$, we must have $k \leq 6$) and $\al + (1 2 3) = -(7 8 9)$. Similarly, if $\al + (1 2 4) \in \Phi_H$, then $\al = (3 5 6)$ and $\al + (1 2 4) = -(7 8 9)$. Now suppose $v \in V(M)(k)$. 
Then the above analysis shows that $[v, e_1 \wedge e_2 \wedge e_3]$ and $[v, e_1 \wedge e_2 \wedge e_4]$ are both in the span of $e_1 \wedge e_2 \wedge e_3 \wedge e_4 \wedge e_5 \wedge e_6$, 
so span$\{e_1 \wedge e_2 \wedge e_3, e_1 \wedge e_2 \wedge e_4\}$ contains a non-zero element of the centralizer $\frz_{\frh}(v)$. But since $(1 2 3), (1 2 4) \in \Phi_V^-$, every element of $\Span\{e_1 \wedge e_2 \wedge e_3, e_1 \wedge e_2 \wedge e_4\}$ is nilpotent, so the result follows from Lemma \ref{lem-reducible} part 3.

If $(3 4 8) \in M$, then using similar logic as in the previous case, we see that for any element $v \in V(M)(k)$, $\ad v$ maps $\mathbb{U}_1 := \Span\{e_1 \wedge e_2 \wedge e_k \mid 3 \leq k \leq 7\}$ to $\mathbb{U}_2 := \Span\{e_1 \wedge e_2 \wedge e_3 \wedge \dots \wedge \hat{e}_l \wedge \dots \wedge e_7 \mid 3 \leq l \leq 7\}$ (where the hat denotes omission). Thus $\ad v$ gives a linear map $\mathbb{U}_1 \to \mathbb{U}_2$. If $\mathbb{W}_5$ denotes the 5-dimensional $k$-vector space with basis $e_3, \dots, e_7$, then under the natural isomorphisms $\mathbb{U}_1 \simeq \mathbb{W}_5$ and $\mathbb{U}_2 \simeq \wedge^4 \mathbb{W}_5$, the map $\ad v|_{\mathbb{U}_1}$ corresponds to a map $\mathbb{W}_5 \to \wedge^4 \mathbb{W}_5$ of the form $x \mapsto \omega \wedge x$, for some $\omega \in \wedge^3 \mathbb{W}_5$. Regardless of the choice of $\omega$, such a map has a non-trivial kernel. Since every element in $\mathbb{U}_1$ is nilpotent, the result follows from Lemma \ref{lem-reducible} part 3.

If $(2 5 8) \in M$, then any element $v \in V(M)(k)$ admits a decomposition 
\[ v = e_1 \wedge \omega_1 + \omega_2 + \omega_3, \]
where $\omega_1 \in \wedge^2 \Span\{ e_2, \dots, e_9 \}$, $\omega_2 \in \wedge^3 \Span \{ e_2, \dots, e_7 \}$, and $\omega_3 \in \wedge^2 \Span \{ e_2, e_3, e_4 \} \wedge \Span \{e_8, e_9 \}$. After acting by an element of $G(k)$, we can assume moreover that $\omega_3 \in \Span \{ e_2 \wedge e_3, e_2 \wedge e_4 \} \wedge \Span \{ e_8, e_9 \}$. The previous paragraph now shows that $v$ is $k$-reducible (as $v \in V(M')$ for any $M'$ that contains $(348)$).

Case 2: $M \cap (\Phi_V^- - \{(1 5 9)\}) \neq \emptyset$.

Now suppose $M$ contains an element $\al \in \Phi_V^- - \{(1 5 9)\}$. Then there exists an element $\al_i \in S_H = \{(2 6 7), (2 5 8), (3 4 8), (1 6 9), (3 5 7), (2 4 9), (1 7 8), (4 5 6)\}$ such that $\al \leq_G \al_i$, which implies $\al_i \in M$. By the first part of the proposition, we may assume that $\al_i \in \{(1 6 9), (2 4 9)\}$. If $(2 4 9) \notin M$, then without loss of generality we may assume $\al = (1 6 8)$. Since $\gamma \in \Phi_V$ and $\gamma - (6 8 9) \in \Phi_H$ implies $(1 6 8) \leq_G \gamma$, in this case the result follows from part 4 of Lemma \ref{lem-reducible}. 

Otherwise $(2 4 9) \in M$, and without loss of generality we may assume $\al \in \{(2 4 8), (239), (149)\}$. If $\al = (248)$, then $(3 4 8) \in M$ and the result follows from the first part of the lemma. Suppose $\al = (2 3 9) \in M$. Note that $((123), (239)) = 1$, so $(123) - (239) \in \Phi_H$. If $\gamma \in \Phi_V$ and $\gamma + (123) - (239) \in \Phi_H$, then $(2 3 9) \leq_G \gamma$, and the result again follows from part 4 of Lemma \ref{lem-reducible}. Lastly, suppose $(1 4 9) \in M$. Note that $\gamma_1 := (135) - (159) \in \Phi_H^-$ and $\gamma_2 := (126) - (169) \in \Phi_H^-$. If $v \in V(M)(k), v_1 \in \frh_{\gamma_1}$, and $v_2 \in \frh_{\gamma_2}$, then $[v, v_1]$ and $[v, v_2]$ are both in $\Span\{e_1 \wedge e_2 \wedge e_3\}$. Thus $v$ is centralized by a non-zero nilpotent element, and the result follows from part 4 of Lemma \ref{lem-reducible}.

Case 3: $\{(1 6 9), (2 6 8) \} \subset M$.

If $\{(1 69), (268)\} \subset M$ and $v \in V(M)(k)$, then $[v, e_1 \wedge e_2 \wedge e_3 \wedge e_4 \wedge e_5 \wedge e_7]$ and $[v, e_1 \wedge e_2 \wedge e_3 \wedge e_4 \wedge e_5 \wedge e_6]$ both lie in the root space $\frh_{\alpha, k}$, where $\alpha = t_1 / t_9$. Some non-zero linear combination of $e_1 \wedge e_2 \wedge e_3 \wedge e_4 \wedge e_5 \wedge e_7$ and $e_1 \wedge e_2 \wedge e_3 \wedge e_4 \wedge e_5 \wedge e_6$ is then a nilpotent element that centralizes $v$.

Case 4: $\{(159), (5 6 7)\} \subset M$.

Let $\lam = -(1 2 3) - (1 4 6) + (1 6 9)$. Then if $\al \in \Phi_V$ and $(\lam, \al) > 0$, we have $(15 9) \leq_G \al$ or $(5 6 7) \leq_G \al$, so the proof follows from part 2 of Lemma \ref{lem-reducible}. 

Case 5: $\{ (1 6 9), (3 4 9), (3 6 7) \} \subset M$.

The proof also follows from part 2 of Lemma \ref{lem-reducible} by taking $\lam = - (1 2 3) + (7 8 9) + (3 6 9)$. 

Case 6: $\{ (1 7 9), (2 4 9), (4 5 7) \} \subset M$.

If $v \in V(M)(k)$, then $\Span\{[v, e_1 \wedge e_2 \wedge e_3 \wedge e_4 \wedge e_6 \wedge e_8], [v, e_1 \wedge e_2 \wedge e_3 \wedge e_4 \wedge e_5 \wedge e_8], [v, e_1 \wedge \dots \wedge \hat{e}_k \wedge \wedge e_7] \mid 4 \leq k \leq 7\}$ is contained in a 5-dimensional subspace of $\frh_k$, namely the span of the root vectors corresponding to the roots $t_1 / t_7$, $t_1 / t_8$, $t_1 / t_9$, $t_2 / t_9$, and $t_3 / t_9$, and so $v$ is centralized by a non-zero nilpotent element.

Case 7: $\{(1 6 9), (2 4 9), (2 7 8), (4 6 7)\} \subset M$.

Let $M'$ be the minimal element of $\cC$ containing $S := \{(169), (249), (278), (467)\}$, i.e. $M' = \{ \gamma \in \Phi_V \mid \al \leq_G \gamma \text{ for some } \al \in S\}$. Note that $V(M) \subset V(M')$, so it suffices to prove that every element of $V(M')(k)$ is $k$-reducible. We use condition 5 of Lemma \ref{lem-reducible} with $\al = (15 9)$ and $\be = (1 5 9) - (1 4 9)$.  By definition of $\cC$, $\{\gamma + \be \mid \gamma \in M'\} \cap \Phi_V \subset M'$. Suppose $\gamma = (i j k) \in M'$ and $\gamma - \be \in \Phi_V$. Then $5 \in \{i, j, k\}$. If $i = 5$, then $\gamma \geq_G (5 6 7)$, so $\gamma - \be \geq_G (4 6 7)$ and thus $\gamma - \be \in M'$. If $j = 5$, then $\gamma \geq_G (2 5 9)$, so $\gamma - \be_G \geq (2 4 9)$. Note that $k \neq 5$ since no weights of the form $(i j 5)$ are in $M'$. Thus conditions (a) and (b) of Lemma \ref{lem-reducible} part 5 are satisfied. Condition (c) is satisfied by Part 4 of the proposition.

\end{proof}

\begin{lemma}\label{lem_cusp_induction}
Suppose that $M_0', M_0'' \in \cC$ with $M_0'' \subset M_0'$, that $M_1' \subset \Phi_V - M_0'$, and that there exists a function $f': M_1' \to \bbR_{\geq 0}$ satisfying the conditions of Theorem \ref{prop_killing_cusp_data} for $M_0'$. If there exists a function $g: (M_0' - M_0'') \to M_1'$ such that
\begin{enumerate}
\item $\al - g(\al) \in \Phi_G^+$ for all $\al \in M_0' - M_0''$ and
\item $f'(\al) - \lvert g^{-1}(\al)\rvert \geq 0$ for all $\al \in M_1'$
\end{enumerate}
then Theorem \ref{prop_killing_cusp_data} holds for any $M_0 \in \cC$ such that $M_0'' \subset M_0 \subset M_0'$. 
\end{lemma}

\begin{proof}
Given such a subset $M_0$, define $f: M_1' \to \bbR_{\geq 0}$ by $f(\al) = f'(\al) - \lvert g^{-1}(\al) \cap (M_0'- M_0)\rvert$. We have
\begin{eqnarray*}
\sum_{\al \in M_1'} f(\al) &=& \sum_{\al \in M_1'} f'(\al) - \lvert M_0' - M_0\rvert\\
&<& \lvert M_0' \rvert -  \lvert M_0' - M_0 \rvert = \lvert M_0\rvert.
\end{eqnarray*}
For each $\beta \in S_G$, we have
\begin{eqnarray*}
\langle \sum_{\al \in \Phi_G^+} \al - \sum_{\al \in M_0} \al + \sum_{\al \in M_1'} f(\al)\al, \check{\omega}_\beta \rangle &=&
\langle \sum_{\al \in \Phi_G^+} \al - \sum_{\al \in M_0'} \al + \sum_{\al \in M_0'- M_0} \al\\& & + \sum_{\al \in M_1'} f'(\al)\al - \sum_{\al \in M_0'- M_0} g(\al) , \check{\omega}_\beta \rangle\\
&>& \langle \sum_{\al \in M_0'- M_0} \al - \sum_{\al \in M_0'- M_0}g(\al), \check{\omega}_\beta \rangle \\
&=& \langle \sum_{\al \in M_0'- M_0} (\al - g(\al)), \check{\omega}_\beta \rangle,
\end{eqnarray*}
which is greater than 0 because $\al - g(\al) \in \Phi_G^+$ for all $\al \in M_0'- M_0$.
\end{proof}

\begin{lemma}\label{lem_M0_leq_10}
If $\lvert M_0\rvert \leq 10$, then $\langle \sum_{\al \in \Phi_G^+} \al -  \sum_{\al \in M_0} \al, \check{\omega}_\beta \rangle > 0$ for all $\beta \in S_G$. In particular, given any subset $M_1 \subset \Phi_V - M_0$, the function $f(\al) = 0$ for all $\al \in M_1$ satisfies the conditions of Theorem \ref{prop_killing_cusp_data} for $M_0$.
\end{lemma}

\begin{proof}
Note that
\begin{eqnarray*}
\sum_{\be \in \Phi_G^+} \be 
&=& 8\be_1 + 14\be_2 + 18\be_3 + 20\be_4 + 20\be_5 + 18\be_6 +  14\be_7 + 8\be_8.
\end{eqnarray*}
It suffices to show that if $\lvert M_0 \rvert \leq 10$, then $\sum_{\al \in M_0} n_i(\al) < n_i(8\be_1 + 14\be_2 + 18\be_3 + 20\be_4 + 20\be_5 + 18\be_6 +  14\be_7 + 8\be_8)$ for all $i$. Note that the highest weight of $V$ is
\begin{equation*}
(789) = \frac{1}{3}(\be_1 + 2\be_2 + 3\be_3 + 4\be_4 + 5\be_5 + 6 \beta_6 + 4\be_7 + 2\be_8).
\end{equation*}
So we have $(789) \in M_0$ and if $\al \in M_0$, then $\al \leq_G (789)$. If $\lvert M_0 \rvert = 1$, then the lemma is obvious. So we may assume $\lvert M_0 \rvert \geq 2$, and so $(689) = (789) - \be_6 \in M_0$, and for all $\al \in M_0$ with $\al \neq (789)$, we have $\al \leq_G (689)$. Thus
\begin{equation*}
n_i(\sum_{\al \in M_0} \al) \leq \lvert M_0 \rvert n_i((789)) - (\lvert M_0\rvert - 1)n_i(\be_6).
\end{equation*}
This implies the lemma. 
%Sanity check: for $i = 1, 2, 3, 4, 5, 8$, this is obvious. For $i = 6$ the above inequality becomes 
%
%$n_6(\sum_{\al \in M_0} \al) \leq (\# M_0)(2) - (\# M_0 -1) = \# M_0 + 1$
%
%For $i = 7$,  we have $\frac{4}{3}(\# M_0) < 14$. 
\end{proof}

\begin{proposition}\label{prop_killing_subsets_pos}
If $M_0 \in \cC$ and $M_0 \subsetneq \Phi_V^+ - S_H$, then there exists a function $f: S_H \to \bbR_{\geq 0}$ satisfying the conditions of Theorem \ref{prop_killing_cusp_data} for $M_0$.
\end{proposition}

\begin{proof}
It's not hard to check that if $\lvert M_0\rvert \geq 10$, then $M_0'' := \{(7 8 9), (6 8 9), (5 8 9), (6 7 9), (5 7 9), (4 8 9)\} \subset M_0$. Thus by Lemma \ref{lem_M0_leq_10} we may assume $M_0'' \subset M_0$. 
%By Lemma \ref{lem_cusp_induction} it suffices to assume $M_0 = \Phi_V^+\setminus(S_H \cup \{\al\})$ for some $\al \in \Phi_V$. 
Let
\begin{eqnarray*}
\gamma_1 &=& (2 6 8)\\
\gamma_2 &=& (367)\\
\gamma_3 &=& (35 8)\\
\gamma_4 &=& (2 5 9)\\
\gamma_5 &=& (3 4 9)\\
\gamma_6 &=& (1 7 9)\\
\gamma_7 &=& (4 5 7)
\end{eqnarray*}

Since $M_0 \in \cC$, we have $M_0 \subseteq \Phi_V^+ - (S_H \cup \{\gamma_i \})$ for some $i \in \{1, ..., 7\}$. 
Define $g_0: [\Phi_V^+ - (S_H \cup M_0'')] \to S_H$ as in Table \ref{table-g} and $f_i: S_H \to \bbR_{\geq 0}$ ($i \in \{1, ..., 7\}$) as follows:

\begin{tabular}{|c | c | c | c | c | c | c | c | c |}
\hline
$\al$ & $\lvert g_0^{-1}(\al) \rvert$ & $f_1(\al)$ & $f_2(\al)$ & $f_3(\al)$ & $f_4(\al)$ & $f_5(\al)$ & $f_6(\al)$ & $f_7(\al)$\\
\hline
(2 6 7) & 2 & 1041/512 & {1041}/{512} & {1553}/{512} & {1553}/{512} &{1553}/{512} & {1553}/{512} & 1553/{512} \\
(2 5 8) & 3 & {1573}/{512} & {2085}/{512} & {1573}/{512} & {1573}/{512} & {2085}/{512} &  {2085}/{512}  & {2089}/{512} \\
(3 4 8) & 5 & {2709}/{512} & {2709}/{512} & 2709/{512} & {3221}/{512} & {2709}/{512} & {3221}/{512} & {3157}/{512}\\
(1 6 9) & 7 & {3767}/{512} & {3767}/{512} & {3767}/{512} & {3767}/{512} & {3767}/{512} & {3767}/{512} & {4215}/{512}\\
(3 5 7) & 6 & {3755}/{512} & {3243}/{512} & {3243}/{512} & {3243}/{512} & {3243}/{512} & {3243}/{512} & {3179}/{512}\\
(2 4 9) & 4 & {2635}/{512} &  {2635}/{512} &  {2635}/{512} & {2123}/{512} &  {2123}/{512} &  {2123}/{512} & {2315}/{512}\\
(1 7 8) & 2 & {137}/{64} &  {137}/{64} &  {137}/{64} &  {137}/{64} &  {137}/{64} & {73}/{64} & {97}/{64}\\
(4 5 6) & 1 & {9}/{8} & {9}/{8} & {9}/{8} & {9}/{8} & {9}/{8} & {9}/{8} & {1}/{2}\\
\hline
\end{tabular}

First let $j \in \{1, ..., 6\}$, let $M_0' = \Phi_V^+ - (S_H \cup \{\gamma_i \})$ and suppose $M_0 \subseteq M_0'$. Define $g$ to be the restriction of $g_0$ to $M_0' - M_0''$. Then $f' = f_i$ and $g$ satisfy the conditions of Lemma \ref{lem_cusp_induction} for $M_1' = S_H$, so the proposition is true for $M_0$. 

Now let $M_0' = \Phi_V^+ - (S_H \cup \{\gamma_7 \})$ and suppose $M_0 \subseteq M_0'$. Define $g: (M_0'- M_0'') \to S_H$ by
\begin{eqnarray*}
g(\al) &=&
\begin{cases}
g_0(\al) & \text{ if } \al \neq (1 7 9)\\
(1 6 9) & \text{ if } \al = (1 7 9).
\end{cases}
\end{eqnarray*}
Then $f' = f_7$ and $g$ satisfy the conditions of Lemma \ref{lem_cusp_induction} for $M_1' = S_H$, so the proposition is proven.
\end{proof}

\begin{table}[h]
\centering
\begin{tabular}{|c | c | c | c | c | c | c | c | c |}
\hline
 $\al$ & $g_0(\al)$ & $g_1(\al)$ & $g_{2, 1}(\al)$ & $g_{2, 2}(\al)$ & $g_{3, 1}(\al)$ & $g_{3, 2}(\al)$ & $g_{4, 1}(\al)$ & $g_{4, 2}(\al)$\\
\hline
(2 4 9) & & (2 3 9) & & & & & & \\
\hline
(2 6 8) & (2 6 7) & & & & (2 6 7) & (2 6 7) & & \\ 
\hline
 (3 6 7) & (3 5 7)& & & & (2 6 7) & (2 6 7) & & (3 5 7)\\
 \hline
 (3 5 8) & (3 4 8) & (3 4 8) & (3 4 8) & (3 4 8) & (3 4 8) & (3 4 8) & (3 4 8) & (3 5 7)\\ 
 \hline
 (2 5 9) & (2 5 8) & & & & & & (2 4 9) & (1 5 9)\\
 \hline
 (3 4 9) & (2 4 9) & (1 4 9) & & & & & (2 4 9) & \\
 \hline
 (1 7 9) & (1 7 8) & & & & & (1 6 9) & &\\
 \hline
 (4 5 7) & (4 5 6) & & & & (4 5 6) & & (4 5 6) & (4 5 6)\\
 \hline
 (3 6 8) & (3 4 8) & (2 6 8) & (3 4 8) & (3 4 8) & (3 4 8) & (3 4 8) & (3 6 7) & (2 6 8)\\ 
 \hline
 (2 6 9) & (1 6 9) & & & & & & & \\
 \hline
 (2 7 8) & (2 5 8) & (2 6 8) & & (2 6 8) & (2 5 8) & (2 5 8) & (2 6 8) & (2 6 8)\\
 \hline
 (4 6 7) & (3 5 7) & & (3 6 7) & & (4 5 6) & (4 5 7) & (3 6 7) & (3 5 7)\\
 \hline
 (3 5 9) & (3 5 7) & & & & & & (1 5 9) & (1 5 9)\\
 \hline
 (1 8 9) & (1 6 9) & & & & (1 7 9) & (1 6 9) & &\\
  \hline
 (4 5 8) & (3 4 8) & (3 4 8) & (3 4 8) & (3 4 8) & (3 4 8) & (4 5 7) & (3 4 8) & (4 5 6)\\
  \hline
 (3 6 9) & (1 6 9) & & & & & & & \\
  \hline
 (3 7 8) & (3 4 8) & (2 6 8) & (2 7 8) & (3 4 8) & (3 4 8) & (3 4 8) & (3 6 7) & (2 6 8)\\ 
  \hline
 (4 6 8) & (3 4 8) & (2 6 8) & (3 6 7) & (3 4 8) & (3 4 8) & (4 5 7) & (3 6 7) & (2 6 8)\\
  \hline
 (2 7 9) & (2 6 7) & & & & & & &\\ 
  \hline
 (5 6 7) & (3 5 7) & & (3 6 7) & (4 6 7) & (4 5 6) & (4 5 7) & (3 6 7) & (3 5 7)\\
  \hline
 (4 5 9) & (2 4 9) & & & & & & (1 5 9) & (3 4 9)\\
  \hline
 (2 8 9) & (2 4 9) & & & & & & &\\ 
  \hline
 (5 6 8) & (2 5 8) & (5 6 7) & (3 6 7) & (4 6 7) & (4 5 6) & (4 5 7) & (3 6 7) & (2 6 8)\\
  \hline
 (3 7 9) & (1 6 9) & & & & & & &\\ 
  \hline
 (4 6 9) & (1 6 9) & & & & & & & \\
  \hline
 (4 7 8) & (3 5 7) & (2 6 8) & (2 7 8) & (4 6 7) & (3 4 8) & (4 5 7) & (3 6 7) & (2 6 8)\\ 
  \hline
 (3 8 9) & (1 6 9) &  & & & & & &\\
  \hline
 (5 6 9) & (1 6 9) & & & & & & &\\ 
  \hline
 (5 7 8) & (3 5 7) & (5 6 7) & (2 7 8) & (4 6 7) & (4 5 6) & (3 4 8) & (1 7 8) & (2 6 8)\\
  \hline
 (4 7 9) & (2 4 9) & & & & & & &\\ 
  \hline
 (6 7 8) & (1 7 8) & (5 6 7) & (2 7 8) & (2 6 8) & (3 5 7) & (3 4 8) & (1 7 8) & (1 7 8)\\
\hline
\end{tabular}
\caption{\label{table-g} Maps $(M_0' - M_0'') \to M_1'$}
\end{table}

\begin{proof}[Proof of Theorem \ref{prop_killing_cusp_data}]
Note that by 1 and 2 of Proposition \ref{lem-simple-not-stable}, it suffices to consider $M_0 \subset (\Phi_V^+ - S_H) \cup \{(1 6 9), (2 4 9), (1 5 9)\}$, and by Proposition \ref{prop_killing_subsets_pos} we may assume $M_0 \cap \{(1 6 9), (2 4 9), (1 5 9)\} \neq \emptyset$. We now break into cases according to the value of $\{(1 6 9), (2 4 9), (1 5 9)\} \cap M_0$. In each case we assume that that $V(M_0)(\bbQ)$ contains $\bbQ$-irreducible elements, and identify the following data:
\begin{itemize}
\item subsets $M_0'$ and $M_0''$ such that $M_0'' \subset M_0 \subset M_0'$
\item a subset $M_1' \subset \Phi_V - M_0'$
\item a function $f': M_1' \to \bbR_{\geq 0}$ satisfying the conditions of Theorem \ref{prop_killing_cusp_data} for $M_0'$
\item a function $g: (M_0'- M_0'') \to M_1'$ satisfying the conditions of Lemma \ref{lem_cusp_induction}.
\end{itemize}

For the reader's convenience, we also give $\lvert g^{-1}(\al) \rvert$.

Case 1: $(1 5 9) \in M_0$.

By Proposition \ref{lem-simple-not-stable} part 4, we see that $(567) \notin M_0$, and thus because $M_0 \in \cC$, we have $\{(3 6 7), (4 6 7), (4 5 7) \} \cap M_0 = \emptyset$. Since $M_0 \in \cC$, we have $(1 6 9) \in M_0$, and so by Proposition \ref{lem-simple-not-stable} part 3 we have $(2 6 8) \notin M_0$. Thus we may assume $M_0 \subset M_0' := [\Phi_V^+ - (S_H \cup \{(2 6 8), (3 6 7), (4 6 7), (4 5 7), (5 6 7)\})]\cup \{(1 6 9), (2 4 9), (1 5 9)\}$. By assumption $M_0'' := \{\al \in \Phi_V \mid (1 5 9) \leq_G \al\} \subset M_0$. Let $M_1' = \{(3 4 8), (1 7 8), (2 6 8), (5 6 7), (2 3 9), (1 4 9)\}$. Define $g_1: (M_0'- M_0'') \to M_1'$ as in Table \ref{table-g} and define $f': M_1' \to \bbR_{\geq 0}$ as follows:

\begin{center}
\begin{tabular}{| c | c | c | c | c | c | c |}
\hline
$\al$ & (3 4 8) & (1 7 8) & (2 6 8) & (5 6 7) & (2 3 9) & (1 4 9)\\
\hline
$f'(\al)$ & 53/16 & 35/16 & 5 & 103/16 & 8 & 121/16\\
\hline
$\lvert g^{-1}(\al)\rvert$ & 2 & 0 & 5 & 3 & 1 & 1\\
\hline
\end{tabular}
\end{center}

% \begin{eqnarray*}
% f'((3 4 8) &=& \frac{53}{16}\\
% f'((1 7 8)) &=& \frac{35}{16}\\
% f'((2 6 8)) &=& 5\\
% f'((5 6 7)) &=& \frac{103}{16}\\
% f'((2 3 9)) &=& 8\\
% f'((1 4 9)) &=& \frac{121}{16}
% \end{eqnarray*}

Then $g = g_1$ and $f'$ satisfies the conditions of Lemma \ref{lem_cusp_induction}, so the theorem is proven for $M_0$. 

Case 2: $(1 6 9), (2 4 9) \in M_0$, $(1 5 9) \not\in M_0$.

By Proposition \ref{lem-simple-not-stable}, we see that $M_0 \subset [\Phi_V^+ - (S_H \cup \{(2 6 8), (3 6 7), (4 5 7)\})] \cup\{(1 6 9), (2 4 9)\}$, and by assumption $M_0'' : = \{\al \in \Phi_V \mid (1 6 9) \leq_G \al \text{ or } (2 4 9) \leq_G \al\} \subset M_0$. 

Case 2.1. First suppose $(2 7 8) \notin M_0$, so $M_0 \subset M_0' := [\Phi_V^+ - (S_H \cup \{(2 6 8), (2 7 8), (3 6 7), (4 5 7)\})] \cup\{(1 6 9), (2 4 9)\}$. Let $g: (M_0'- M_0'') \to \{(3 4 8), (3 6 7), (4 5 7), (2 7 8), (1 5 9), (2 3 9)\}$ be given by $g_{2, 1}$ as in Table \ref{table-g} and $f'$ as below to satisfy the conditions of Lemma \ref{lem_cusp_induction}.

\begin{center}
\begin{tabular}{| c | c | c | c | c | c | c |}
\hline
$\al$ & (3 4 8) & (3 6 7) & (4 5 7) & (2 7 8) & (1 5 9) & (2 3 9)\\
\hline
$f'(\al)$ & 201/64 & 533/128 & 553/128 & 815/128 & 41/4 & 545/128\\
\hline
$\lvert g^{-1}(\al)\rvert$ & 3 & 4 &  0 & 4 & 0 & 0\\
\hline
\end{tabular}
\end{center}

% 
% \begin{eqnarray*}
% f'((3 4 8)) &=& \frac{201}{64}\\
% f'((3 6 7)) &=& \frac{533}{128}\\
% f'((4 5 7)) &=& \frac{533}{128}\\
% f'((2 7 8)) &=& \frac{815}{128}\\
% f'((1 5 9)) &=& \frac{41}{4}\\
% f'((2 3 9)) &=& \frac{545}{128}
% \end{eqnarray*}

Case 2.2. Now suppose $(2 7 8) \in M_0$. By Proposition \ref{lem-simple-not-stable} we see that $ (4 6 7) \not\in M_0 = \emptyset$. 
Thus we let $M_0' = [\Phi_V^+ - (S_H \cup \{(2 6 8), (3 6 7), (4 5 7), (4 6 7)\})] \cup\{(1 6 9), (2 4 9), (3 5 8)\}$; let $M_1' =  \{(1 7 8), (2 6 8), (3 4 8), (4 6 7), (1 5 9), (2 3 9)\}$; let $g = g_{2, 2}$ as defined in Table \ref{table-g}; and define $f'$ as below.

\begin{center}
\begin{tabular}{| c | c | c | c | c | c | c |}
\hline
$\al$ & (1 7 8) & (2 6 8) & (3 4 8) & (4 6 7) & (1 5 9) & (2 3 9)\\
\hline
$f'(\al)$ & 9/8 & 39/8 & 5 & 47/8 & 67/8 & 23/4\\
\hline
$\lvert g^{-1}(\al)\rvert$ & 0 & 2 & 5 & 4 & 0 & 0\\
\hline
\end{tabular}
\end{center}

% \begin{eqnarray*}
% f((1 7 8)) &=& \frac{37}{32}\\
% f((2 6 8)) &=& \frac{7}{2}\\
% f((3 5 8)) &=& 4\\
% f((4 6 7)) &=& \frac{99}{16}\\
% f((1 5 9)) &=& \frac{227}{32}\\
% f((2 3 9)) &=& \frac{105}{16}
% \end{eqnarray*}

Case 3: $(2 4 9) \in M_0$, $(1 5 9), (1 6 9) \notin M_0$.

We let $M_0'' = \{\al \in \Phi_V \mid (2 4 9) \leq_G \al\}$. Note that by Proposition \ref{lem-simple-not-stable} part 6 if $\{(1 7 9), (4 5 7)\} \subset M_0$, then every element of $V(M_0)(k)$ is $k$-reducible. Thus we are in one of the following two cases.

Case 3.1. Let $M_0' = [\Phi_V^+ - (S_H \cup \{(179)\})] \cup \{ (2 4 9) \}$, and assume $M_0 \subset M_0'$. Let 
\[ M_1' = \{(2 6 7), (2 5 8), (3 4 8), (3 5 7), (4 5 6), (1 7 9), (2 3 9)\}. \]
 Let $f'$ be defined as below and let $g = g_{3, 1}$ as defined in Table \ref{table-g}.

\begin{center}
\begin{tabular}{| c | c | c | c | c | c | c | c |}
\hline
$\al$ & (2 6 7) & (2 5 8) & (3 4 8) & (3 5 7) & (4 5 6) & (1 7 9) & (2 3 9)\\
\hline
$f'(\al)$ & {27}/{8} & 73/16 & 207/32 & 33/32 & 177/32 & 31/4 & 137/32\\
\hline
$\lvert g^{-1}(\al)\rvert$ & 2 & 1 & 6 & 1 & 5 & 1 & 0\\
\hline
\end{tabular}
\end{center}

Case 3.2 Let $M_0' = \Phi_V^+ - (S_H \cup \{(4 5 7)\}) \cup \{( 2 4 9)\}$, and assume $M_0 \subset M_0'$. Let 
\[ M_1' = \{(2 6 7), (2 5 8), (3 4 8), (1 6 9), (1 7 8), (4 5 7), (2 3 9)\}. \]
Let $f'$ be defined as follows; and let $g = g_{3, 2}$ as defined in Table \ref{table-g}.

\begin{center}
\begin{tabular}{| c | c | c | c | c | c | c | c |}
\hline
$\al$ & (2 6 7) & (2 5 8) & (3 4 8) & (1 6 9) & (1 7 8) & (4 5 7) & (2 3 9)\\
\hline
$f'(\al)$ & {785/256} & 1061/256 & 1621/256 & 2167/256 & 33/32 & 1643/256 & 1291/256\\
\hline
$\lvert g^{-1}(\al)\rvert$ & 2 & 1 & 5 & 2 & 0 & 6 & 0\\
\hline
\end{tabular}
\end{center}

Case 4: $(1 6 9) \in M_0; (2 4 9), (1 5 9) \notin M_0$.

Let $M_0'' = \{\al \in \Phi_V \mid (1 6 9) \leq_G \al\}$. Note that if $\{(3 6 7), (3 4 9)\} \subset M_0$, or if $(268) \in M_0$, then every element of $V(M_0)(k)$ is $k$-reducible, by Proposition \ref{lem-simple-not-stable}. Thus $M_0 \subset M_0'$ for $M_0'$ defined as in one of the following two cases.

Case 4.1. Take $M_0' = [\Phi_V^+ -(S_H \cup \{(268), (3 6 7)\})] \cup \{(1 6 9) \}$; $M_1' = \{(3 4 8), (2 4 9), (1 7 8), (4 5 6), (2 6 8), (3 6 7), (1 5 9)\}$; $f'$ as defined below; and $g = g_{4, 1}$ as defined in Table \ref{table-g}.

\begin{center}
\begin{tabular}{| c | c | c | c | c | c | c | c |}
\hline
$\al$ & (3 4 8) & (2 4 9) & (1 7 8) & (4 5 6) & (2 6 8) & (3 6 7) & (1 5 9)\\
\hline
$f'(\al)$ & {57/16} & 99/16 & 69/32 & 33/32 & 69/16 & 253/32 & 219/32\\
\hline
$\lvert g^{-1}(\al)\rvert$ & 2 & 2 & 2 & 1 & 1 & 7 & 1\\
\hline
\end{tabular}
\end{center}

Case 4.2. Take $M_0' = [\Phi_V^+ - (S_H \cup \{(268), (3 4 9)\})] \cup \{(1 6 9)\}$; $M_1' = \{(3 5 7), (1 7 8), (4 5 6), (2 6 8), (3 4 9), (1 5 9)\}$; $f'$ as defined below; and $g = g_{4, 2}$.

\begin{center}
\begin{tabular}{| c | c | c | c | c | c | c | }
\hline
$\al$ & (3 5 7) & (1 7 8) & (4 5 6) & (2 6 8) & (3 4 9) &  (1 5 9)\\
\hline
$f'(\al)$ & 4 & 25/8 & 2 & 73/8 & 61/8 & 37/8\\
\hline
$\lvert g^{-1}(\al)\rvert$ & 3 & 1 & 2 & 7 & 1 & 1 \\
\hline
\end{tabular}
\end{center}

\end{proof}
Let $N$ be the integer of \S \ref{sec_spreading_out}, and let $p > N$ be a prime. We define $V_p^\text{red} \subset \intV(\bbZ_p)$ to be the set of vectors $v \in \intV(\bbZ_p)$ such that either $p | \Delta_0(v)$, or $p \nmid \Delta_0(v)$ and the image $\overline{v}$ of $v$ in $\intV(\bbF_p)$ is $\intG(\bbF_p)$-conjugate to $\Kostant(\pi(\overline{v}))$. Similarly, we define $V_p^\text{bigstab} \subset \intV(\bbZ_p)$ to be the set of vectors $v \in \intV(\bbZ_p)$ such that either $p | \Delta_0(v)$, or $p \nmid \Delta_0(v)$ and the image $\overline{v}$ of $v$ in $\intV(\bbF_p)$ has non-trivial stabilizer in $\intG(\bbF_p)$.
\begin{proposition}\label{prop_bigstab}
	We have \[ \lim_{Y \to \infty} \prod_{N < p < Y} \int_{v \in V_p^\text{red}} \, dv = 0 \]
	and
	\[ \lim_{Y \to \infty} \prod_{N < p < Y} \int_{v \in V_p^\text{bigstab}} \, dv = 0. \]
\end{proposition}
\begin{proof}
	This can be proved using \cite[Proposition 9.15]{Ser12}. We illustrate the method for $V_p^\text{bigstab}$. The number of points of $\intV(\bbF_p)$ of zero discriminant is $O(p^{83})$. The number of points of $\intV(\bbF_p)$ of non-zero discriminant equals $| B^\text{rs}(\bbF_p) | |\intG(\bbF_p)|$. For a prime $p \equiv 1 \text{ mod }3$, let $C \subset \mathrm{Sp}_4(\bbF_3)$ be the set of elements $\gamma$ which have 1 as an eigenvalue. Then \cite[Proposition 9.15]{Ser12} gives
	\[ \int_{v \in V_p^\text{bigstab}} \, dv =\frac{1}{p^{84}} ( | \{ f \in B^\text{rs}(\bbF_p) \mid \Frob_f \in C \} | \cdot | \intG(\bbF_p) | + O(p^{83}) ) = \frac{|C|}{|\mathrm{Sp}_4(\bbF_3)|} + O(p^{-1/2}). \]
	Since $C \neq \mathrm{Sp}_4(\bbF_3)$, this implies what we need. 
\end{proof}
This concludes the proof of Theorem \ref{thm_point_count}.

\subsection{Counting points with infinitely many congruence conditions}

We now observe that using the results of \cite{Bha18} (see also \cite{Bha15a}), we can prove a strengthened version of Theorem \ref{thm_point_count} in which we impose infinitely many congruence conditions. This will be the analogue of \cite[Theorem 42]{Bha13}. We state the result following \cite{Bha13}. Suppose we are given for each prime $p$ a $\intG(\bbZ_p)$-invariant function $w_p : \intV(\bbZ_p) \to [0, 1]$ satisfying the following conditions:
\begin{itemize}
	\item $w_p$ is locally constant outside the closed subset $\intV(\bbZ_p) - \intV(\bbZ_p)^\text{rs} \subset \intV(\bbZ_p)$.
	\item For all sufficiently large primes $p$, we have $w_p(v) = 1$ for all $v \in \intV(\bbZ_p)$ such that $p^2 \nmid \Delta_0(v)$. 
\end{itemize}
Then we can define a function $w : \intV(\bbZ) \to [0, 1]$ by the formula $w(v) = \prod_p w_p(v)$ if $\Delta_0(v) \neq 0$, and $w(v) = 0$ otherwise. If $X \subset \intV(\bbZ)$ is an $\intG(\bbZ)$-invariant subset, then we extend the definition in Section \ref{section_counting_finitely_many} by again defining
\[ N_w(X, \htvar) = \sum_{\substack{v \in \intG(\bbZ) \backslash X \\ \Ht(v) < \htvar}} \frac{w(v)}{| Z_{\intG}(v)(\bbZ)|}. \]
Our strengthened theorem is then as follows.
\begin{theorem}\label{thm_point_count_with_infinitely_many_conditions}
	We have
	\[ N_w(\intV(\bbZ)^\irr, a) = \frac{|W_0|}{9} \left( \prod_p \int_{v \in \intV(\bbZ_p)} w_p(v) \, dv \right) \vol(\intG(\bbZ) \backslash G(\bbR)) \htvar^{7/10} + o(\htvar^{7/10}). \]
\end{theorem}
Following the proof of \cite[Theorem 24]{Bha15a} and \cite[Proposition 25]{Bha15a}, for primes $p > N$ we define 
\[ \cW_p = \{ v \in \intV(\bbZ_p)^\text{rs} \mid p^2 \text{ divides } \Delta_0(v) \}. \]
Let $\cW_p^1 \subset \cW_p$ denote the set of points $v$ such that at least one of the following three conditions holds: $\pi(v) \text{ mod }p$ has more than 1 repeated root, $\pi(v) \text{ mod }p$ has a triple root, or $v\text{ mod }p$ is not regular. (The proof of Proposition \ref{prop_existence_of_orbits_of_squarefree_discriminant} shows that if $v$ is such an element, then $\Delta_0(v)$ is necessarily divisible by $p^2$.) Let $\cW_p^2 \subset \cW_p$ denote the set of points $v$ such that $\pi(v) \text{ mod }p$ has 1 double root and no other repeated roots, and such that $v \text{ mod }p$ is regular. Then $\cW_p = \cW_p^1 \cup \cW_p^2$. In order to prove Theorem \ref{thm_point_count_with_infinitely_many_conditions} using the method of  \cite[Theorem 24]{Bha15a}, it will suffice to define a map 
\[ \psi : \intG(\bbZ) \backslash (\intV(\bbZ) \cap \cW_p^2) \to \intG(\bbZ) \backslash (\intV(\bbZ) \cap \cW_p^1) \]
with the following properties:
\begin{itemize}
	\item $\Ht \circ \psi = \Ht$.
	\item The fibres of $\psi$ have cardinality at most 3. 
\end{itemize}
We will construct this map as follows: for any $v \in \cW_p^2$, we will define an element $g_{v, p} \in G(\bbQ_p)$ with the following properties: 
\begin{itemize}
	\item $g_{v, p} \cdot v \in \cW_p^1$.
	\item If $k_p \in \intG(\bbZ_p)$, then $g_{k_p \cdot v, p} = k_p g_{v, p} k_p^{-1}$. 
\end{itemize}
The elements $g_{v, p}$ determine a map $\psi_p : \cW_p^2 \to \cW_p^1$, by the formula $v \mapsto g_{v, p} \cdot v$. For each $w \in \cW_p^1$, define $\Pi_p(w) = \{ h_p \in G(\bbQ_p) \mid h_p^{-1} w \in \cW_p^2 \text{ and } h_p = g_{h_p^{-1}w, p} \}$. We note that if $k_p \in \intG(\bbZ_p)$, then $\Pi_p(k_p w) = k_p \Pi_p(w) k_p^{-1}$. We will show that $\Pi_p(w)$ has cardinality at most 3.

%We will show that $\Pi_p(w)$ has the following properties:
%\begin{itemize}
%%	\item If $h_{p} \in \Pi_p(w)$, then $h_{p}^{-1} w \in \cW_p^2$.
%	\item If $k_p \in \intG(\bbZ_p)$, then $\Pi_p(k_p w) = k_p \Pi_p(w) k_p^{-1}$.
%	\item $\Pi_p(w)$ has cardinality at most 3.
%\end{itemize}
Before giving the construction, we explain why it implies the existence of a map $\psi$ with the desired properties. Note that $\intG(\bbZ) \backslash \intG(\bbZ[1/p]) \to \intG(\bbZ_p) \backslash \intG(\bbQ_p)$ is bijective (because $\intG$ has class number 1). It follows that given an element $v \in \cW_p^2 \cap \intV(\bbZ)$, there exists an element $g_v \in ( G(\bbZ_p) \cdot g_{v, p} ) \cap G(\bbQ)$, and $g_v$ is well-defined up to left multiplication by $G(\bbZ)$. We define $\psi(v) = g_v \cdot v$. If $\gamma \in \intG(\bbZ)$ then $g_{\gamma \cdot v} = \gamma g_v \gamma^{-1}$ modulo left multiplication by $\intG(\bbZ)$, so we get a well-defined map $\psi : \intG(\bbZ) \backslash (\intV(\bbZ) \cap \cW_p^2) \to \intG(\bbZ) \backslash (\intV(\bbZ) \cap \cW_p^1)$ that by definition satisfies $\Ht \circ \psi = \Ht$.

To bound the cardinality of the fibres of $\psi$, note that if $w \in \cW_p^1 \cap \intV(\bbZ)$ and $  w = \psi(v)$ (modulo the action of $\intG(\bbZ)$) for some $v \in \cW_p^2 \cap \intV(\bbZ)$, then by definition $w = g_v \cdot v = k_p g_{v, p} \cdot v$, where $k_p \in \intG(\bbZ_p)$, hence $k_p^{-1} w = g_{v, p} \cdot v$, hence $g_{v, p} \in \Pi_p(k_p^{-1} w) = k_p^{-1} \Pi_p(w) k_p$. This shows that $g_v \in \Pi_p(w) k_p$, hence $v = g_v^{-1} w \in (G(\bbZ_p) \Pi_p(w)^{-1} \cap G(\bbQ)) \cdot w$. Again using the fact that $\intG$ has class number 1, we see that $G(\bbZ_p) \Pi_p(w)^{-1} \cap G(\bbQ)$ consists of at most 3 $\intG(\bbZ)$-orbits under left multiplication, hence that the fibre of $\psi$ above the $\intG(\bbZ)$-orbit of $w$ indeed has cardinality at most 3. 

We now construct the element $g_{v, p}$. We will use similar arguments to those of the proof of Proposition \ref{prop_existence_of_orbits_of_squarefree_discriminant}. Let $v \in \cW_p^2$, and let $v_{\bbF_p}$ denote its reduction modulo $p$. Let $v_{\bbF_p} = y_s + y_n$ be its Jordan decomposition. There is a decomposition $\intfrh_{\bbZ_p} = \intfrh_{0, \bbZ_p} \oplus \intfrh_{1, \bbZ_p}$, where $\ad(v)$ acts as a topologically nilpotent operator on $\intfrh_{0, \bbZ_p}$ and acts invertibly on $\intfrh_{1, \bbZ_p}$. Moreover, there is a unique closed subgroup $\underline{L} \subset \intH_{\bbZ_p}$ that is smooth over $\bbZ_p$ with connected fibres and with Lie algebra $\intfrh_{0, \bbZ_p}$ (the argument is the same as in the proof of Proposition \ref{prop_existence_of_orbits_of_squarefree_discriminant}). By assumption, $y_n$ is a regular nilpotent element in $\intfrh_{0, \bbF_p} = \frz_\intfrh(y_s)$.

There is an isomorphism $\intfrh_{0, \bbZ_p}^\text{der} \cong \frs\frl_{3, \bbZ_p}$ that intertwines $\theta|_{\intfrh_{0, \bbZ_p}^\text{der}}$ with $\zeta \mapsto \Ad(\diag(1, \zeta, \zeta^2))$ and sends $y_n$ to the element
\[ \left( \begin{array}{ccc} 0 & 1 & 0 \\ 0&0&1\\0&0&0 \end{array}\right) \]
of $\frs\frl_{3, \bbF_p}$. (Indeed, there is a unique such isomorphism modulo $p$, which then lifts by Hensel's lemma to an isomorphism over $\bbZ_p$.) Similarly, there is a map $\varphi_v : \SL_{3, \bbZ_p} \to \underline{L}$ that intertwines $\Ad(\diag(1, \zeta, \zeta^2))$ with $\theta_{\underline{L}} := \theta|_{\underline{L}}$ and that is compatible with the above isomorphism on Lie algebras. The map $\varphi_v$ is uniquely determined up to conjugation by diagonal matrices in $\PGL_3(\bbZ_p)$; the element $g_{v, p}  = \varphi_v(\diag(p, 1, p^{-1})) \in \underline{L}(\bbQ_p)$ is therefore independent of any choices.

To see that this $g_{v, p}$ has the desired properties, let $v'$ denote the projection of $v$ to $\frh_{0, \bbZ_p}^\text{der}$, and note that the image of $v'$ in $\frs\frl_{3, \bbZ_p}$ has the form 
\[ v' = \left( \begin{array}{ccc} 0 & a & 0 \\ 0 & 0 &  b \\ c & 0 & 0 \end{array}\right), \]
where $a \equiv b \equiv 1 \text{ mod }p$ and $p^2 | c$ (because of our assumption that $p^2$ divides $\Delta_0(v)$). Thus we have
\[ g_{v, p} \cdot v' = \left( \begin{array}{ccc} 0 & p a & 0 \\ 0 & 0 &  p b \\ c / p^2 & 0 & 0 \end{array}\right). \]
The reduction modulo $p$ of $g_{v, p} \cdot v$ is no longer regular, showing that $g_{v, p} \cdot v \in \cW^1_p$. This defines the map $\psi_p$.

We now describe the set $\Pi_p(w)$ for $w \in \cW_p^1$. Let $w \in \cW_p^1$, let $w_{\bbF_p}$ be its reduction modulo $p$, and let $w_{\bbF_p} = z_s + z_n$ be the Jordan decomposition. As before, we have a decomposition $\intfrh_{\bbZ_p} = \intfrh_{0, \bbZ_p} \oplus \intfrh_{1, \bbZ_p}$ where $\ad(w)$ acts as a topologically nilpotent operator on $\intfrh_{0, \bbZ_p}$ and acts invertibly on $\intfrh_{1, \bbZ_p}$, and $\intfrh_{0, \bbZ_p}$ is the Lie algebra of a Levi subgroup $\underline{L} \subset \intH_{\bbZ_p}$. 

Observe that if $w = g_{v, p} \cdot v$ for some $v \in \cW_p^2$, then the derived subalgebra of $\frz_\intfrh(z_s)$ is isomorphic to $\frs\frl_{3, \bbF_p}$ (i.e.\ it is split) and its grading is conjugate to the $\bbZ / 3 \bbZ$-grading given by the formula  $\zeta \mapsto \Ad(\diag(1, \zeta, \zeta^2))$ (in fact, $\frz_\intfrh(z_s)$ coincides with the derived subalgebra of $\frz_{\intfrh}(y_s)$ in the above discussion). We can therefore assume without loss of generality that $\frz_\intfrh(z_s)$ is split and has a grading of this form (otherwise $\Pi_p(w)$ is empty).

If we fix an isomorphism $\varphi: \frz_\intfrh(z_s)^\text{der} \to \frs\frl_{3, \bbF_p}$ that identifies $\theta|_{\frz_\intfrh(z_s)^\text{der}}$ with the $\bbZ / 3 \bbZ$-grading $\zeta \mapsto \Ad(\diag(1, \zeta, \zeta^2))$, then there is a compatible morphism $\SL_{3, \bbZ_p} \to \underline{L}$, uniquely determined up to conjugation by diagonal elements of $\PGL_3(\bbZ_p)$. Let $h_{p, \varphi} \in H(\bbQ_p)$ be the image of $\diag(p^{-1}, 1, p) \in \SL_3(\bbQ_p)$. 

There are three possible choices of isomorphism between $\frz_\intfrh(z_s)^\text{der}$ and $\frs\frl_{3, \bbF_p}$, up to $\SL_{3, \bbF_p}^\theta$-conjugacy, so there are three elements of the form $h_{p, \varphi} \in \underline{L}(\bbQ_p)$. The set $\Pi_p(w)$ is contained in the set of elements $h_{p, \varphi}$ constructed this way. Thus $\Pi_p(w)$ has cardinality at most 3. 
%The other claimed properties of the set $\Pi_p(w)$ follow from the definition. 
We have therefore completed the proof of Theorem \ref{thm_point_count_with_infinitely_many_conditions}.

\section{The main theorem}\label{sec_main_theorem}

We can now prove the theorems stated in the introduction. We begin by re-establishing notation. Thus $\Equations$ denotes the set of polynomials $f(x) = x^5 + c_{12} x^3 + c_{18} x^2 + c_{24} x + c_{30} \in \bbZ[x]$ of non-zero discriminant, and $\Equations_\text{min} \subset \Equations$ denotes the set of polynomials $f(x)$ not of the form $n \cdot g = n^{10} g(x / n^2) \in \Equations$ for any $g \in \Equations$ and integer $n \geq 2$. As in the previous section, if $f \in \Equations$ is of the form above, then we define the height of $f$ by the formula
\[ \Ht(f) = \sup_i | c_i |^{120 / i}. \]
Thus for any $\htvar > 0$, the set $\{ f \in \Equations \mid \Ht(f) < \htvar \}$ is finite. To any $f \in \Equations$ we associate the smooth, projective curve $\Curve_f$ and its Jacobian $\Jacobian_f$ over $\bbQ$.
\begin{theorem}\label{thm_main_theorem_first_version}
We have
\[ \lim_{\htvar \to \infty} \frac{ \sum_{f \in \Equations_\text{min}, \Ht(f) < \htvar } | \Sel_3(\Jacobian_f) | }{ | \{ f \in \Equations_\text{min} \mid \Ht(f) < \htvar \} |} = 4. \]
\end{theorem}
We first prove a `local' result. Let $\intG, \intV$ be the group and representation defined in \S \ref{sec_stable_grading_of_E8}, and let $N \geq 1$ be the integer of \S \ref{sec_spreading_out}; thus our main constructions make sense over $\bbZ[1/N]$. If $p$ is a prime, then we write $\Equations_p$ for the set of polynomials $f(x) = x^5 + c_{12} x^3 + c_{18} x^2 + c_{24} x + c_{30} \in \bbZ_p[x]$ of non-zero discriminant, and $\Equations_{p, \text{min}} \subset \Equations_p$ for the set of polynomials not of the form $p^{10} g(x / p^{2})$ for any polynomial $g(x) \in \Equations_p$.
\begin{proposition}\label{prop_local_version_of_main_theorem}
Let $f_0(x) \in \Equations_\text{min}$. Then  we can find for each prime $p \leq N$ an open compact neighbourhood $W_p$ of $f_0(x)$ in $\Equations_p$ such that the following condition holds. Let $\Equations_W = \Equations \cap (\prod_{p \leq N} W_p)$, and let $\Equations_{W, \text{min}} = \Equations_W \cap \Equations_\text{min}$. Then we have
\[  \lim_{\htvar \to \infty} \frac{ \sum_{f \in \Equations_{W, \text{min}}, \Ht(f) < \htvar} | \Sel_3(\Jacobian_f) | }{ | \{ f \in \Equations_{W,\text{min}} \mid \Ht(f) < \htvar \} |} = 4. \]
\end{proposition}
(The intersection $\Equations \cap (\prod_{p \leq N} W_p)$ is taken in $\prod_{p \leq N} \Equations_p$, where we view $\Equations$ as a subset via the diagonal embedding.)
\begin{proof}
We choose the sets $W_p$ for $p \leq N$, together with integers $n_p \geq 0$, so that the conclusion of Corollary \ref{cor_existence_of_integrable_representatives_for_selmer_orbits} holds. If $p > N$, let $W_p = \Equations_{p, \text{min}}$ and $n_p = 0$. Let $M = \prod_p p^{n_p}$. After possibly shrinking the $W_p$ with $p \leq N$, we can assume that the $W_p$ with $p \leq N$ satisfy $W_p \subset \Equations_{p, \text{min}}$.

For $v \in \intV(\bbZ)$ with $\pi(v) = f$, define $w(v) \in \bbQ_{\geq 0}$ by the following formula:
\[ w(v) = \left\{ \begin{array}{cc} \left( \sum_{v' \in \intG(\bbZ) \backslash (\intG(\bbQ) \cdot v \cap \intV(\bbZ))} \frac{| Z_{\intG}(v')(\bbQ) |}{| Z_{\intG}(v')(\bbZ) |}\right)^{-1} & \text{if }f \in M \cdot \Equations_{W, \text{min}}\text{ and }\intG(\bbQ) \cdot v \in \eta_f(\Sel_3(\Jacobian_f))\\
0 & \text{otherwise.} \end{array}\right. \]
We have
\[ \sum_{\substack{f \in \Equations_{W, \text{min}} \\ \Ht(f) < \htvar}} \frac{| \Sel_3(\Jacobian_f) | - 1}{| \Jacobian_f[3](\bbQ) |} = \sum_{\substack{ v \in \intG(\bbZ) \backslash \intV(\bbZ)^\text{irr} \\ \Ht(v) < M^{120} \htvar}} w(v). \]
For $v \in \intV(\bbZ_p)$ with $\pi(v) = f$, define $w_p(v) \in \bbQ_{\geq 0}$ by the following formula:
\[ w_p(v) = \left\{ \begin{array}{cc} \left(\sum_{v' \in \intG(\bbZ_p) \backslash (\intG(\bbQ_p) \cdot v \cap \intV(\bbZ_p)} \frac{| Z_{\intG}(v')(\bbQ_p) |}{| Z_{\intG}(v')(\bbZ_p) |}\right)^{-1} & \text{if }f \in p^{n_p} W_p \text{ and }{\intG}(\bbQ_p) \cdot v \in \eta_f(\Jacobian_f(\bbQ_p))\\
0 & \text{otherwise.} \end{array}\right. \]
Then for any $v \in \intV(\bbZ)$, we have $w(v) = \prod_p w_p(v)$, and the function $w$ satisfies the conditions described before the statement of Theorem \ref{thm_point_count_with_infinitely_many_conditions}.

Let $W_0 \in \bbQ^\times$ be the constant of Proposition \ref{prop_integration_in_fibres}. That proposition implies that for any prime $p$, we have the formula
\[ \int_{v \in \intV(\bbZ_p)} w_p(v) \, dv = | W_0 / 9 |_p p^{- \dim_\bbQ V \cdot n_p} \vol(W_p) \vol(\intG(\bbZ_p)), \]
where we have used the equality $| \Jacobian_f(\bbQ_p) / 3 \Jacobian_f(\bbQ_p) | = |  1 / 9 |_p | Z_{\intG}(\Kostant(f))(\bbQ_p) |$ for any $f \in \Equations_p$. By Theorem \ref{thm_point_count_with_infinitely_many_conditions} and Proposition \ref{prop_tamagawa_number_of_G}, we therefore have
\[ \begin{split} \lim_{\htvar \to \infty} \sum_{\substack{f \in \Equations_{W, \text{min}} \\ \Ht(f) < \htvar}} \frac{| \Sel_3(\Jacobian_f) | - 1}{\htvar^{7/10}| \Jacobian_f[3](\bbQ) |}  & = \frac{M^{120}}{9} | W_0 |_\infty \vol(\intG(\bbZ) \backslash \intG(\bbR)) \prod_p | W_0 / 9 |_p p^{- \dim V \cdot n_p} \vol(W_p) \vol(\intG(\bbZ_p)) \\ & = 3 \prod_p \vol(W_p).\end{split} \] 
On the other hand, we have
\[ \lim_{ \htvar \to \infty} \frac{| \{ f \in \Equations_{W, \text{min}} \mid \Ht(f) < \htvar \} |}{\htvar^{7/10}} = \prod_p \vol(W_p). \]
At this point we have proved that
\[ \lim_{\htvar \to \infty} \left( \sum_{\substack{f \in \Equations_{W, \text{min}} \\ \Ht(f) < \htvar}} \frac{| \Sel_3(\Jacobian_f) | - 1}{ |\Jacobian_f[3](\bbQ) |} \right) \left( | \{ f \in \Equations_{W, \text{min}} \mid \Ht(f) < \htvar \} | \right)^{-1}    = 3. \]
It remains to eliminate the appearance of the term $| \Jacobian_f[3](\bbQ) |$. This can be done by combining Proposition \ref{prop_bigstab} and Theorem \ref{thm_point_count}.
\end{proof}
To deduce Theorem \ref{thm_main_theorem_first_version} from Proposition \ref{prop_local_version_of_main_theorem}, we choose for each $i \geq 1$ sets $W_{p, i}$ ($p \leq N$) such that if $W_i = \Equations \cap (\prod_{p \leq N} W_{p, i})$, then $W_i$ satisfies the conclusion of Proposition \ref{prop_local_version_of_main_theorem} and we have a countable partition $\Equations_{\text{min}} = \Equations_{W_1, \text{min}} \sqcup \Equations_{W_2, \text{min}} \sqcup \Equations_{W_3, \text{min}} \sqcup \dots$.  We will show that for all $\epsilon > 0$, there exists $k \geq 1$ such that
\[ \limsup_{\htvar \to \infty} \frac{ \sum_{f \in \sqcup_{i \geq k} \Equations_{W_i, \text{min}}, \Ht(f) < \htvar} | \Sel_3(\Jacobian_f) | - 1}{ | \{ f \in \Equations_{\text{min}} \mid \Ht(f) < \htvar \} |} < \epsilon. \]
Combined with Proposition \ref{prop_local_version_of_main_theorem}, which applies to each set $\Equations_{W_i, \text{min}}$ taken individually, this will imply the desired result. For each $f \in \Equations$, let 
\[ \Sel_3(\Jacobian_f)^r = \ker( \Sel_3(\Jacobian_f) \to \prod_{p \leq N} \Jacobian_f(\bbQ_p) / 3 \Jacobian_f(\bbQ_p)). \]
Then there exists an integer $N_0 \geq 1$, depending only on $N$, such that for any $f \in \Equations$, $| \Sel_3(\Jacobian_f) | \leq N_0 | \Sel_3(\Jacobian_f)^r |$. It will therefore suffice to show that for all $\epsilon > 0$, there exists $k \geq 1$ such that
\[ \limsup_{\htvar \to \infty} \frac{ \sum_{f \in \sqcup_{i \geq k} \Equations_{W_i, \text{min}}, \Ht(f) < \htvar} | \Sel_3(\Jacobian_f)^r | - 1}{ | \{ f \in \Equations_{\text{min}} \mid \Ht(f) < \htvar \} |} < \epsilon. \]
Fix $k \geq 1$ and let $\Equations_k = \sqcup_{i \geq k} \Equations_{W_i, \text{min}}$. We now use that for any $f \in \intB(\bbZ)$, $\Kostant(N \cdot f) \in \intV(\bbZ)$ (see \S \ref{sec_spreading_out}).  It follows that we have
\[ \sum_{f \in \Equations_k, \Ht(f) < \htvar} \frac{ | \Sel_3(\Jacobian_f)^r | - 1}{|\Jacobian_f[3](\bbQ)|} = \sum_{\substack{v \in \intG(\bbZ) \backslash \intV(\bbZ)^\text{irr} \\ \Ht(v) < N^{120}\htvar}} w^r(v), \]
where the weight $w^r(v)$ is defined in the formula
\[ w^r(v) = \left\{ \begin{array}{cc} \left( \sum_{v' \in \intG(\bbZ) \backslash (\intG(\bbQ) \cdot v \cap \intV(\bbZ))} \frac{| Z_{\intG}(v')(\bbQ) |}{| Z_{\intG}(v')(\bbZ) |}\right)^{-1} & \text{if }f \in N \cdot \Equations_{k}\text{ and }G(\bbQ) \cdot v \in \eta_f(\Sel_3(\Jacobian_f)^r)\\
0 & \text{otherwise.} \end{array}\right. \]
Running through the same argument as in the proof of Proposition \ref{prop_local_version_of_main_theorem}, we get
\[ \limsup_{\htvar \to \infty} \frac{ \sum_{f \in \Equations_k, \Ht(f) < \htvar} | \Sel_3(\Jacobian_f)^r | - 1}{\htvar^{7/10}|\Jacobian_f[3](\bbQ)|} \leq 3 \prod_{p \leq N} \vol(\sqcup_{i \geq k} W_{p, i}), \]
which becomes arbitrarily small as $k \to \infty$. This completes the proof of Theorem \ref{thm_main_theorem_first_version}.
\begin{remark}\label{rmk_generalized_version_of_main_theorem}
	Using Theorem \ref{thm_point_count_with_infinitely_many_conditions} and \cite[Theorem 44]{Bha13}, one can prove the analogue of Theorem \ref{thm_main_theorem_first_version} for any `large' subset of $\Equations_\text{min}$, where `large' has the same meaning as in \cite[\S 11]{Bha13}; this includes in particular any subset defined by finitely many congruence conditions on the cofficients of $f(x) = x^5 + c_{12} x^3 + c_{18} x^2 + c_{24} x + c_{30}$.
\end{remark}
Our final result (Theorem \ref{thm_intro_poonen_stoll} of the introduction) follows readily from the above techniques and from the work of Poonen--Stoll:
\begin{theorem}\label{thm_poonen_stoll}
We have
\[ \liminf_{\htvar \to \infty} \frac{ | \{ f \in \Equations_\text{min} \mid \Ht(f) < \htvar, | \Curve_f(\bbQ) | = 1 \} |}{| \{ f \in \Equations_\text{min} \mid \Ht(f) < \htvar \} |} > 0. \]
\end{theorem}
\begin{proof}
	According to \cite[Remark 10.5]{Poo14}, this follows if one can establish property $\operatorname{Eq}_2(3)$ of \emph{op. cit.}, which asserts that after fixing a `trivializing congruence class' $U_3 \subset \Equations_{3, \text{min}}$ in which the groups $\Jacobian_f(\bbQ_3) / 3 \Jacobian_f(\bbQ_3) = F$ are independent of $f \in U_3$, the images $x|_3$ of 3-Selmer elements $x \in \Sel_3(\Jacobian_f)$ in the local groups $\Jacobian_f(\bbQ_3) / 3 \Jacobian_f(\bbQ_3) = F$ are equidistributed for $f \in \Equations_\text{min} \cap U_3$.	This can be proved by a small modification of the proof of Theorem \ref{thm_main_theorem_first_version}, analogous to the proof of \cite[Theorem 47]{Bha13}. We omit the details.
\end{proof}
\bibliographystyle{alpha}
\bibliography{selmer}
\end{document}